\documentclass{amsart}
\usepackage[pdfusetitle,colorlinks,citecolor=blue,urlcolor=black,linkcolor=black]{hyperref}
\usepackage{amsfonts,amssymb,amsmath,amsthm,stmaryrd}
\usepackage{enumitem}

\newtheorem*{maintheorem}{Main Theorem}
\newtheorem{theorem}{Theorem}[section]
\newtheorem{claim}{Claim}[theorem]
\newtheorem{lemma}[theorem]{Lemma}
\newtheorem{proposition}[theorem]{Proposition}
\newtheorem{corollary}[theorem]{Corollary}
\newtheorem{fact}[theorem]{Fact}

\theoremstyle{definition}
\newtheorem{definition}[theorem]{Definition}

\theoremstyle{remark}
\newtheorem{remark}[theorem]{Remark}
\newtheorem{convention}[theorem]{Convention}

\DeclareMathOperator{\id}{id}
\DeclareMathOperator{\cf}{cf}
\DeclareMathOperator{\chr}{Chr}
\DeclareMathOperator{\dom}{dom}

\DeclareMathOperator{\Succ}{Succ}
\DeclareMathOperator{\im}{Im}
\newcommand{\stem}{{\Omega}}

\DeclareMathOperator{\on}{On}
\DeclareMathOperator{\tp}{top}
\DeclareMathOperator{\Lev}{Lev}
\DeclareMathOperator{\crit}{crit}
\DeclareMathOperator{\Proj}{Tr}
\DeclareMathOperator{\hgt}{ht}
\DeclareMathOperator{\Aut}{Aut}
\DeclareMathOperator{\reg}{Reg}
\DeclareMathOperator{\Ult}{Ult}
\DeclareMathOperator{\otp}{otp}
\DeclareMathOperator{\col}{Col}
\DeclareMathOperator{\rlm}{Realm}
\DeclareMathOperator{\Add}{Add}
\DeclareMathOperator{\acc}{acc}
\DeclareMathOperator{\nacc}{nacc}
\DeclareMathOperator{\drop}{drop}
\DeclareMathOperator{\rank}{rank}
\newcommand\s{\subseteq}
\newcommand\br{\blacktriangleright}
\newcommand\diagonal{\bigtriangleup}
\renewcommand\mid{\mathrel{|}\allowbreak}
\def\Mid{\mathrel{\bigm|}}

\def\sd{\framebox[3.0mm][l]{$\diamondsuit$}\hspace{0.5mm}{}}

\newcommand*\axiomfont[1]{\textsf{\textup{#1}}}
\newcommand\zfc{\axiomfont{ZFC}}
\newcommand\gch{\axiomfont{GCH}}
\newcommand\sch{\axiomfont{SCH}}

\subjclass[2010]{Primary 03E35; Secondary 05C63, 03E55}

\usepackage[export]{adjustbox}
\newcommand{\ssim}{\mathrel{
\ooalign{\raise0.2ex\hbox{\ensuremath{\subset}}
\cr\hidewidth\raise-0.8ex\hbox{\scalebox{0.9}{\ensuremath{\sim}}}\hidewidth\cr}}}

\author{Sittinon Jirattikansakul}
\address{Department of Mathematics and Computer Science, Faculty of Science,
Chulalongkorn University, Bangkok 10330 Thailand}
\email{jir.sittinon@gmail.com}

\author{Inbar Oren}
\address{Einstein Institute of Mathematics, The Hebrew University of Jerusalem, Israel}
\urladdr{https://inbarorendoesmaths.website}
\email{inbar.oren2@mail.huji.ac.il}

\author{Assaf Rinot}
\address{Department of Mathematics, Bar-Ilan University, Ramat-Gan 5290002, Israel.}
\urladdr{http://www.assafrinot.com}

\title{A model for global compactness}

\date{Preprint as of December 17, 2025. For updates, visit \textsf{http://p.assafrinot.com/67}.}

\begin{document}
\begin{abstract} In a classical paper by
Ben-David and Magidor, a model of set theory was exhibited in which $\aleph_{\omega+1}$ carries a uniform ultrafilter that is $\theta$-indecomposable for every uncountable cardinal $\theta<\aleph_\omega$.
In this paper, we give a global version of this result, as follows:

Assuming the consistency of a supercompact cardinal, we produce a model of set theory in which
for every singular cardinal $\lambda$, there exists a uniform ultrafilter on $\lambda^+$
that is $\theta$-indecomposable for every cardinal $\theta$ such that $\cf(\lambda)<\theta<\lambda$.
In our model, many instances of compactness for chromatic numbers hold, from which we infer that Hajnal's gap-$1$ counterexample to Hedetniemi's conjecture
is best possible on the grounds of $\zfc$.
\end{abstract}
\maketitle
\section{Introduction}

An ultrafilter $U$ over an uncountable cardinal $\kappa$ is \emph{uniform} iff the map $X\mapsto |X|$ is constant over $U$.
It is \emph{$\theta$-indecomposable} (for an infinite cardinal $\theta<\kappa$) iff for every function $f:\kappa\rightarrow\theta$, there exists a set $X\in U$ such that $|f[X]|<\theta$.
Note that for $\theta$ a regular cardinal, $\theta$-indecomposability coincides with the requirement that $U$ be closed under intersections of decreasing sequences of length exactly $\theta$.

The existence of uniform indecomposable ultrafilters yields various forms of compactness.
As a simple example, we mention that for every pair $\theta<\kappa$ of infinite regular cardinals,
the existence of a $\theta$-indecomposable uniform ultrafilter on $\kappa$ implies that every stationary subset of $E^\kappa_\theta:=\{\alpha<\kappa\mid\cf(\alpha)=\theta\}$ reflects.
Richer combinatorial applications of indecomposable ultrafilters may be found in \cite{MR1117029,MR3051629,MR3009737,Sh:1160,paper64}.

\smallskip

Sheard \cite{Sheard83} proved that a non-measurable inaccessible $\kappa$ may consistently carry a uniform ultrafilter that is $\theta$-indecomposable for all $\theta<\kappa$ except for $\theta=\omega$.
This was recently extended \cite{paper69} to a non-measurable $\kappa$ that is moreover weakly compact.
As for successor cardinals, assuming the consistency of a supercompact cardinal,
Ben-David and Magidor \cite{BenMag86} constructed a model of set theory in which $\aleph_{\omega+1}$ carries a uniform ultrafilter that is $\theta$-indecomposable for every uncountable cardinal $\theta<\aleph_\omega$.
This gave the first model in which the square principle $\square_{\aleph_\omega}$ fails, but its weaker sibling $\square^*_{\aleph_\omega}$ holds.
A global result in this spirit was obtained by Apter and Cummings in \cite{MR1740483}.
A variation in another direction, yielding a uniform ultrafilter over $\aleph_{\omega_1+1}$
that is $\theta$-indecomposable for all $\theta\in(\aleph_1,\aleph_{\omega_1})$ appeared in a note by Unger \cite{UngerNote}. This variation also secures homogeneity of the notion of forcing involved.

In this paper, a global version of these results is obtained. Our main theorem takes care of all successors of singulars simultaneously.
Further features of the model are explored, as follows.
\begin{maintheorem} Assuming the consistency of a supercompact cardinal, there exists a model of $\zfc+\gch$ in which,
for every singular cardinal $\lambda$:
\begin{enumerate}[label=\textup{(\arabic*)}]
\item there exists a uniform ultrafilter on $\lambda^+$ that is $\theta$-indecomposable for every cardinal $\theta$ with $\cf(\lambda)<\theta<\lambda$;
\item every family of fewer than $\cf(\lambda)$ many stationary subsets of $E^{\lambda^+}_{>\cf(\lambda)}$ reflect simultaneously;
\item $\square_{\lambda,{<\cf(\lambda)}}$ fails and $\square_{\lambda,\cf(\lambda)}$ holds;
\item for every graph $\mathcal G$ of size $\lambda^+$,
for every cardinal $\theta\in E^\lambda_{>\cf(\lambda)}$,
if every subgraph $\bar{\mathcal G}$ of $\mathcal G$ of size less than $\lambda$ satisfies $\chr(\bar{\mathcal G})\le\theta$,
then $\chr(\mathcal G)\le\theta$;
\item for every two graphs $\mathcal G,\mathcal H$, each of size at most $\lambda^+$,
if $\min\{\chr(\mathcal G),\chr(\mathcal H)\}\ge\lambda$, then $\chr(\mathcal G\times\mathcal H)\ge\lambda$;
\item if $\cf(\lambda)>\omega$,
then in the forcing extension to add $\lambda$-many Cohen reals, $\mathfrak{u}_{\kappa}<2^{\kappa}$
holds for every cardinal $\kappa<\lambda$ that is the successor of a singular of cofinality $\neq\cf(\lambda)$.
\end{enumerate}
\end{maintheorem}

Compared to the proof of Ben-David and Magidor \cite{BenMag86}, our proof replaces Magidor forcing by Radin forcing
and weak homogeneity is secured here by a variant of guiding generics that we call \emph{gurus}.
Let us now describe the breakdown of this paper.
\begin{itemize}
\item In Section~\ref{conventions}, we provide a few preliminaries, mostly around indecomposable ultrafilters and their combinatorial consequences.
\item In Section~\ref{mastersequenceandcoherentsequence}, we build a coherent sequence $\vec{U}$ of supercompact measures, together with a guru sequence $\vec{\mathbf t}$, while ensuring that the two cohere in a certain way.
\item In Section~\ref{radinforcing}, we define \emph{Radin forcing with gurus} $\mathbb R_{\vec U,\vec{\mathbf t}}$ using the objects obtained in Section~\ref{mastersequenceandcoherentsequence}.
We verify basic properties of this forcing as well as verify that it has the strong Prikry property.
\item In Section~\ref{cardinalstructureradin}, we determine the cardinal structure in the forcing extension by $\mathbb R_{\vec U,\vec{\mathbf t}}$.
\item In Section~\ref{pradinforcing}, we study a natural projection $\Pi$ of our Radin forcing $\mathbb R_{\vec U,\vec{\mathbf t}}$ to Prikry forcing $\mathbb P_{\vec U,\vec{\mathbf t}}$.
\item In Section~\ref{intermediateforcing}, we study intermediate models between the extension by $\mathbb P_{\vec U,\vec{\mathbf t}}$ and the extension by $\mathbb R_{\vec U,\vec{\mathbf t}}$.
\item In Section~\ref{weakhomogeneity}, we establish the aforementioned weak homogeneity of our forcing.
\item In Section~\ref{finalmodel}, we obtain the final model witnessing the conclusion of our Main Theorem.
\item In Section~\ref{concluding}, we outline a variation of the final model in which the $\sch$ fails everywhere.
\end{itemize}

\section{Indecomposable ultrafilters and additional preliminaries}\label{conventions}
The class of infinite regular cardinals is denoted by $\reg$, and we write $\reg(\lambda)$ for ${\reg}\cap\lambda$.
For every set of ordinals $x$, we let $\pi_x:\otp(x)\to x$ denote its inverse collapsing map.
By a \emph{supercompact measure} on $\mathcal{P}_\kappa(\lambda)$, we mean a normal fine $\kappa$-complete ultrafilter on $\mathcal{P}_\kappa(\lambda)$.

Given a filter $D$ over a cardinal $\kappa$, we write $D^*$ for its dual ideal $\{ \kappa\setminus X\mid X\in D\}$,
and $D^+$ for the collection $\mathcal P(\kappa)\setminus D^*$ of its positive sets.
A \emph{base} for $D$ is a subset $\mathcal {B}\s D$ such that for every $A \in D$, there is a $B\in\mathcal B$ such that $|B \setminus A|<\kappa$;
the least size of a base is denoted by $\chi(D):=\min\{|\mathcal{B}| \mid \mathcal{B}\text{ is a base for }D\}$.

\begin{definition}[Keisler, Prikry \cite{prikrythesis}]\label{def22}
Let $D$ be a filter over a cardinal $\kappa$, and let $\theta$ be a infinite cardinal.
$D$ is said to be \emph{$\theta$-indecomposable} iff for every $B\in D^+$ and every function $f:B\rightarrow \theta$,
there exists a $T \in [\theta]^{<\theta}$ such that $f^{-1}[T]$ is in $D^+$.
\end{definition}

We say that $D$ is \emph{$[\chi,\nu)$-indecomposable} iff it is $\theta$-indecomposable for every cardinal $\theta$ with $\chi\le\theta<\nu$.
Note that $D$ is \emph{$\nu$-complete} iff it is $[\omega,\nu)$-indecomposable.

\begin{fact}[Kunen-Prikry, {\cite[Theorem~0.2(b)]{MR0302441}}]\label{limitation1}
Suppose that $D$ is a uniform ultrafilter over a successor $\lambda^+$ of a singular cardinal $\lambda$.
If $\{\theta\in\reg(\lambda)\mid D\text{ is }\theta\text{-indecomposable}\}$
is cofinal in $\lambda$, then $D$ is not $\cf(\lambda)$-indecomposable.
\end{fact}

\begin{fact}[Usuba, {\cite[Proposition~7.3]{usuba2025}}]\label{limitation2}
Suppose that $D$ is a uniform ultrafilter over a successor $\lambda^+$ of a singular cardinal $\lambda$.
If $D$ is $\lambda$-indecomposable, then it is $\cf(\lambda)$-indecomposable.
\end{fact}

\begin{corollary}\label{atthesingulars} Suppose $\lambda$ is a singular cardinal and $\chi<\lambda$ is another cardinal.
If there exists a uniform $(\chi,\lambda)$-indecomposable ultrafilter over $\lambda^+$,
then there also exists a uniform $(\chi,\lambda)$-indecomposable ultrafilter over $\lambda$.
\end{corollary}
\begin{proof} Suppose $D$ is a uniform $(\chi,\lambda)$-indecomposable ultrafilter over $\lambda^+$.
By Fact~\ref{limitation1}, $D$ is not $\cf(\lambda)$-indecomposable. By Fact~\ref{limitation2}, then, $D$ is not $\lambda$-indecomposable either.
As $D$ is an ultrafilter, it altogether follows that we may fix a function $f:\lambda^+\rightarrow \lambda$ such that for every $X\in[\lambda]^{<\lambda}$, $f^{-1}[X]$ is not in $D$.
That is, $E:=\{ X\s\lambda\mid f^{-1}[X]\in D\}$ is a uniform filter over $\lambda$.
Clearly, it is moreover an ultrafilter.
\begin{claim} Let $\theta$ such that $D$ is $\theta$-indecomposable. Then $E$ is $\theta$-indecomposable.
\end{claim}
\begin{proof} Given a function $g:\lambda\rightarrow\theta$, consider $h:=g\circ f$ which is a function from $\lambda^+$ to $\theta$.
Find $T\in[\theta]^{<\theta}$ such that $h^{-1}[T]$ is in $D$.
In particular, $X:=f[h^{-1}[T]]$ is in $E$
and clearly $g^{-1}[T]$ covers $X$ (since $g[X]=g[f[h^{-1}[T]]]=T$).
\end{proof}
Thus, the uniform ultrafilter $E$ is indeed $(\chi,\lambda)$-indecomposable.
\end{proof}

\begin{fact}[Silver, {\cite[Lemma~2]{MR0360276}}]\label{finest}
Suppose $\chi=\cf(\chi)<2^\chi<\nu\le\kappa$ are infinite cardinals,
and $U$ is a $\chi$-incomplete $[\chi,\nu)$-indecomposable uniform ultrafilter over $\kappa$.
Then there exist a $\vartheta<\chi$ and a map $\varphi: \kappa\rightarrow \vartheta$ such that the following two hold:
\begin{itemize}
\item for every $\tau<\vartheta$, $\varphi^{-1}[\tau]\notin U$, and
\item for every function $f: \kappa\rightarrow \mu$ with $\mu<\nu$, there exists a function $g: \vartheta\rightarrow\mu$ such that $f = g\circ \varphi \pmod {U}$.
\end{itemize}
\end{fact}

\begin{definition}[Keisler] A filter $D$ is \emph{$(\lambda,\kappa)$-regular} iff there is a sequence $\langle A_\beta\mid\beta<\kappa\rangle$ of sets in $D$ such that
$\bigcap_{\beta\in B}A_\beta=\emptyset$ for every $B\in[\kappa]^\lambda$.
\end{definition}

Note that any uniform ultrafilter over a regular uncountable $\kappa$ is $(\kappa,\kappa)$-regular.

\begin{fact}[Kanamori, {\cite[Corollary~2.4]{MR480041}}]\label{kanreg} For every singular cardinal $\lambda$, every uniform ultrafilter over $\lambda^+$ is $(\lambda,\lambda^+)$-regular.
\end{fact}

Next, we recall a Menas-type theorem for indecomposable ultrafilters on singular cardinals.
To this end, we'd need the following definition.
\begin{definition}\label{limitultrafilter}
Suppose $\kappa$ is a regular cardinal, $\langle \lambda_i \mid i<\kappa \rangle$ is a non-decreasing sequence of cardinals bigger than $\kappa$,
where each $\lambda_i$ carries a uniform ultrafilter $U_i$.
Suppose $E$ is a uniform ultrafilter on $\kappa$.
Define $\lambda:=\sup_{i<\kappa}\lambda_i$ and $$E\text{-}\lim_{i<\kappa}U_i:=\{X \subseteq \lambda \mid \{i<\kappa \mid X \cap \lambda_i \in U_i\}\in E\}.$$
\end{definition}

\begin{proposition}\label{menas}
Continuing with the notation of Definition~\ref{limitultrafilter}, we let $U:=E\text{-}\lim_{i<\kappa}U_i$.
Then:
\begin{enumerate}[label=\textup{(\arabic*)}]
\item $U$ is a uniform ultrafilter;
\item For every cardinal $\theta\in E^\lambda_{>\kappa}$ such that $\{i<\kappa\mid U_i\text{ is }\theta\text{-indecomposable}\}$ is in $E$,
$U$ is $\theta$-indecomposable;
\item For every cardinal $\theta\in(\kappa,\lambda)$ such that $\{i<\kappa\mid U_i\text{ is }\theta^+\text{-complete}\}$ is in $E$,
$U$ is $\theta$-indecomposable.
\end{enumerate}
\end{proposition}

\begin{proof}
(1) As $\lambda \cap \lambda_i=\lambda_i \in U_i$ for each $i<\kappa$, we infer that $\lambda \in U$.
To see that $U$ is closed under the taking supersets, let $A \subseteq B \subseteq \lambda$ with $A\in U$.
In particular, $Y:=\{i<\kappa\mid A\cap\lambda_i\in U_i\}$ is in $E$. For every $i\in Y$,
since $A \cap \lambda_i \s B \cap \lambda_i$, we have $B \cap \lambda_i \in U_i$. Consequently, $B \in U$.
As for closure under intersections, given $A_0,A_1 \in U$, first derive the corresponding $Y_0,Y_1 \in E$ as above.
Then $Y_0 \cap Y_1 \in E$, and for every $i \in Y_0 \cap Y_1$, $(A_0 \cap A_1) \cap \lambda_i=(A_0 \cap \lambda_i) \cap (A_1 \cap \lambda_i) \in U_i$. Consequently, $A_0 \cap A_1 \in U$.
For each $X \in U$, for cofinally many $i<\kappa$, we have that $X \cap \lambda_i \in U_i$, in particular, $|X \cap \lambda_i|=\lambda_i$. Hence $|X|=\lambda$. Thus $U$ is uniform.
Finally, to see that $U$ is an ultrafilter, let $A\s\lambda$.
Define $g: \kappa \rightarrow 2$ via $g(i):=0$ iff $A \cap \lambda_i \in i$.
Pick $Y \in E$ on which $g$ is a constant. If $g[Y]=\{0\}$, then $A \in U$; otherwise, $\lambda \setminus A \in U$.

(2) Suppose $\theta\in E^\lambda_{>\kappa}$ is a cardinal such that $X:=\{i<\kappa\mid U_i\text{ is }\theta\text{-indecomposable}\}$ is in $E$.
Now, given a function $f:B\rightarrow \theta$ with $B\in U$,
consider $Y:=\{i<\kappa\mid B\cap\lambda_i\in U_i\}$ which is a set in $U$.
For each $i\in X\cap Y$, as $B\cap\lambda_i\in U_i$,
we may pick some $T_i\in[\theta]^{<\theta}$ such that $(f\restriction\lambda_i)^{-1}[T_i]$ is in $U_i$.
As $\cf(\theta)>\kappa$, $T:=\bigcup_{i\in X\cap Y}T_i$ is in $[\theta]^{<\theta}$,
and for every $i\in X\cap Y$, it is the case that $f^{-1}[T]\cap\lambda_i\supseteq (f\restriction\lambda_i)^{-1}[T_i]\in U_i$,
so that $f^{-1}[T]\in U$, as sought.

(3) Suppose $\theta\in(\kappa,\lambda)$ is a cardinal such that $X:=\{i<\kappa\mid U_i\text{ is }\theta^+\text{-complete}\}$ is in $E$.
Now, given a function $f:B\rightarrow \theta$ with $B\in U$,
consider $Y:=\{i<\kappa\mid B\cap\lambda_i\in U_i\}$ which is a set in $U$.
For each $i\in X\cap Y$, as $B\cap\lambda_i\in U_i$,
we may pick some $\tau_i<\theta$ such that $(f\restriction\lambda_i)^{-1}[\{\tau_i\}]$ is in $U_i$.
Then $T:=\{\tau_i\mid i\in X\cap Y\}$ is an element of $[\theta]^{\le\kappa}\s[\theta]^{<\theta}$
for which $f^{-1}[T]\in U$, as sought.
\end{proof}

\begin{remark}
Every uniform ultrafilter $U$ over a singular cardinal $\lambda$ is not $\cf(\lambda)$-indecomposable.
To see it, fix a strictly increasing sequence of cardinals $\langle \lambda_i \mid i<\cf(\lambda) \rangle$ converging to $\lambda$,
and define $f: \lambda \rightarrow \cf(\lambda)$ via $f(\alpha):=\min\{i<\cf(\lambda) \mid \alpha<\lambda_i\}$.
Then for every $i<\cf(\lambda)$, $|f^{-1}[i]|=\lambda_i<\lambda$, and hence, $f^{-1}[i]$ is not in $U$.
\end{remark}

\subsection{Ultrafilter number}
The ultrafilter number of an infinite cardinal $\kappa$ is
$$\mathfrak{u}_\kappa:=\min\{\chi(U) \mid U \text{ is a uniform ultrafilter on }\kappa\}.$$
The following is an easy corollary to the work of Raghavan and Shelah \cite{Sh:1160}:
\begin{lemma}\label{l25} Suppose that:
\begin{itemize}
\item $\theta$ is a regular uncountable cardinal;
\item $S$ is a nonempty set of singular cardinals;
\item for every $\sigma\in S$, there is a $\theta$-indecomposable uniform ultrafilter over $\sigma^+$;
\item $\mu$ is a singular strong limit greater than $\sup(S)$ and of cofinality $\theta$.
\end{itemize}

Then, in the forcing extension to add $\mu$-many Cohen reals, $\mathfrak{u}_{\sigma^+}<2^{\sigma^+}$ for every $\sigma\in S$.
\end{lemma}
\begin{proof} Let $\sigma\in S$ and fix a $\theta$-indecomposable uniform ultrafilter $U$ over $\kappa:=\sigma^+$.
Set $\lambda:=\aleph_0$, so that $\lambda^{<\lambda}=\lambda$. Evidently, $\lambda<\cf(\mu)<\kappa<\mu$,
with $U$ being $\cf(\mu)$-indecomposable.
As $U$ is an ultrafilter, $U=U^+$ and then \cite[Definition~4]{Sh:1160} coincides with Definition~\ref{def22} above.
Thus, by \cite[Theorem~7]{Sh:1160}, in $V[G]$,
for $G$ an $\Add(\aleph_0,\mu)$-generic,
every uniform ultrafilter on $\sigma^+$ that extends $U$ has a basis of size no more than $\mu$.
Thus, $\mathfrak{u}_{\sigma^+}\le\mu$.
In addition, $2^{\sigma^+}\ge 2^{\aleph_0}\ge\mu$ with $\cf(2^{\sigma^+})>\sigma^+>\cf(\mu)$ and hence $2^{\sigma^+}>\mu\ge \mathfrak{u}_{\sigma^+}$.
\end{proof}
\subsection{Chromatic number of graphs}\label{compactnessgraph}
A \emph{graph} is a structure $\mathcal{G}=(G,E)$ where $E$ is an irreflexive symmetric relation on $G$.
A coloring $f:G \rightarrow\mu$ is \emph{good} iff for every $(x,y) \in E$, $f(x) \neq f(y)$.
The \emph{chromatic number of $\mathcal G$} is $\chr(\mathcal G):=\min\{\mu\in\on \mid f: G \to \mu \text{ is a good coloring} \}$.

The following lemma can be extracted from the proof of \cite[Theorem~5.2]{MR1117029}.

\begin{lemma}\label{2.12} Suppose that all of the following hold:
\begin{itemize}
\item $\lambda,\kappa,\mu,\theta$ are infinite cardinals with $\lambda\le\kappa$ and $\mu^{\theta^+}<\kappa$;
\item $U$ is a $(\lambda,\kappa)$-regular uniform ultrafilter on $\kappa$ that is $(\theta,\mu^{\theta^+}]$-indecomposable;
\item $\mathcal G$ is a graph of size no more than $\kappa$ such that all of its subgraphs of size less than $\lambda$ have chromatic number no more than $\mu$.
\end{itemize}

Then $\chr(\mathcal G)\le\mu^\theta$.
\end{lemma}
\begin{proof} Without loss of generality, $\mathcal G=(\kappa,E)$.
Define a binary relation $\sim$ on the collection $\mathcal F:=\bigcup_{A\in U}{}^A\mu$,
letting $f\sim f'$ iff $\{\alpha\in\dom(f)\cap\dom(f')\mid f(\alpha)=f(\alpha')\}\in U$.
As $U$ is in particular a filter, $\sim$ is an equivalence relation.
\begin{claim}\label{claim2121} $(\mathcal F,\sim)$ has no more than $\mu^\theta$ many equivalence classes.
\end{claim}
\begin{proof} If $U$ is $\theta^+$-complete, then since it is also $(\theta,\mu]$-indecomposable,
it is $\mu^+$ complete, and then any $f\in\mathcal F$ is $\sim$-equivalent to some constant map,
and there are no more than $\mu$ many such maps.
Hereafter, assume that $U$ is not $\theta^+$-complete.
Thus, letting $\chi:=\theta^+$ and $\nu:=(\mu^{\theta^+})^+$,
it is the case that $U$ is $[\chi,\nu)$-indecomposable that is not $\chi$-complete.
It is also that case that $2^\chi=2^{\theta^+}\le\mu^{\theta^+}<\nu\le\kappa$.
By Lemma~\ref{finest}, then, we may fix a $\vartheta\le\theta$ and a map $\varphi: \kappa\rightarrow \vartheta$ such that,
for any $f: \kappa\to \mu$, there exists a function $g: \vartheta\rightarrow\mu$ such that $f = g\circ \varphi \pmod {U}$.
In particular, for every pair $f\not\sim f'$ of elements of $\mathcal F$ there are corresponding distinct elements $g,g'$ of ${}^{\vartheta}\mu$.
Therefore, the number of equivalence classes is no more than $\mu^{\vartheta}\le\mu^\theta$, as sought.
\end{proof}

Let $\langle A_\beta\mid\beta<\kappa\rangle$ be a witness that $U$ is $(\lambda,\kappa)$-regular.
Consequently, for every $\alpha<\kappa$, $G_\alpha:=\{\beta<\kappa\mid \alpha\in A_\beta\}$
has size less than $\lambda$, and hence we may pick a good coloring $c_\alpha:G_\alpha\rightarrow\mu$.
For every $\beta<\kappa$, for every $\alpha\in A_\beta$, it is the case that $\beta\in G_\alpha$,
so we may define a function $f_\beta:A_\beta\rightarrow\mu$ via $f_\beta(\alpha):=c_\alpha(\beta)$.
\begin{claim} $\beta\mapsto [f_\beta]_\sim$ is a good coloring of $\mathcal G$.
\end{claim}
\begin{proof} Fix an arbitrary pair $\beta\mathrel{E}\beta'$.
The set $A_\beta\cap A_{\beta'}$ is in $U$,
and for every $\alpha\in A_\beta\cap A_{\beta'}$, we have that $\beta,\beta'\in G_\alpha$ with $c_\alpha(\beta)\neq c_\alpha(\beta')$ (since $c_\alpha$ is a good coloring).
So $f_\beta\not\sim f_{\beta'}$.
\end{proof}
By the last two claims, $\chr(\mathcal G)\le\mu^\theta$.
\end{proof}

\begin{corollary}\label{cor211} Suppose that $\lambda$ is a singular strong limit cardinal,
and there exists a $(\cf(\lambda),\lambda)$-indecomposable uniform ultrafilter over $\lambda^+$.

For every graph $\mathcal G$ of size $\lambda^+$, for every cardinal $\mu$,
if $\chr(\mathcal H)\le \mu$ for every subgraph $\mathcal H$ of $\mathcal G$ of size less than $\lambda$,
then $\chr(\mathcal G)\le \mu^{\cf(\lambda)}$.
\end{corollary}
\begin{proof} The case $\mu \geq \lambda$ is trivial. The case $\mu<\lambda$ follows from
Lemma~\ref{2.12} using $\kappa:=\lambda^+$ and $\theta:=\cf(\lambda)$, bearing Fact~\ref{kanreg} in mind.
\end{proof}

\begin{definition}[Product graph]
Given two graphs $\mathcal G_0=(G_0,E_0)$ and $\mathcal G_1=(G_1,E_1)$, the product graph $\mathcal G_0\times\mathcal G_1$ is defined as follows:
\begin{itemize}
\item $V(\mathcal G_0\times\mathcal G_1)$ is $G_0\times G_1:=\{ (g_0,g_1)\mid g\in G_0, g_1\in G_1\}$;
\item $E(\mathcal G_0\times\mathcal G_1)$ is $E_0*E_1:=\{ \{(g_0,g_1), (g_0',g_1')\}\mid (g_0,g_0')\in E_0\ \&\ (g_1,g_1')\in E_1\}$.
\end{itemize}
\end{definition}

Motivated by Hedetniemi's conjecture \cite{MR2615860},
Hajnal \cite{MR815579} proved that for every infinite cardinal $\lambda$,
there are graphs $\mathcal G_0,\mathcal G_1$ of chromatic number $\lambda^+$ whose product has chromatic number $\lambda$.
Then, Soukup \cite{MR937544} gave a consistent example of a 2-cardinal gap by forcing to add graphs $\mathcal G_0,\mathcal G_1$ of size and chromatic number $\aleph_2$
whose product is countably chromatic.
Finally, in \cite{paper16}, an arbitrary gap was shown to be consistently feasible,
where for every infinite cardinal $\lambda$, the axiom $\sd_\lambda$ yields
two graphs $\mathcal G_0,\mathcal G_1$ of size and chromatic number $\lambda^+$ whose product is countably chromatic.
We now show that a $1$-cardinal gap is best possible on the grounds of $\zfc$ alone.

\begin{corollary}\label{c215} Suppose that $\lambda$ is a singular strong limit cardinal,
and there exists a $(\cf(\lambda),\lambda)$-indecomposable uniform ultrafilter over $\lambda^+$.
Then for every two graphs $\mathcal G_0,\mathcal G_1$ of size $\lambda^+$, if $\min\{\chr(\mathcal G_0),\chr(\mathcal G_1)\}\ge\lambda$,
then $\chr(\mathcal G_0\times\mathcal G_1)\ge\lambda$.
\end{corollary}
\begin{proof} Hajnal \cite{pims2004} proved that for every two infinitely chromatic graphs $\mathcal G,\mathcal H$,
every subgraph of $\mathcal G$ of size less than $\chr(\mathcal H)$ has chromatic number $\le\chr(\mathcal G\times\mathcal H)$.
Now, towards a contradiction, suppose that $\mathcal G_0,\mathcal G_1$ are graphs of size $\lambda^+$ such that $\lambda\le\chr(\mathcal G_1)\le\chr(\mathcal G_0)$,
and yet $\mu:=\chr(\mathcal G_0\times\mathcal G_1)$ is strictly smaller than $\lambda$.
By Hajnal's lemma, then, every subgraph of $\mathcal G_0$ of size less than $\lambda$ has chromatic number $\le\mu$.
But, then, by Corollary~\ref{cor211}, $\chr(\mathcal G_0)\le\mu^{\cf(\lambda)}<\lambda$ since $\lambda$ is a strong limit. This is a contradiction.
\end{proof}

\subsection{Stationary reflection}
A family $\mathcal S$ of stationary subsets of a regular uncountable cardinal $\kappa$
is said to \emph{reflect simultaneously} iff there exists an ordinal $\beta\in E^\kappa_{>\omega}$ such that $S\cap\beta$ is stationary for every $S\in\mathcal S$.

\begin{fact}[\cite{paper59}]\label{statreflection} Suppose that $D$ is a $\nu$-complete uniform filter over a regular uncountable cardinal $\kappa$,
and there is a family of fewer than $\nu$ many stationary subsets of $\{\alpha<\kappa\mid D\text{ is }\cf(\alpha)\text{-indecomposable}\}$
that do not reflect simultaneously. Then any element of $D^+$ may be decomposed into $\kappa$-many $D^+$-sets, in particular, $D$ is not an ultrafilter.
\end{fact}

\subsection{Squares and the $C$-sequence number}
A \emph{$C$-sequence} over a regular uncountable cardinal $\kappa$
is a sequence $\vec C=\langle C_\beta\mid\beta<\kappa\rangle$
such that for every $\beta<\kappa$, $C_\beta$ is a closed subset of $\beta$ with $\sup(C_\beta)=\sup(\beta)$.
Square principles have to do with analogous sequences, but possibly wider, as follows.

\begin{definition}[{\cite[Definition~1.16]{MR3914943}}] $\square_\xi(\kappa,{<}\nu,{\sqsubseteq_\chi})$ asserts the existence of a sequence $\langle \mathcal C_\beta\mid\beta<\kappa\rangle$ such that:
\begin{itemize}
\item for every $\beta<\kappa$, $\mathcal C_\beta$ is a nonempty collection of fewer than $\nu$ many closed subsets $C$ of $\beta$ with $\otp(C)\le\xi$ and $\sup(C)=\sup(\beta)$;
\item for every $\beta<\kappa$, for every $C\in\mathcal C_\beta$ with $\otp(C)\ge\chi$,
for every $\alpha\in\acc(C)$, $C\cap\alpha\in\mathcal C_\alpha$;\footnote{Here, $\acc(C):=\{\alpha\in C\mid \sup(C\cap\alpha)=\alpha>0\}$.}
\item for every club $D\s\kappa$, there exists some $\beta\in\acc(D)$ such that $D\cap\beta\notin\mathcal C_\beta$.
\end{itemize}
\end{definition}
\begin{remark}
\begin{enumerate}
\item $\square(\kappa,{<}\nu)$ stands for $\square_\kappa(\kappa,{<}\nu,{\sqsubseteq_\omega})$;
\item $\square_{\lambda,{<}\nu}$ stands for $\square_\lambda(\lambda^+,{<}\nu,{\sqsubseteq_\omega})$ and $\square_{\lambda,\nu}$ stands for $\square_{\lambda,{<}\nu^+}$;
\item $\square_\lambda$ stands for $\square_\lambda(\lambda^+,{<}2,{\sqsubseteq_\omega})$;
\item $\square_\lambda^B$ stands for $\square_\lambda(\lambda^+,{<}2,{\sqsubseteq_\lambda})$.
\end{enumerate}
\end{remark}

\begin{fact}[\cite{paper59}]\label{squares}
Suppose that $\square(\kappa,{<}\nu)$ holds for a given pair $\nu<\kappa$ of infinite regular cardinals,
and that $D$ is a $\nu$-complete uniform filter over $\kappa$.

If there exists a cardinal $\theta\in\reg(\kappa)$ such that $D$ is $\theta$-indecomposable,
then any element of $D^+$ may be decomposed into $\kappa$-many $D^+$-sets, in particular, $D$ is not an ultrafilter.
\end{fact}

\begin{proposition}\label{529}
Suppose that $\square_{\lambda,{<\cf}(\lambda)}$ holds for a given singular cardinal $\lambda$,
and that $D$ is a uniform filter over $\lambda^+$.

If there exists a cardinal $\theta\in\reg(\lambda)$ such that $D$ is $\theta$-indecomposable,
then any element of $D^+$ may be decomposed into $\lambda^+$-many $D^+$-sets, in particular, $D$ is not an ultrafilter.
\end{proposition}
\begin{proof} Assuming the existence of $\theta\in\reg(\lambda)$ such that $D$ is $\theta$-indecomposable,
it is the case that $S:=\{\alpha<\lambda^+\mid D\text{ is }\cf(\alpha)\text{-indecomposable}\}$ is stationary.
By Fact~\ref{statreflection}, then, it suffices to prove that $S$ admits a stationary subset that does not reflect.
Now, by \cite[Theorem~2.3]{MR1838355}, $\square_{\lambda,{<\cf}(\lambda)}$ moreover implies that every stationary subset of $\lambda^+$ admits a stationary subset that does not reflect.
\end{proof}

\begin{proposition}[Apter-Cummings-Schimmerling]\label{552} Suppose:
\begin{itemize}
\item $V$ is an inner model of $W$;
\item in $V$, $\lambda$ is an inaccessible and $\square_\lambda^B$ holds;
\item in $W$, $\lambda$ is a singular cardinal;
\item $(\lambda^+)^V=(\lambda^+)^W$.
\end{itemize}

Then $W\models \square_{\lambda,\cf(\lambda)}$.
\end{proposition}
\begin{proof} This is well-known and follows from a minor tweak of the arguments of \cite[Lemma~2.1]{MR1740483} and \cite[Theorem~2.0]{MR1360144}, but we provide a proof.

Work in $V$.
Let $S:=E^{\lambda^+}_{<\lambda}$. As $\lambda$ is inaccessible, by \cite[Lemma~4.4]{MR1942302},
we may fix a matrix $\langle C_{\beta,i}\mid \beta\in S,~i\in D_\beta\cap E^{\lambda}_{>\omega}\rangle$ satisfying the following:
\begin{enumerate}
\item For every $\beta\in S$, $D_\beta$ is a club in $\lambda$;
\item For all $\beta\in S$ and $i\in D_\beta\cap E^{\lambda^+}_{>\omega}$, $C_{\beta,i}$ is a club in $\beta$ of size less than $\lambda$;\footnote{In particular, $\acc(C_{\beta,i})\s S$.}
\item For all $\beta\in S$, $i\in D_\beta\cap E^{\lambda^+}_{>\omega}$ and $\alpha\in\acc(C_{\beta,i})$,
$i\in D_\alpha$ and $C_{\alpha,i}=C_{\beta,i}\cap\alpha$.
\end{enumerate}

Still in $V$, as $\square^B_\lambda$ holds, we may fix a $C$-sequence $\langle C_\beta\mid \beta<\lambda^+\rangle$
such that for every $\beta\in E^{\lambda^+}_\lambda$:
\begin{enumerate}[label=(\roman*)]
\item $\otp(C_\beta)=\lambda$ and
\item for every $\alpha\in\acc(C_\beta)$, $C_\alpha=C_\beta\cap\alpha$.
\end{enumerate}
Put $B:=E^{\lambda^+}_\lambda\cup\{\acc(C_\beta)\mid \beta\in E^{\lambda^+}_\lambda\}$,
noting that
\begin{enumerate}[label=(\roman*), resume]
\item for every $\beta\in B$, for every $\alpha\in\acc(C_\beta)$, $\alpha\in S\cap B$ and $C_\alpha=C_\beta\cap\alpha$.
\end{enumerate}

Appealing to \cite[Corollary~4.2]{MR3856157} with $\theta:=\cf^W(\lambda)$, $\kappa:=\lambda$
and the sequence $\langle D_\beta\mid\beta<\lambda^+\rangle$,
we obtain in $W$ a subset $D\s\lambda$ of order-type $\theta$ such that for every $\beta<\lambda^+$,
there is a $\delta\in D\cap D_\beta$ with $\cf^V(\delta)>\omega$.

Working in $W$, we define a sequence $\langle\mathcal C_\beta\mid\beta<\lambda^+\rangle$ as follows:
\begin{itemize}
\item If $\beta=0$, then $\mathcal C_\beta:=\{\emptyset\}$;
\item If $\beta=\alpha+1$, then $\mathcal C_\beta:=\{\{\alpha\}\}$;
\item If $\beta\in B\setminus S$, then $\mathcal C_\beta:=\{C_\beta\}$;
\item If $\beta\in S\setminus B$, then $\mathcal C_\beta:=\{C_{\beta,i}\mid i\in D\cap D_\beta, \cf^V(\beta)>\omega\}$;
\item If $\beta\in S\cap B$, then $\mathcal C_\beta:=\{C_\beta,C_{\beta,i}\mid i\in D\cap D_\beta, \cf^V(\beta)>\omega\}$.
\end{itemize}

It is clear that for every $\beta<\lambda^+$, $\mathcal C_\beta$ is a nonempty collection of no more than $\cf(\lambda)$ many closed subsets $C$ of $\beta$
with $\sup(C)=\sup(\beta)$ and $\otp(C)\le\lambda$.
Thus, to verify that our sequence witnesses that $\square_{\lambda,\cf(\lambda)}$ holds, it suffices to show that for every $\beta<\lambda^+$, every $C\in\mathcal C_\beta$ and every $\alpha\in\acc(C)$,
it is the case that $C\cap\alpha\in\mathcal C_\alpha$.
To this end, let such $\beta$, $C$ and $\alpha$ be given. In particular, $\beta\in\acc(\lambda^+)$.
\begin{itemize}[label=$\blacktriangleright$]
\item If $\beta\in B$ and $C=C_\beta$, then, by Clause~(iii),
$\alpha\in S\cap B$ and $C\cap\alpha=C_\alpha\in\mathcal C_\alpha$.
\item Otherwise, $\beta\in S$ and $C=C_{\beta,i}$ for some $i\in D\cap D_\beta$ with $\cf^V(\beta)>\omega$.
In this case, by Clause~(2), $\alpha\in S$ and by Clause~(3), $i\in D_\alpha$ and $C_{\alpha,i}=C_{\beta,i}\cap\alpha$.
Therefore $C\cap\alpha=C_{\alpha,i}\in\mathcal C_\alpha$.\qedhere
\end{itemize}
\end{proof}

\begin{definition}[The $C$-sequence number of $\kappa$, \cite{MR4194561}]\label{defcnm}
If $\kappa$ is weakly compact, then let $\chi(\kappa):=0$. Otherwise, let
$\chi(\kappa)$ denote the least cardinal $\chi\le\kappa$
such that, for every $C$-sequence $\langle C_\beta\mid\beta<\kappa\rangle$,
there exist $\Delta\in[\kappa]^\kappa$ and $b:\kappa\rightarrow[\kappa]^{\chi}$
with $\Delta\cap\alpha\s\bigcup_{\beta\in b(\alpha)}C_\beta$
for every $\alpha<\kappa$.
\end{definition}

An inspection of the proof of \cite[Theorem~4.5]{paper64} makes it clear that the following holds.
\begin{fact}\label{csequencenumber} Suppose that $\lambda$ is a singular strong limit and
$\lambda^+$ carries a $(\cf(\lambda),\lambda)$-indecomposable uniform ultrafilter.
Then $\chi(\vec C)\le\cf(\lambda)$
for every transversal $\vec C$ for $\square(\lambda^+,{<}\lambda)$.
\end{fact}

\subsection{Trees and ascending paths}
A \emph{$\mu$-ascending path} through a $\kappa$-tree $\mathbf T=(T,\le_T)$ is
a sequence $\langle \psi_\beta\mid \beta\in B\rangle$ such that:
\begin{itemize}
\item $B$ is a cofinal subset of $\kappa$;
\item for every $\beta\in B$, $\psi_\beta$ is a function from $\mu$ to the $\beta^{\text{th}}$-level of $T$;
\item for every pair $\alpha<\beta$ of points in $B$, there are $i,j<\mu$ such that $\psi_\alpha(i)\le_T \psi_\beta(j)$.
\end{itemize}
\begin{remark}
A simple interpolation argument shows that the existence of a $\mu$-ascending path
is equivalent to the existence of one indexed by $B:=\kappa$.
\end{remark}

\begin{fact}[implicit in \cite{MR0434818}]\label{fact222} For every pair $\nu<\kappa$ of infinite regular cardinals,
if there exists a $\nu$-complete uniform ultrafilter over $\kappa$,
then no $\kappa$-Aronszajn tree admits a $\theta$-ascending path with $\theta<\nu$.
\end{fact}

\begin{lemma}\label{lemma223} Suppose that:
\begin{itemize}
\item $\lambda$ is a singular strong limit cardinal;
\item there exists a $\cf(\lambda)$-indecomposable uniform ultrafilter over $\lambda^+$;
\item for cofinally many $\nu<\lambda$, there exists a $(\cf(\lambda),\nu)$-indecomposable uniform ultrafilter over $\lambda^+$.
\end{itemize}

Then every $\lambda^+$-tree admits a $\cf(\lambda)$-ascending path.
\end{lemma}
\begin{proof} Suppose that $\mathbf T=(T,\le_T)$ is a $\lambda^+$-tree.
For every $\beta<\lambda^+$, we denote by $T_\beta$ the $\beta^{\text{th}}$-level of $\mathbf T$.
For every $y\in T\setminus\bigcup_{\alpha<\beta}T_\alpha$, we write $y\restriction\beta$ for the unique $x\in T_\beta$ to satisfy $x\le_T y$.
We start exactly as in the proof of \cite[Theorem~3.1]{MR1420265}.
\begin{claim} For some cardinal $\mu<\lambda$, $\mathbf T$ admits a $\mu$-ascending path.
\end{claim}
\begin{proof} Let $U_0$ be a $\cf(\lambda)$-indecomposable uniform ultrafilter on $\lambda^+$.
Let $\langle \lambda_k\mid k<\cf(\lambda)\rangle$ be the increasing enumeration of some cofinal subset of $\lambda$ of minimal possible order-type.

For each $\beta<\lambda^+$,
fix a surjection $s_\beta:\lambda\rightarrow T_\beta$,
and define a function $f_\beta:\lambda^+\setminus\beta$ via
$$f_\beta(\gamma):=\min\{k<\cf(\lambda)\mid s_\gamma(0)\restriction \beta\in s_\beta[\lambda_k]\},$$
and then pick a $k_\beta<\cf(\lambda)$ for which $f_\beta^{-1}[k_\beta]\in U_0$.

By the pigeonhole principle, find a $k^*<\cf(\lambda)$ for which $B:=\{\beta<\lambda^+\mid k_\beta=k^*\}$ is cofinal in $\lambda^+$.
Denote $\mu:=\lambda_{k^*}$ and $\psi_\beta:=s_\beta\restriction\mu$.
We claim that $\langle \psi_\beta\mid \beta\in B\rangle$ constitutes a $\mu$-ascending path.
To see it, fix a pair of points $\alpha<\beta$ in $B$. As $f_\alpha^{-1}[k^*]$ and $f_\beta^{-1}[k^*]$ are both in $U_0$, we may fix some $\gamma$ in their intersection.
It follows that $s_\gamma(0)\restriction\alpha\in s_\alpha[\mu]$ and $s_\gamma(0)\restriction\beta\in s_\beta[\mu]$.
Pick $i,j<\mu$ such that $s_\alpha(i)=s_\gamma(0)\restriction\alpha$ and $s_\beta(j)=s_\gamma(0)\restriction\beta$. Then $\psi_\alpha(i)=s_\alpha(i)\le_T s_\beta(j)=\psi_\beta(j)$.
\end{proof}

Let $\mu$ be given by the previous claim, and
let $\langle \psi_\beta\mid \beta<\lambda^+\rangle$ be a $\mu$-ascending path through $\mathbf T$.
If $\mu\le\cf(\lambda)$, then we are done, so suppose this is not the case.

Fix a large enough $\nu<\lambda$ such that the following two hold:
\begin{itemize}
\item there exists a $(\cf(\lambda),\nu)$-indecomposable uniform ultrafilter $U_1$ over $\lambda^+$;
\item $\max\{2^{\cf(\lambda)^+},\mu\}<\nu$.
\end{itemize}

If $U_1$ is $\cf(\lambda)^+$-complete, then it is $\nu$-complete,
and then Fact~\ref{fact222} implies that $\mathbf T$ is not Aronszajn, and so it moreover admits a $1$-ascending path.
Thus, assume that $U_1$ is not $\cf(\lambda)^+$-complete,
and appeal to Fact~\ref{finest} with $(\chi,\kappa):=(\cf(\lambda)^+,\lambda^+)$
to obtain a corresponding map $\varphi: \lambda^+\rightarrow\cf(\lambda)$.

For every $\beta<\lambda^+$, pick a function $f_\beta:\lambda^+\setminus\beta\rightarrow\mu\times\mu$ satisfying that for every $\gamma\in\lambda^+\setminus\beta$,
if $f_\beta(\gamma)=(i,j)$, then $\psi_\beta(i)\le_T \psi_\gamma(j)$.
Next, by the choice of $\varphi$, fix a function $g_\beta:\cf(\lambda)\rightarrow\mu\times\mu$ such that
$f_\beta = g_\beta\circ \varphi \pmod {U_1}$, i.e.,
the set $\Gamma_\beta:=\{\gamma\in\lambda^+\setminus\beta \mid f_\beta(\gamma)=g_\beta(\varphi(\gamma))\}$ is in $U_1$.
As $\lambda$ is a strong limit, fix a function $g:\cf(\lambda)\rightarrow \mu\times\mu$ for which $B:=\{\beta<\lambda^+\mid g_\beta=g\}$ is cofinal in $\lambda^+$.
Consider $I:=\{i<\mu\mid\exists j<\mu\,((i,j)\in\im(g))\}$
which is a set of size no more than $\cf(\lambda)$.
We claim that $\langle \psi_\beta\restriction I\mid \beta\in B\rangle$ witnesses that $\mathbf T$ admits a $\cf(\lambda)$-ascending path.
Furthermore:

\begin{claim} For every pair $\alpha<\beta$ of points in $B$, there exists an $i\in I$ such that $\psi_\alpha(i)\le_T \psi_\beta(i)$.
\end{claim}
\begin{proof} Let a pair $\alpha<\beta$ be as above.
Pick $\gamma\in \Gamma_\alpha\cap\Gamma_\beta$.
Set $(i,j):=g(\varphi(\gamma))$, and notice that $i\in I$. In addition,
\begin{itemize}
\item $f_\alpha(\gamma)=g_\alpha(\varphi(\gamma))=g(\varphi(\gamma))=(i,j)$, and
\item $f_\beta(\gamma)=g_\beta(\varphi(\gamma))=g(\varphi(\gamma))=(i,j)$.
\end{itemize}

Consequently, $\psi_\alpha(i)\le_T \psi_\gamma(j)$ and $\psi_\beta(i)\le_T\psi(j)$, so that $\psi_\alpha(i)\le_T \psi_\beta(i)$.
\end{proof}
This completes the proof.
\end{proof}

\begin{corollary}\label{cor224} Suppose that:
\begin{itemize}
\item $\lambda$ is a singular strong limit cardinal;
\item there exists a $\cf(\lambda)^+$-complete uniform ultrafilter over $\lambda^+$;
\item there exists a $(\cf(\lambda),\lambda)$-indecomposable uniform ultrafilter over $\lambda^+$.
\end{itemize}

Then there are no $\lambda^+$-Aronszajn trees.
\end{corollary}
\begin{proof} By Lemma~\ref{lemma223} and Fact~\ref{fact222}.
\end{proof}

\begin{corollary}[Shelah-Magidor, {\cite[Theorem~3.1]{MR1420265}}]
If $\lambda$ is the singular limit of strongly compact cardinals, then there are no $\lambda^+$-Aronszajn trees.
\end{corollary}
\begin{proof} Suppose $\lambda$ is as above. In particular it is a singular strong limit,
and for every cardinal $\nu<\lambda$, there is a $\nu$-complete uniform ultrafilter over $\lambda^+$.
Now, appeal to Lemma~\ref{lemma223}.
\end{proof}

The following is well-known.
\begin{fact}\label{prop226} For every cardinal $\lambda$,
if there exists an inner model with the same $\lambda^+$ in which $\lambda$ is strongly inaccessible,
then there exists a special $\lambda^+$-Aronszajn tree.
\end{fact}

\section{Guru sequences and a coherent sequence of measures}\label{mastersequenceandcoherentsequence}
In this section we build a coherent sequence of supercompact measures with some strengthening
of coherent guiding generics that we call \emph{gurus}.\footnote{We thank Inamdar for suggesting this terminology.}
These will be used to construct the Radin forcing and its variants in future sections of this paper.

As a first step, for a regular cardinal $\alpha$, we attach a few objects, as follows:
\begin{itemize}
\item for every $x\in\mathcal P_\alpha(\alpha^+)$, we write $\alpha_x:=\otp(\alpha\cap x)$;
\item $\mathcal{R}_\alpha(\alpha^+) := \{x\in\mathcal{P}_\alpha (\alpha^+) \mid \alpha\cap x\in \reg\}$;
\item $\mathcal{A}_\alpha(\alpha^+) := \{x\in \mathcal{R}_\alpha(\alpha^+) \mid \otp(x) = (\alpha _x)^+\}$;
\item $\mathcal{B}_\alpha(\alpha^+):=\{x\in \mathcal{A}_\alpha(\alpha^+)\mid \alpha_x\text{ is strongly inaccessible}\}$;
\item for every $x \in \mathcal{R}_\alpha(\alpha^+)$, let $\mathbb{C}_{\alpha,x}:=\col((\alpha_x)^{++},{<}\alpha)$.
\end{itemize}

\begin{definition}[Guru]\label{def_guru} For a regular cardinal $\alpha$, a sequence $\vec{t}=\langle t_i \mid i<\alpha^{++} \rangle$ is a \emph{guru for $\alpha$} iff all of the following hold:
\begin{enumerate}
\item for every $i<\alpha^{++}$, $t_i:\mathcal P_\alpha(\alpha^+)\rightarrow V$ is a function such that $t_i(x) \in \mathbb{C}_{\alpha,x}$ for every $x\in \mathcal{A}_\alpha(\alpha^+)$;
\item for all $j<i<\alpha^{++}$, the set $\{ x\in\mathcal P_\alpha(\alpha^+)\mid t_j(x) \nsubseteq t_i(x)\}$ is nonstationary in $\mathcal P_\alpha(\alpha^+)$;
\item for every function $D: \mathcal{A}_\alpha(\alpha^+) \to V$ such that $D(x)$ is a dense subset of $\mathbb{C}_{\alpha,x}$ for every $x\in \mathcal{A}_\alpha(\alpha^+)$,
there is an $i<\alpha^{++}$ such that $t_i(x) \in D(x)$ for all $x\in \mathcal{A}_\alpha(\alpha^+)$.
\end{enumerate}
\end{definition}
\begin{remark}
We call $\vec{t}$ a guru since it can guide us similarly to a guiding generic but with a wider reach.
Let us explain. First, by convention, we write ``$s \in \vec{t}$'' to express that $s=t_i$ for some $i<\alpha^{++}$.
Now, if $U$ is a supercompact measure on $\mathcal{P}_\alpha(\alpha^+)$ and $j_U:V \to M_U$ is the corresponding ultrapower embedding, then $\mathcal{G}_U:=\{[s]_U \mid s \in \vec{t}\}$ satisfies all of the following:
\begin{itemize}
\item $\mathcal G_U$ is a subset of $\col(\alpha^{++},<j_U(\alpha))^{M_U}$;
\item For all $u,v \in \mathcal{G}_U$, there is a $w \in \mathcal{G}_U$ such that $w \leq u,v$;
\item For every dense open subset $D$ of $\col(\alpha^{++},<j_U(\alpha))^{M_U}$ living in $M_U$, we have $\mathcal{G}_U \cap D \neq \emptyset$.
\end{itemize}
Thus, $\mathcal{G}_U$ generates an $(M_U,\col(\kappa^{++},{<}j_U(\kappa)))$-generic filter.
\end{remark}

\begin{lemma}\label{guru_lemma}
Suppose that $\alpha$ is a regular limit cardinal such that $2^{\alpha^+}=\alpha^{++}$. Then there is a guru for $\alpha$.
\end{lemma}
\begin{proof} The proof is a standard diagonalization argument as in \cite[Lemma~3.5]{MR2607544} and even more so as in \cite[Lemma~36]{EskewHayut2018}.
Let $\langle D_i \mid i<\alpha^{++} \rangle$ be an enumeration of all functions $D$ as in Clause~(3) of Definition~\ref{def_guru}.
We build $\vec{t}=\langle t_i \mid i<\alpha^{++} \rangle$ by recursion, as follows.

$\br$ $t_0:\mathcal P_\alpha(\alpha^+)\rightarrow V$ is chosen arbitrarily subject to the requirement that $t_0(x)\in D_0(x)$ for every $x\in \mathcal{A}_\alpha(\alpha^+)$, and
$t_0(x):=\emptyset$ for any other $x$.

$\br$ For every $i<\alpha^{++}$ such that $\langle t_j\mid j\le i\rangle$ has already been successfully defined to satisfy (1)--(3),
pick a $t_{i+1}:\mathcal P_\alpha(\alpha^+)\rightarrow V$ satisfying that $t_{i+1}(x)$ is an element of $D_{i+1}(x)$ extending $t_i(x)$ for every $x\in \mathcal{A}_\alpha(\alpha^+)$,
and
$t_{i+1}(x):=\emptyset$ for any other $x$.
It is clear that $\langle t_j\mid j\le i+1\rangle$ maintains requirements (1)--(3).

$\br$ For every $i\in\acc(\alpha^{++})$ such that $\langle t_j\mid j<i\rangle$ has already been successfully defined to satisfy (1)--(3),
first let $\langle j_\beta \mid \beta<\cf(i) \rangle$ be a strictly increasing sequence of ordinals converging to $i$.
For all $\beta<\gamma<\cf(i)$, let $C_{\beta,\gamma}$ be a club in $\mathcal P_\alpha(\alpha^+)$ disjoint from $\{ x\in\mathcal P_\alpha(\alpha^+)\mid t_{j_\beta}(x) \nsubseteq t_{j_\gamma}(x)\}$.
Then consider the following set $C$ which is a club in $\mathcal P_\alpha(\alpha^+)$:
$$C:=\diagonal_{\beta<\gamma<\cf(i)} C_{\beta,\gamma}:=\{x \in \mathcal{P}_\alpha(\alpha^+) \mid \forall \beta,\gamma\in x\,(\beta<\gamma<\cf(i)\rightarrow x \in C_{\beta,\gamma}) \}.$$
The definition of $t_{i}:\mathcal P_\alpha(\alpha^+)\rightarrow V$ is now divided into two:

$\br\br$ For every $x \in C\cap\mathcal A_\alpha(\alpha^+)$, since $\langle t_{j_\beta}(x) \mid \beta \in x\cap \cf(i) \rangle$ is a decreasing sequence of conditions
in $\mathbb{C}_{\alpha,x}$ and since $| x\cap \cf(i) | <(\alpha_x)^{++}$, we may define $t_i(x)$ to be an element of $D_i(x)$ extending all the conditions in the said sequence.

$\br\br$ For any other $x$, we simply let $t_i(x):=\emptyset$.

It is clear that $\langle t_j\mid j\le i\rangle$ maintains requirements (1)--(3).
\end{proof}

Hereafter, we work in a model $V$ of $\zfc$ in which all of the following hold:
\begin{itemize}
\item $\gch$;
\item $\kappa$ is a $\kappa^{++}$-supercompact cardinal;
\item $\square_\lambda^B$ holds for every inaccessible $\lambda\le\kappa$.
\end{itemize}

This can be obtained by starting with any $\zfc$ model where $\kappa$ is a supercompact cardinal,
running the Jensen iteration to force $\gch$ while preserving the supercompactness of $\kappa$,
then running \cite[Theorem~2]{MR1740483} to preserve the $\gch$ and $\kappa^{++}$-supercompactness of $\kappa$,
while adding $\square_\lambda^B$ for every inaccessible $\lambda\le\kappa$.

\medskip

The following lemma tells us that there is a \emph{coherent sequence of supercompact measures $\vec{U}$}
and a corresponding \emph{coherent sequence of gurus $\vec{\mathbf{t}}$}.
Its proof is an adaptation of \cite[Proposition~2.2]{Krueger2007}.
\begin{lemma}\label{coherentsequence}
There are:
\begin{itemize}
\item a map $o^{\vec{U}}$ from a set of strongly inaccessible cardinals to the ordinals,
\item a sequence $\vec{U}=\left\langle\langle U_{\alpha,i} \mid i<o^{\vec{U}}(\alpha)\rangle\Mid \alpha \in \dom(o^{\vec{U}})\right\rangle$,\footnote{It is possible that $o^{\vec{U}}(\alpha)=0$, in which case, $\langle U_{\alpha,i} \mid i<o^{\vec{U}}(\alpha)\rangle$ is the empty sequence. These cardinals may end up as successor elements of the generic Radin club.} and
\item a sequence $\vec{\mathbf{t}}=\langle \vec{t}_\alpha \mid \alpha \in \dom(o^{\vec{U}}) \rangle$
\end{itemize}
such that all of the following hold:
\begin{enumerate}[label=\textup{(\arabic*)}]
\item $\max(\dom(o^{\vec{U}}))=\kappa$ with $o^{\vec{U}}(\kappa)=\kappa^{+3}$;
\item for every $\alpha \in \dom(o^{\vec{U}})$, $\vec{t}_\alpha$ is a guru for $\alpha$;
\item for every $\alpha \in \dom(o^{\vec{U}})$ and every $i<o^{\vec{U}}(\alpha)$:
\begin{enumerate}[label=\textup{(\alph*)}]
\item $U_{\alpha,i}$ is a supercompact measure on $\mathcal{P}_\alpha(\alpha^+)$,
and we let $j_{\alpha,i}: V \to M_{\alpha,i}$ denote the corresponding ultrapower embedding;
\item $\{ x\in \mathcal B_\alpha(\alpha^+)\mid \alpha_x\in\dom(o^{\vec{U}})\}\in U_{\alpha,i}$;
\item $j_{\alpha,i}(\vec{U}\restriction \alpha) \restriction \alpha=\vec{U} \restriction \alpha$;
\item $j_{\alpha,i}(\vec{U}\restriction \alpha) (\alpha)=\langle U_{\alpha,k}\mid k<i\rangle$;
\item $j_{\alpha,i}(\vec{\mathbf{t}} \restriction \alpha)(\alpha)=\vec{t}_\alpha$.
\end{enumerate}
\end{enumerate}
\end{lemma}
\begin{proof} For the sake of this proof, define a relation $\preceq$ over $V\times V$,
by letting $(a,b)\preceq (c,d)$ iff $a=c$ and $b\s d$.
Let $j:V \to M$ witness that $\kappa$ is $\kappa^{++}$-supercompact. Define a map $g: \kappa \to V_\kappa$ by recursion on $\alpha<\kappa$, as follows.
For every $\alpha<\kappa$ such that $g\restriction\alpha$ has already been defined,
set $g(\alpha):=\emptyset$ unless $\alpha$ is strongly inaccessible and $2^{\alpha^+}=\alpha^{++}$,
in which case, we let $g(\alpha)$ be a maximal element of $(X_\alpha,\preceq)$, where $X_\alpha$ is the collection of all pairs $(\vec{t},\langle U_{i} \mid i<l \rangle)$, where $l$ could possibly be $0$, such that all of the following hold:
\begin{itemize}
\item $\vec t$ is a guru for $\alpha$;
\item for every $i<l$, $U_i$ is a supercompact measure on $\mathcal P_\alpha(\alpha^+)$, and we let $j_{U_i}:V \to M_i$ denote the corresponding ultrapower embedding;
\item $\langle U_i\mid i<l\rangle$ is $\lhd$-increasing, that is, $U_i\in M_k$ for all $i<k<l$, or $l=0$, i.e. the sequence $\langle U_i \mid i<l \rangle$ is empty;
\item for every $i<l$, $j_{U_i}(g\restriction\alpha)(\alpha)$ coincides with $(\vec{t},\langle U_k \mid k<i \rangle)$.
\end{itemize}

Note that by Lemma~\ref{guru_lemma}, there exists a $\vec t$ for which $(\vec t,\emptyset)\in X_\alpha$. As $X_\alpha$ is nonempty,
by Zorn's lemma, $(X_\alpha,\preceq)$ indeed admits a maximal element.
This completes the recursion.

Next, define $o^{\vec{U}}$ to be the following function from $\{\alpha<\kappa\mid g(\alpha)\neq\emptyset\}\cup\{\kappa\}$ to the ordinals, as follows.
For every $\alpha\in\dom(o^{\vec{U}})\cap\kappa$, write $(\vec{t}_\alpha,\langle U_{\alpha,i} \mid i<l_\alpha \rangle):=g(\alpha)$,
and then let $o^{\vec{U}}(\alpha):=l_\alpha$.
Finally, to define $o^{\vec{U}}(\kappa)$, write $(\vec{t}_\kappa,\langle U_{\kappa,i} \mid i<l_\kappa \rangle):=j(g)(\kappa)$
and then let $o^{\vec{U}}(\kappa):=l_\kappa$.

\begin{claim} $l_\kappa=\kappa^{+3}$.
\end{claim}
\begin{proof} For every $\iota<l_\kappa$, let $j_{U_{\kappa,\iota}}:V \to M_{\kappa,\iota}$ denote the corresponding ultrapower embedding.
Note that for every $k<l_\kappa$, for every $\iota<k$, $U_{\kappa,\iota}$ is in $M_{\kappa,k}$ and is represented by a function from $\mathcal{P}_\kappa(\kappa^+)$ to $V_\kappa$,
and the number of such functions is $\kappa^{++}$, so that $k<\kappa^{+3}$. Therefore, $o^{\vec{U}}(\kappa) \leq \kappa^{+3}$.
Suppose for a contradiction that $o^{\vec{U}}(\kappa)<\kappa^{+3}$. Using our original $\kappa ^{++}$-supercompact embedding we define a fine, normal ultrafilter $U^*$ on $\mathcal{P}_\kappa(\kappa^+)$ as follows:
$$U^*:=\{ A \subseteq \mathcal{P}_\kappa(\kappa^+)\mid j``\kappa^+ \in j(A)\}.$$
By the closure degree of $M$, it is the case that $U^* \in M$. Let $j^*:M \to \Ult(M,U^*)$. Also let $i:V \to N$ be the ultrapower embedding using $U^*$ in $V$,
and let $k:N \to M$ be defined via $k([f])_{U^*}:=j(f)(j``\kappa^+)$.
It is routine to check that $j=k \circ i$.
Since $M$ and $N$ are closed under $\kappa^+$-sequences of $V$, we have that $\kappa^{++}=(\kappa^{++})^M=(\kappa^{++})^N$, and hence, $\crit(k)\geq (\kappa^{+3})^N$.

Let $(\langle \vec{t}',U_{\kappa,\iota}')\mid \iota<l_\kappa' \rangle:= i(g)(\kappa)$. We see that $k(\vec{t}')=\vec{t}_\kappa$.
Since $k(l_\kappa')=l_\kappa<\kappa^{+3}$, we have $l_\kappa'<(\kappa^{+3})^N$. Therefore, $k(l_\kappa')=l_\kappa'$ and for all $\iota$, $k(U_{\kappa,\iota}')=U_{\kappa,\iota}'$.
It follows that $i(g)(\kappa)=\langle (\vec{t},U_{\kappa,\iota}) \mid \iota<l_\kappa \rangle$.
By \cite[Lemma~1.4]{Krueger2007}, $j^*=i \restriction M$, so it is the case that $j^*(g)(\kappa)=\langle (\vec{t},U_{\kappa,\iota}) \mid \iota<l_\kappa \rangle$.
Because we can attach $U^*$ to the end of $j^*(g)(\kappa)$, this contradicts the maximality used in the definition of $g$.
\end{proof}
Let $\vec{U}:=\langle\langle U_{\alpha,i} \mid i<o^{\vec{U}}(\alpha)\rangle\mid \alpha \in \dom(o^{\vec{U}})\rangle$ and
$\vec{\mathbf{t}}:=\langle \vec{t}_\alpha \mid \alpha \in \dom(o^{\vec{U}}) \rangle$.
Note that $\vec{U}$ and $\vec{\mathbf{t}}$ are as sought.
Indeed, a verification at level $\alpha<\kappa$ follows from the definition of $g$. A verification at level $\kappa$ follows from the definition of $g$ and an agreement between $V$ and $M$.
\end{proof}

Let $\vec{U}$ and $\vec{\mathbf{t}}$ be as in Lemma~\ref{coherentsequence}.
\begin{convention}
For each $\alpha\in\dom(o^{\vec{U}})$, write $\vec{t}_\alpha=\langle t_{\alpha,i} \mid i<\alpha^{++} \rangle$.
We write $\bigcap \vec{U}(\alpha)$ for $\bigcap_{i<o^{\vec{U}}(\alpha)} U_{\alpha, i}$.
Next, for each $A\in \bigcap \vec{U}(\alpha)$, we disjointify $A=A_0 \cup A_1$ as follows:
\begin{itemize}
\item $A_0:=\{ x \in A\mid o^{\vec{U}}(\alpha_x)=0\}$, and
\item $A_1:=\{x \in A \mid o^{\vec{U}}(\alpha_x)>0\}$.
\end{itemize}
Note that $A_0 \in U_{\alpha,0}$.
\end{convention}
\begin{definition}[Implicit gurus] For each $\alpha\in\dom(o^{\vec{U}})$, we let $\mathcal G(\alpha)$ be the collection of all functions $I$ with domain in $\bigcap\vec U(\alpha)$ such that, for some $i<\alpha^{++}$,
$I\s t_{\alpha,i}$.
\end{definition}

\begin{definition}[Strong inclusion]\label{stronginclusion}
For two sets of ordinals $x,y$, we write
$$x \ssim y\text{ iff }x \subseteq y\text{ and }|x|^+<\otp(|\sup(x)|\cap y).$$
And we denote a class $x^\uparrow:=\{ y\mid x\ssim y\}$.
\end{definition}

The motivation of Definition~\ref{stronginclusion} is that for $\alpha,i,$ if $A_x \in U_{\alpha,i}$ for all $x \in \mathcal{P}_\alpha(\alpha^+)$, then the corresponding diagonal intersection

$$A:=\diagonal_{x \in \mathcal{P}_\alpha(\alpha^+)} A_x:=\{y \in \mathcal{P}_\alpha(\alpha^+) \mid \forall x (x \ssim y \to y \in A_x) \}$$
will belong to $U_{\alpha,i}$, which is necessary when proving the Prikry property (see for example the proof of Lemma~\ref{strongprikrylemma1} below).
Furthermore, note that for $x,y\in\mathcal A_\alpha(\alpha^+)$, $x\ssim y$ iff $x\s y$ and $(\alpha_x)^{++}<\alpha_y$.

\begin{definition}
For $A \in \bigcap \vec{U}(\kappa)$ and $x \in \mathcal{P}_\kappa(\kappa^+)$, write $A \restriction x:=\{y \in A \mid y \ssim x \}$.
\end{definition}
Assume that $x \in \mathcal{A}_\kappa(\kappa^+)$. Recall that $\pi_x$ stands for the inverse collapsing map of $x$ and $\kappa_x=\otp(\kappa \cap x)$.
This map naturally induces an isomorphism from $\mathcal{P}_{\kappa_x}((\kappa_x)^+)$ to $\mathcal{P}_{\kappa_x}(x)$ via $\mathcal{P}_{\kappa_x}((\kappa_x)^+) \ni x \to \pi_x[a] \in \mathcal{P}_{\kappa_x}(x) $ (we denote the function by $\pi_x^*$), which in turn gives rise to a translation of each $U_{\kappa_x,i}$ to a corresponding ultrafilter on $\mathcal{P}_{\kappa_x}(x)$, which we hereon call $\mathbf{U}_{x, i}$.
Note that
\begin{align*}
\mathbf{U}_{x, i}&=\{A \subseteq \mathcal{P}_{\kappa_x}(x) \mid (\pi_x^*)^{-1}[A] \in U_{\kappa_x,i}\}\\
&=\{A \subseteq \mathcal{P}_{\kappa_x}(x) \mid j_{\kappa_x,i}[\kappa_x^+] \in j_{\kappa_x,i}((\pi_x^*)^{-1}[A])\} \\
&=\{A \subseteq \mathcal{P}_{\kappa_x}(x)\mid j_{\kappa_x,i}(\pi_x^*)(j_{\kappa_x,i}[\kappa_x^+]) \in j_{\kappa_x,i} (A)\}\\
&=\{A \subseteq \mathcal{P}_{\kappa_x}(x)\mid j_{\kappa_x,i}[x] \in j_{\kappa_x,i} (A)\}
\end{align*}
Thus, $\mathbf{U}_{x, i}=\{A \in \mathcal P_{\kappa_x}(x)\mid j_{\kappa_x,i}``x \in j_{\kappa_x,i} (A)\}$.
We likewise define $\vec{\mathbf{U}}_{x}:=\langle \mathbf{U}_{x, i} \mid i<o^{\vec{U}}(\kappa_x) \rangle$ and $\bigcap \vec{\mathbf{U}}(x):= \bigcap_{i<o^{\vec{U}}(\kappa_x)} U_{x,i}$.
We also disjointify each $A$ in $\bigcap \vec{\mathbf{U}}(x)$ as $A_0 \cup A_1$ in the same manner as before, so that $A_0 \in \mathbf{U}_{x,0}$.
\begin{proposition}\label{measureone1}
Let $A \in \bigcap \vec{U}(\kappa)$.
The set $B:= \{x\in \mathcal{P}_\kappa(\kappa^+) \mid A \restriction x \in \bigcap \vec{\mathbf{U}}({x})\}$ is in $\bigcap \vec{U}(\kappa)$.
\end{proposition}
\begin{proof}
Note that
$$\begin{aligned}B=&\ \{x\in\mathcal P_\kappa(\kappa^+)\mid \forall \iota <o^{\vec U}(\kappa_x)\,[A\restriction x\in \mathbf U_{x,i}]\}\\
=&\ \{x\in\mathcal P_\kappa(\kappa^+)\mid \forall \iota <o^{\vec U}(\kappa_x)\,[j_{\kappa_x,\iota}``x\in j_{\kappa_x,\iota} (A\restriction x)]\}\\
=&\ \{x\in\mathcal P_\kappa(\kappa^+)\mid \forall \iota <o^{\vec U}(\kappa_x)\,[j_{\kappa_x,\iota}``x\in j_{\kappa_x,\iota} (\{y\in A\mid y\ssim x\})]\}\\
=&\ \{x\in\mathcal P_\kappa(\kappa^+)\mid \forall \iota <o^{\vec U}(\kappa_x)\,[j_{\kappa_x,\iota}``x\in \{y\in j_{\kappa_x,\iota}(A)\mid y\ssim j_{\kappa_x,\iota}(x)\}]\}\\
=&\ \{x\in\mathcal P_\kappa(\kappa^+)\mid \forall \iota <o^{\vec U}(\kappa_x)\,[j_{\kappa_x,\iota}``x\in j_{\kappa_x,\iota}(A)\ \&\ j_{\kappa_x,\iota}``x\ssim j_{\kappa_x,\iota}(x)]\}.
\end{aligned}$$

Now, to show that $B\in\bigcap\vec U(\kappa)$, let $i<o^{\vec{U}}(\kappa)$ and we shall show that $B\in U_{\kappa,i}$, i.e., that $x:=j_{\kappa,i}``\kappa^+$ is in $j_{\kappa,i}(B)$.
Since $j_{\kappa,i}(\kappa)_x=\otp(j_{\kappa,i}(\kappa) \cap x)=\kappa$ and $j_{\kappa,i}(o^{\vec U})(j_{\kappa,i}(z \mapsto \kappa_z))(x)=i$, this amounts to showing that for all $\iota<i$, it is the case that
$$j_{\kappa,\iota}``x\in j_{\kappa,\iota}(j_{\kappa,i}(A))\ \&\ j_{\kappa,\iota}``x\ssim j_{\kappa,\iota}(x).$$
The former statement follows from the choice of $A$. To show that $j_{\kappa,\iota}``x\ssim j_{\kappa,\iota}(x)$, we first note that clearly $j_{\kappa,\iota}``x\subseteq j_{\kappa,\iota}(x)$. Let $M=\Ult_{U_{\kappa,\iota}}(\Ult_{U_{\kappa,i}}(V))$.
Notice that in $M$, $j_{\kappa, \iota}(x)$ is an increasing sequence of ordinals of order-type $j_{\kappa,\iota}(\kappa^+)$,
where the $(j_{\kappa,\iota}(\kappa))^{\text{th}}$-element is $j_{\kappa,i}(j_{\kappa,\iota}(\kappa))$, which is a cardinal in $M$.
Let $\gamma<j_{\kappa,\iota}(\kappa)$ and denote by $\alpha_\gamma$
be the $\gamma^{\text{th}}$-element of $j_{\kappa, \iota}(x)$.
Then,
$\alpha_\gamma<j_{\kappa,i}(j_{\kappa,\iota}(\kappa))$,
and
$$\alpha_\gamma \in | \sup j_{\kappa,\iota}"x| \cap j_{\kappa,\iota}(x).$$
Consequently, $\otp(| \sup j_{\kappa,\iota}"x| \cap j_{\kappa,\iota}(x))\ge j_{\kappa,\iota}(\kappa^+)>\kappa^{++}=|\sup j_{\kappa,\iota}"x|^+$.
Hence, we are done.
\end{proof}

Fix $\alpha \in \dom(o^{\vec{U}})$, $s \in \vec{t}_\alpha$. For each $x \in \mathcal{P}_\alpha(\alpha^+)$, let $s_x$ be a function whose domain is $\mathcal{P}_{\alpha_x}((\alpha_x)^+)$ and for each $y \in \dom(s_x)$, $s_x(y):=s(\pi_x[y])$.

\begin{proposition}
Let $s \in \vec{t}_\kappa$. The set $B:=\{x \in \mathcal{P}_\kappa(\kappa^+)\mid s_x \in \vec{t}_{\kappa_x}\}$ is in $\bigcap \vec{U}(\kappa)$.
\end{proposition}
\begin{proof}
Fix $i<o^{\vec{U}}(\kappa)$ and we shall show that $x:=j_{\kappa,i}``\kappa^+$ is in $j_{\kappa,i}(B)$. Note that
$$j_{\kappa,i}(B)=\{z\in \mathcal P_{j_{\kappa,i}(\kappa)}(j_{\kappa,i}(\kappa^+))\mid j_{\kappa,i}(y\mapsto s_y)(z) \text{ is in } j_{\kappa, i}(\vec{t})_{j_{\kappa,i}(\kappa)_z } \}.\footnote{Here, $j_{\kappa,i}(\kappa)_z=\otp(j_{\kappa,i}(\kappa) \cap z)$.}$$
Since $j_{\kappa,i}(\kappa)_x=\kappa$, this amounts to showing that $t:=j_{\kappa,i}(y\mapsto s_y)(x)$ is in $\vec{t}_\kappa$. However, we see that $\dom(t)=\dom(s)$ and for $z \in \dom(t)$,
$$t(z)=j_{\kappa,i}(\pi_x[z])=j_{\kappa,i}(s)(j_{\kappa,i}(z))=j_{\kappa,i}(s(z))=s(z).$$
Hence, $t=s$, and the proof is done.
\end{proof}

Finally, we define the relativized version of implicit gurus. For $x \in \mathcal{P}_\kappa(\kappa^+)$
we let $\mathbf G(x)$ be the collection of all functions $I$ with domain in $\bigcap\vec{\mathbf{U}}(x)$ such that, for some $i<(\kappa_{x})^{++}$,
$I\s t_{x,i}$, where $t_{x,i}:\mathcal P_{\kappa_x}(x)\rightarrow V$ is defined via
$$t_{x,i}(y):= t_{\kappa_x,i}(\pi_x^{-1}[y]).$$

Note that $t_{x,i}(y) \in \mathbb{C}(\kappa_x,\pi_x^{-1}[y])=\col(\kappa_y^{++},\kappa_x)$.

\section{Radin forcing with gurus}\label{radinforcing}
We continue with our setup from Section~\ref{mastersequenceandcoherentsequence}. In particular, $2^{\kappa^+}=\kappa^{++}$, there is a coherent sequence $\vec{U}$ of supercompact measures,
and a corresponding coherent sequence of guru sequences $\vec{\mathbf{t}}$.

We will define a forcing like in \cite{Mag77}, but work for the case of uncountable cofinalities, and we shall want to preserve the regularity of $\kappa$.
Thus, the forcing will be Radin-like.
One may view the forcing in the upcoming Definition~\ref{radinforcingdef} as a Radin forcing with interleaved collapses.
The major difference here is that we employ a strengthening of guiding generics, called \emph{gurus}.
The gurus will pay off when we project our forcing into an appropriate forcing in Section~\ref{pradinforcing} and obtain a lower chain condition as in Proposition~\ref{pradincc}.

We define a supercompact Radin forcing with interleaved
collapses $\mathbb{R}_{\vec{U},\vec{\mathbf{t}}}$, as follows.

\begin{definition}\label{radinforcingdef}
The forcing $\mathbb{R}_{\vec{U},\vec{\mathbf{t}}}$ consists of conditions of the form $$p=\langle c_{-1},w_0,c_0, \ldots,w_{n-1},c_{n-1}, w_n\rangle,$$ where

\begin{enumerate}
\item $n<\omega$;
\item for every $i\le n$:
\begin{enumerate}
\item if $0=i<n$:
\begin{itemize}
\item $w_0$ is a pair $\langle x_0,I_0 \rangle$;
\item $x_0\in \mathcal{P}_\kappa(\kappa^+)$;
\item if $o^{\vec{U}}(\kappa_{x_0})>0$, then $I_0 \in \mathbf{G}({x_0})$;
otherwise, $I_0=\emptyset$;
\item $c_{-1} \in \col(\omega_1,{<}\min\{\kappa_{x_0},\kappa_x\mid x\in \dom(I_0)\})$;
\end{itemize}
\item if $0<i<n$:
\begin{itemize}
\item $w_i$ is a pair $\langle x_i,I_i \rangle$;
\item $x_i\in \mathcal{P}_\kappa(\kappa^+)$;
\item if $o^{\vec{U}}(\kappa_{x_i})>0$, then $I_i \in \mathbf{G}({x_i})$;
otherwise, $I_i=\emptyset$;
\item $c_{i-1} \in \col((\kappa_{x_{i-1}})^{++},{<}\min\{\kappa_{x_i},\kappa_x\mid{x\in\dom(I_i)}\})$;
\end{itemize}
\item if $i=n$:
\begin{itemize}
\item $w_n$ is a pair $\langle x_n,I_n \rangle$;
\item $x_n=\kappa^{+}$;
\item $I_n \in \mathcal{G}({\kappa})$;
\item $c_{n-1} \in \col((\kappa_{x_{n-1}})^{++},{<}\min\{\kappa_x \mid x\in \dom(I_n)\})$;
\end{itemize}
\end{enumerate}
\item $x_0 \ssim \cdots \ssim x_{n-1}$.

\end{enumerate}
\end{definition}
\begin{convention} For the rest of this section, we write $\mathbb{R}$ instead of $\mathbb{R}_{\vec{U},\vec{\mathbf{t}}}$.
\end{convention}

If $p$ is as above, we write $\ell(p):=n$.
The \emph{working part} of $p$, denoted $\stem(p)$, is $\langle x_0, \ldots, x_{n-1} \rangle$.
For $x$ in $\stem(p)$, we shall write $i_p(x)$ for the unique $i$ such that $x=x_i$.
The \emph{collapse part of $p$} is $\langle c_{-1}, \ldots, c_{n-1} \rangle$. The \emph{top part of $p$}, denoted $\tp(p)$, is $\langle \kappa^{+},I_n \rangle$. We sometimes concise the expression of $p$ while highlighting its top part by writing $\langle \vec{c},\vec{w},\langle \kappa^{+},I \rangle \rangle$ for $p$, where $\vec{c}=\langle c_{-1},c_0,\ldots, c_{n-2},c_{n-1} \rangle$ is the collapse part of $p$ and $\vec{w}=\langle w_0,\ldots, w_{n-1} \rangle$.
Alternatively, we may concise $p$ partially as
$$\langle \vec{c},\vec{w},w_m,c_m,\ldots, w_{n-1},c_{n-1},\langle \kappa^{+},I \rangle \rangle$$ where $\vec{c}=\langle c_{-1},\ldots, c_{m-1} \rangle$ and $\vec{w}=\langle w_0,\ldots, w_{m-1} \rangle$.
Later on, we may add the superscript $p$ to indicate components of $p$, e.g., $c_i^p,w_i^p, I_i^p$.

Given $p,q \in \mathbb{R}_{\vec{U},\vec{\mathbf{t}}}$, say
\begin{itemize}
\item $p=\langle c_{-1}^p,w_0^p,c_0^p,\ldots, w_{\ell(p)-1}^p,c_{\ell(p)-1}^p,w_{\ell(p)}^p \rangle$, and
\item $q=\langle c_{-1}^q,w_0^q,c_0^q,\ldots, w_{\ell(q)-1}^q,c_{\ell(p)-1}^q, w_{\ell(q)}^q \rangle$,
\end{itemize}
define $p \leq q$ ($p$ is stronger than $q$) iff all of the following hold:
\begin{enumerate}
\item $\ell(p) \geq \ell(q)$;
\item $c_{-1}^p \leq c_{-1}^q$;
\item there are $0 \leq i_0 < i_1<\cdots <i_{\ell(q)}=\ell(p)$ such that, for each $k\leq\ell(q)$,
\begin{itemize}
\item $x_{i_k}^p=x_k^q$,\footnote{Recall that for $k=\ell(q)$ the two $x$'s are in fact $\kappa^+$.} $\dom(I_{i_k}^p)\s\dom(I_k^q)$, and $c_{i_k}^p \leq c_k^q$;\footnote{That is, reverse inclusion.}
\item if $i_{k-1}<i<i_k$ (where $i_{-1}=-1$), then $x_i^p \in\dom(I_k^q)$, $\dom(I_i^p)\s \dom(I_k^q)$, and $c_i^p \leq I_k^q(x_i^p)$;
\item if $i_{k-1}<i \leq i_k$, then for each $x \in\dom(I_i^p)$, $I_i^p(x) \leq I_k^q(x)$.
\end{itemize}

\end{enumerate}

We say that $p$ is a \emph{direct extension of} $q$, denoted $p \leq^*q$, iff $p\leq q$ and $\ell(p)=\ell(q)$.

\begin{remark}\label{directstarextension}
For all $p,p_0,p_1 \in \mathbb{R}_{\vec{U},\vec{\mathbf{t}}}$, if $p_0,p_1 \leq^* p$ and the collapse parts of $p_0$ and $p_1$ are compatible, then there is a $p_2 \leq^* p_0,p_1$.
\end{remark}

\begin{definition}[$0$-step extension]
We say that $p$ is a \emph{$0$-step extension of} $q$, denoted $p \leq^{**} q$, iff
\begin{enumerate}
\item $p \leq^* q$;
\item the collapse parts of $p$ and $q$ are equal;
\item for every $i\leq \ell(p)$, for every $x \in\dom(I_i^p)$,
$I_i^p(x)=I_i^q(x)$.
\end{enumerate}
\end{definition}

\begin{definition}[$1$-step extension]\label{1stepextension}
Let $p$ be a condition.

$\br$ For $x \in \dom(I_i^p)$ with $i< \ell (p)$,
the \emph{$1$-step extension of $p$ by $x$}, denoted $p+\langle x \rangle$,
is the condition $$q=\langle c_{-1}^p,w_0^p,c_0^p,\ldots,w',c',v',c_i^p,w_{i+1}^p,c_{i+1}^p,\ldots, w_{\ell (p)-1}^p,c_{\ell (p)-1}^p, w_{\ell(p)}^p \rangle,$$ where

\begin{enumerate}
\item $w'=\langle x, I_i^p \restriction (\dom(I_i^p) \restriction x) \rangle$;
\item $c'=I_i^p(x)$;
\item $v'=\langle x_i^p, I_i^p \restriction x^\uparrow \rangle$.
\end{enumerate}

$\br$ For $x \in \dom(I_{\ell (p)}^p)$, the \emph{$1$-step extension of $p$ by $x$}, denoted $p+\langle x \rangle$, is

$$q=\langle c_{-1}^p,w_0^p,c_0^p,\ldots, w_{\ell (p)-1}^p,c_{\ell (p)-1}^p,w',c', \langle \kappa^{+}, I' \rangle \rangle,$$
where

\begin{enumerate}
\item $w'=\langle x,I_{\ell (p)}^p \restriction (\dom(I_{\ell (p)}^p) \restriction x) \rangle$;
\item $c'=I_{\ell (p)}^p(x)$;
\item $I'=I_{\ell(p)}^p \restriction x^\uparrow$.
\end{enumerate}
From the notation in Definition~\ref{1stepextension}, we allow such a one-step extension if such $q$ is indeed a condition. This is possible since by Proposition~\ref{measureone1}, the collection of $x$ with $\dom(I_i^p) \restriction x \in \vec{\mathbf{U}}(x)$ (so that $I_i^p \restriction (B \restriction x) \in \mathbf{G}(x)$) is large. Furthermore, the collapse condition $c'$ will belong to $\col( \kappa_x^{++},<\min\{\kappa_{x_i^p}, \kappa_y \mid y \in \dom(I_i^p \restriction x^{\uparrow})\})$.
\end{definition}

Likewise, we can define $p+\langle x_0,\ldots, x_{k+1} \rangle$
recursively as $(p+\langle x_0,\ldots, x_k \rangle)+\langle x_{k+1} \rangle$.
We shall say that $p+\langle x_0,\ldots, x_n\rangle$ is an \emph{$(n+1)$-step extension of $p$}.
Note that $p \leq q$ is equivalent to $p$ being a direct extension of some $n$-step extension of $q$ for some $n<\omega$.
Now, to another useful variation.

\begin{definition}[$1$-step extension while shrinking]\label{def36}
Let $p$ be a condition.

$\br$ For $x \in \dom(I_i^p)$ with $i< \ell (p)$ and $B\in\bigcap\vec{\mathbf{U}}(x)$,
the \emph{$1$-step extension of $p$ by $\langle x,B\rangle$}, denoted $p+\langle x,B \rangle$,
is the condition $$q=\langle c_{-1}^p,w_0^p,c_0^p,\ldots,w',c',v',c_i^p,w_{i+1}^p,c_{i+1}^p,\ldots, w_{\ell (p)-1}^p,c_{\ell (p)-1}^p, w_{\ell(p)}^p \rangle,$$ where

\begin{enumerate}
\item $w'=\langle x, I_i^p \restriction (B\restriction x) \rangle$;
\item $c'$ and $v'$ are as in the definition of $p+\langle x\rangle$.
\end{enumerate}

$\br$ For $x \in \dom(I_{\ell (p)}^p)$ and $B\in\bigcap\vec{\mathbf{U}}(x)$,
the \emph{$1$-step extension of $p$ by $\langle x,B\rangle$}, denoted $p+\langle x,B \rangle$, is the condition

$$q=\langle c_{-1}^p,w_0^p,c_0^p,\ldots, w_{\ell (p)-1}^p,c_{\ell (p)-1}^p,w',c', \langle \kappa^{+}, I' \rangle \rangle,$$
where

\begin{enumerate}
\item $w'=\langle x, I_{\ell (p)}^p \restriction ( B\restriction x) \rangle$;
\item $c'$ and $I'$ are as in the definition of $p+\langle x\rangle$.
\end{enumerate}
\end{definition}

We define $p+\langle x_0,\ldots, x_{k+1},B_0\ldots,B_{k+1}\rangle$
recursively as $(p + \langle x_0,\ldots, x_k,B_0,\allowbreak\ldots, B_k \rangle)+\langle x_{k+1},B_{k+1} \rangle$.
We shall say that $p+\langle x_0,\ldots, x_n,B_0,\ldots,B_n\rangle$ is an \emph{$(n+1)$-step extension of $p$ while shrinking}.

\begin{proposition}\label{radincc}
$\mathbb{R}$ is $\kappa^+$-Linked$_0$, that is, there exists a map $\varphi:\mathbb R\rightarrow\kappa^+$ such that for all $p,q\in\mathbb R$ with $\varphi(p)=\varphi(q)$,
there exists an $r\in\mathbb R$ with $r\le^* p$ and $r\le^*q$.

In particular, $\mathbb{R}$ is $\kappa^{++}$-Knaster, and hence, has the $\kappa^{++}$-cc.
\end{proposition}

\begin{proof} As $|\mathcal H_{\kappa^+}|=\kappa^+$, it suffices to define a map $\varphi:\mathbb R\rightarrow\mathcal H_{\kappa^+}$ with the above crucial property.
We do so as follows.
Given a condition $p = \langle c_{-1}^p, w_0^p, c_0^p, \ldots,\allowbreak w_{\ell(p)-1}^p, c_{\ell(p)-1}^p,\allowbreak w_{\ell(p)}^p \rangle$, we forget the implicit gurus and the top component, i.e.:
\[\varphi (p) := \langle c_{-1}^p,x_0^p,c_0^p,\ldots, x_{\ell(p)-1}^p,c_{\ell(p)-1}^p\rangle.\]

To see this works, let $p,q\in \mathbb{R}$ with $\varphi(p)=\varphi (q)$. By the definition of $\varphi$ we get that $\ell(p)=\ell(q)$ and for all $i<\ell(p)$ $x_i^p=x_i^q$ and $c_i^p=c_i^q$.

$\br$ For each $i<\ell(p)$, as
$I_i^{p},I_i^{q} \in \mathbf{G}({x_i})$, there are
$k^i_p,k^i_q<\kappa_{x_i}^{++}$ such that $I_i^{p}\s t_{x_i,k^i_p}$ and $I_i^{q}\s t_{x_i,k^i_q}$.
Let $l_i:=\max\{k^i_p,k^i_q\}$ and by the properties of the guru
there is a club $C_i$ in $\mathcal P_{\kappa_{x_i}}(x_i)$ such that $t_{x_i,l_i}(x)\leq t_{x_i,k^i_p}(x),t_{x_i,k^i_q}(x)$ for all $x\in C_i$.
Set $A_i:=\dom(I_i^p)\cap \dom(I_i^q)\cap C_i$ and let $\tilde{w}_i:=\langle x_i, \allowbreak t_{x_i,l_i}\restriction{ A_i}\rangle $. Note that $A_i \in \bigcap \vec{\mathbf{U}}(x_i)$.

$\br$
As $I_{\ell(p)}^{p},I_{\ell(q)}^{q} \in \mathcal{G}(\kappa)$ there are $l_p,l_q<\kappa^{++}$ such that $I_i^{p}\s t_{\kappa,l_p} $ and $I_i^{q} \s t_{\kappa,l_q} $, so we let $l:=\max\{l_p,l_q\}$.
By the properties of the guru there is a club $C$ in $\mathcal P_\kappa(\kappa^+)$ such that $t_{\kappa,l}(x)\leq t_{\kappa,l_p}(x),t_{\kappa,l_q}(x)$ for all $x\in C$.
Set $A_{\ell(p)}:=\dom(I_{\ell(p)}^p)\cap \dom(I_{\ell(q)}^q)\cap C$, and then let $\tilde{w}_{\ell(p)}:=\langle \kappa^+, t_{\kappa,l}\restriction{ A_{\ell(p)}}\rangle$.

Finally, let
\[
r:=\langle c_{-1}^p,\tilde{w}_0,c_0^p,\dots, \tilde{w}_{\ell (p) -1},c^p_{\ell(p)-1},\tilde{w}_{\ell(p)} \rangle
\]
Then $r\leq^{*}p,q$, as sought.
\end{proof}

\begin{definition}[Factorization]\label{factorizationdefn}
Given a condition $p\in\mathbb{R}$ and $i<\ell (p)$ with $o^{\vec{U}}(\kappa_{x_i^p})>0$,
letting $x:=x_i^p$, we factor $\mathbb{R}/p$ as $\mathbb{R}^{p,x}_l\times \mathbb{R}^{p,x}_u$ as follows. Each $q\leq p$ is viewed as a pair $(q_l, q_u)$ where
\begin{itemize}
\item $q_l= \langle c_{-1}^q,w_0^q,\ldots, c_{{i'}-1}^q,w_{i'}^q\rangle$, where
$i':=i_q(x)$, and
\item $q_u= \langle c_{{i'}}^q,w_{{i'}+1}^q,\ldots, c_{\ell(q)-1}^q,w_{\ell(q)}^q\rangle$.
\end{itemize}

Note that $\mathbb{R}^{p,x}_l$ is isomorphic to a cone of a natural variation of $\mathbb{R}$ that we denote by $\mathbb{R}_{\vec{U} \restriction (\kappa_x+1), \vec{\mathbf{t}} \restriction (\kappa_x+1)}$.
Specifically, we identify $q_l$ with $$\langle c_{-1}^q, w_0',c_0^q,\ldots,c_{j-1}^q,w_j'\rangle,$$ where
$w_k'=\langle \pi_x^{-1}[x_k^q],I_k^q\circ \pi_x\rangle$ for every $k\le j$.
We sometimes denote this collapsed version of $q_l$ by $\pi_x (q)$.
Also note that $\mathbb{R}^{p,x}_u$ is a regular subposet of $\mathbb{R}$ in which the first component of a condition is an element of $\col((\kappa_x)^{++},{<}\kappa)$,
so that the inherited direct extension is $(\kappa_x)^{++}$-closed.
We sometimes denote $q_u$ by $q\restriction\mathbb R_u^{p,x}$.
\end{definition}\label{lowerpart}
Several properties of the poset {$\mathbb{R}$} have reflected analogs for $\mathbb{R}_{\vec{U} \restriction (\kappa_x+1),\vec{\mathbf{t}} \restriction (\kappa_x+1)}$.
For example, we have seen that $\mathbb{R}_{\vec{U},\vec{\mathbf{t}}}$ has the $\kappa^{++}$-cc (Knaster), and likewise $\mathbb{R}_{\vec{U} \restriction (\kappa_x+1),\vec{\mathbf{t}} \restriction (\kappa_x+1)}$ has the $(\kappa_x)^{++}$-cc (Knaster). We will make use of this kind of analogous properties throughout the paper, especially when we determine the cardinal structure in various generic extensions.

\begin{proposition}\label{closure}
$(\mathbb{R}_u^{p,x},\leq^*)$ is $(\kappa_x)^{++}$-closed.
\end{proposition}
\begin{proof}
Let $\langle q_\beta \mid \beta<\gamma \rangle$ be a $\leq^*$-decreasing sequence of conditions in $\mathbb{R}_u^{p,x}$, with $\gamma<(\kappa_x)^{++}$.
We may assume $\ell(q_\beta)>0$ for each $\beta$, since the other case is simpler.
Write $q_\beta=\langle c_{-1,\beta},w_{0,\beta},c_{0,\beta},\ldots, w_{n-1,\beta},c_{n-1,\beta},w_{n,\beta} \rangle$
where $w_{i,\beta}=\langle x_i,I_{i,\beta} \rangle$ and $c_{-1,\beta} \in \col((\kappa_x)^{++},{<}\kappa_{x_0})$.
For each $i \in n \cup \{-1\}$, let $c_i^*:=\bigcup_{\beta<\gamma} c_{i,\beta}$.

$\br$ For each $i<n$, recall that $I_{i,\beta}\s t_{x_i,l_{i,\beta}}$ for some $l_{i,\beta}<\kappa_{x_i}^{++}$.
Since the club filter on $\mathcal{P}_{\kappa_{x_i}}(x_i)$ is $\kappa_{x_i}$-complete,
we may pick a club $C_i$ and a large enough $l_i^*$ such that for all $\beta<\gamma$ and $y \in C_i$, $t_{{x_i},l_i^*}(y) \leq t_{{x_i},l_{i,\beta}}(y)$. Let $A_i^*:=\bigcap_{\beta<\gamma}\dom(I_{i,\beta}) \cap C_i$ and $I_i^*:=t_{{x_i},l^*_i} \restriction A_i^*$.

$\br$ For $i=n$, we similarly pick $A^*_n \subseteq \bigcap_{\beta<\gamma}\dom(I_{n,\beta})$
and an implicit guru $I^*_n$ such that for all $\beta<\gamma$ and $y \in A_n^*$, $I_n^*(y) \leq I_{n,\beta}(y)$.

If necessary, shrink $A_i$ for $i<n$, and $A_n$ so that
$$q^*:=\langle c_{-1}^*, \langle x_0,I_0^* \rangle,c_0^*, \ldots,\langle x_{n-1},I_{n-1}^* \rangle, c_{n-1}^*,\langle \kappa^{+}, I_n^* \rangle \rangle$$
is a condition which is a $\leq^*$-lower bound of the sequence $\langle q_\beta \mid \beta<\gamma \rangle$.
\end{proof}
\begin{definition}[Tuple below]\label{tuplebelow}
For each $x \in \mathcal{P}_\kappa(\kappa^+)$, a \emph{tuple below $x$}
is a sequence $$t=\langle c_{-1},w_0,c_0, \ldots, w_{m-1},c_{m-1} \rangle$$ for which there exists a condition $q$ in $\mathbb{R}$ of the form $$\langle c_{-1},w_0,c_0,\ldots, w_{m-1},c_{m-1},\langle x,I \rangle, c_m ,\ldots, w_{n-1},c_{n-1}, w_n \rangle,$$
so that $q=t{}^\smallfrown\langle \langle x,\ldots\rangle,\ldots\rangle$.
\end{definition}
An easy calculation yields the following.

\begin{proposition}\label{countbelowx}
The collection of all tuples below $x$ has size at most $(\kappa_x)^+$.\qed
\end{proposition}

\begin{definition}[Trees] A \emph{tree} is a subset $T$ of ${}^{\le n}\mathcal H_{\kappa^+}$ for some $n<\omega$
such that every $s$ in $T$ is a finite nonempty sequence, all of whose nonempty initial segments are in $T$ as well.\footnote{Our trees have no root.}
The least such $n$ is denoted by $n(T)$.\footnote{In particular, $n(T)=0$ iff $T=\emptyset$.}
\begin{itemize}
\item For every $s\in T$, let $\hgt(s):=\dom(s)-1$;
\item For every $j< n$, set $\Lev_j(T):=\{s\in T\mid \hgt(s)=j\}$;
\item For every $s \in T$, let
$\Succ_T(s):=\{y\in\mathcal H_{\kappa^+}\mid s{}^\smallfrown\langle y\rangle\in T\}$;
\item A \emph{maximal node in $T$} is an element $s\in T$ with $\dom(s)=n(T)$.
\item $\rlm(T)=\bigcup\{ \im(s)\mid s\in T\}$.
\end{itemize}
\end{definition}
Note that $T\s{}^{\le n(T)}\rlm(T)$.

\begin{definition}[Side-by-side maximality] For any sequence $T_0,\ldots, T_k$ of trees,
we shall denote by $\mathbf S(T_0,\ldots,T_k)$ the collection
of all sequences $\vec{s}=\langle s_\iota \mid \iota\leq k\ \&\ T_\iota\neq\emptyset\rangle$ such that each $s_\iota$ is a maximal node in $T_\iota$.

\end{definition}
\begin{definition}[Fat trees] A \emph{fat tree} is either $\kappa$-fat tree
or an $x$-fat tree for some $x$,
where the two are defined below:
\begin{itemize}
\item
A \emph{$\kappa$-fat tree} is a tree $T$ such that:
\begin{enumerate}
\item every $s\in T$ is $\ssim$-increasing sequence of elements of $\mathcal P_\kappa(\kappa^+)$;
\item for some $i<o^{\vec U}(\kappa)$, $\{ y\in\mathcal P_\kappa(\kappa^+)\mid \langle y\rangle\in \Lev_0(T)\} \in U_{\kappa,i}$;
\item for every $y\in\Lev_{j}(T)$ that is not a maximal node,
for some $i_y<o^{\vec U}(\kappa)$, $\Succ_T(y) \in {U}_{\kappa,i_y}$.
\end{enumerate}
\item For $x \in \mathcal{P}_\kappa(\kappa^+)$,
an \emph{$x$-fat tree} is a tree $T$ such that:
\begin{enumerate}
\item every $s\in T$ is $\ssim$-increasing sequence of elements of $\mathcal P_{\kappa_x}(x)$;
\item for some $i<o^{\vec U}(\kappa_x)$,
$\{ y\in\mathcal P_{\kappa_x}(x)\mid \langle y\rangle\in \Lev_0(T)\} \in \mathbf U_{x,i}$;
\item for every $y\in\Lev_{j}(T)$ that is not a maximal node,
for some $i_y<o^{\vec U}(\kappa_x)$, $\Succ_T(y) \in \mathbf{U}_{x,i_y}$.
\end{enumerate}
\end{itemize}
\end{definition}

Let $T$ be a fat tree. We say that $T$ is \emph{compatible} with a condition $p$
iff for some $\iota\le\ell(p)$, $T\s{}^{<\omega}\dom(I^p_\iota)$.
For every node $s=\langle y_0,\ldots,y_m\rangle$ in $T$, a vector $\vec{B}=\langle B_0, \ldots, B_m\rangle$ is of \emph{$s$-measure-one}
iff for all $i\le m$, $B_i\in \bigcap\vec{\mathbf{U}}(y_i)$.
Note that in this case, $p+(s,\vec B)$ is meaningful \`a la Definition~\ref{def36}.
Furthermore, whenever $T_0,\ldots, T_k$ is a sequence of nonempty fat trees,
each $s_\iota$ is a maximal node in $T_\iota$, and $\vec{B}_\iota$ is of $s_\iota$-measure-one,
we may define $p+\langle \vec{s},\vec{\mathbf{B}} \rangle$ for $\vec{s}:=\langle s_0,\ldots,s_k\rangle$ and $\vec{\mathbf{B}}:=\langle \vec B_0,\ldots,\vec B_k\rangle$
in the obvious way. We also extend it to accept empty trees by simply ignoring these `ghost' coordinates.

\begin{lemma}\label{predense} Suppose:
\begin{itemize}
\item $p$ is a condition;
\item for every $\iota<\ell(p)$, $T_\iota$ is an $x_\iota^p$-fat tree;
\item $T_{\ell(p)}$ is a $\kappa$-fat tree;
\item for each $\vec{s}$ in $\mathbf S(T_0,\ldots,T_{\ell(p)})$, one attaches a sequence $\vec{\mathbf{B}}^{\vec{s}}=\langle \vec{B}_\iota^{\vec{s}} \mid \iota \leq \ell(p) \ \&\ T_\iota\neq\emptyset\rangle$ such that
each $\vec{B}^{\vec s}_\iota$ is of $s_\iota$-measure-one.
\end{itemize}

Then there are $p^* \leq^{**} p$ and for each $\iota\le\ell(p)$, there is a fat subtree $T_\iota^* \subseteq T_\iota$ with $n(T_\iota^*)=n(T_\iota)$ such that each $T_\iota^*$ is compatible with $p^*$, and the following set
$$\{p^*+\langle \vec{s},\vec{\mathbf{B}}^{\vec{s}} \rangle \mid \vec s\in\mathbf S(T^*_0,\ldots,T^*_{\ell(p)}) \}$$
is predense below $p^*$.
\end{lemma}

\begin{proof}
We only deal with the case $\ell(p)=0$. The general case is obtained by obvious recursion.
Thus, let $p$ be some condition of the form $\langle c_{-1}, \langle \kappa^{+}, I\rangle \rangle$ and let $T$ be a $\kappa$-fat tree. Denote $A:=\dom(I)$.
We induct on $n(T)$. If $n(T)=0$, then $T$ is empty and the lemma is vacuously true.
We assume $n(T)=1$ and $\{z\mid \langle z\rangle \in \Lev_0(T)\} \in U_{\kappa,i}$ for some fixed $i$.
Since $n(T)=1$, each $x \in A$ may be identified with the maximal node $\langle x\rangle$,
so let $B^x \in \bigcap\vec{\mathbf{U}}({x})$ be the associated $\langle x \rangle$-measure-one set. Consider
\begin{itemize}
\item $A_0:=\pi^{-1}_{j_{U_{\kappa,i}}``\kappa^+}(j_{U_{\kappa,i}}(x \mapsto B^x)(j_{U_{\kappa,i}}``\kappa^+))$,
\item $A_1:=\{x \in A \mid (A_0 \restriction x)=B^x\}\cap\rlm(T)$, and
\item $A_2:=\{x \in A \mid \exists k<o^{\vec{U}}(\kappa_x)\,[(A_1 \restriction x)\in \mathbf{U}_{x,k}]\}$.
\end{itemize}
Since $j_{U_{\kappa,i}}(\kappa)_{j_{U_{\kappa,i}}``\kappa^+}=j_{U_{\kappa,i}}(\kappa) \cap j_{U_{\kappa,i}}``\kappa^+=\kappa$ and $j_{U_{\kappa,i}}(o^{\vec{U}})(\kappa)=i$, we have that
$A_0\in \bigcap_{i'<i} U_{\kappa, i'}$. A straightforward calculation shows that $j_{U_{\kappa,i}}(A_0) \restriction j_{U_{\kappa,i}}``\kappa^+=\pi_{j_{U_{\kappa,i}[\kappa^+]}}A_0$ shows that
$A_1 \in U_{\kappa,i}$. Finally, since for $k>i$, $j_{U_{\kappa,k}}(A_1) \restriction j_{U_{\kappa,k}}``\kappa^+=\pi_{j_{U_{\kappa,k}[\kappa^+]}}``A_1 \in \mathbf{U}_{j_{\kappa,k}``\kappa,i}$, we have that $A_2 \in \bigcap_{k>i} U_{\kappa,k}$.
Take $A^*:=(A_0 \cup A_1 \cup A_2) \cap A\cap\rlm(T)$, $I^*:=I\restriction A^*$
and $T^*:={}^1A_1$.
We claim that $$p^*:=\langle c_{-1},\langle \kappa^{+},I^* \rangle \rangle$$
together with $T^*$ satisfy the conclusion of the lemma.
To this end, let $r \leq p^*$.

If $\ell(r)=0$, then $r$ is compatible with $p^*+\langle x,B^x \rangle$ for any $x$ with $\langle x\rangle \in T^*$,\footnote{To ease on the reader, we write $B^x$ instead of the more formally correct $B^{\langle x\rangle}$.} and we are done.
Otherwise, $r \leq^* p^*+\langle x_0, \ldots, x_n \rangle$ for some finite tuple $\langle x_0, \ldots, x_n \rangle$
of elements of $A^*$.
Now, if $\{ x_\iota\mid \iota\le n\}\s A_0$, then we may pick an $x \in A_1\cap\rlm(T)$ such that $x_\iota \ssim x$ for all $\iota\le n$. This means that for each $\iota\le n$, $x_\iota \in {A^* \restriction x} \supseteq {A_0 \restriction x}=B^x$.
It is then straightforward to show that $r$ is compatible with $p^*+\langle \langle x,B^x \rangle\rangle$.

Next, assume $\{ x_\iota\mid \iota\le n\}\nsubseteq A_0$,
and fix the least $\iota\leq n$ such that $x_\iota \in A_1 \cup A_2$.
If $x_{\iota} \in A_1$, then $x_\iota\in\rlm(T^*)$ and a similar analysis as the above shows that $r$ is compatible with $p^*+\langle x_{\iota},B^{x_{\iota}} \rangle$, so we are done.
Finally, suppose that $x_{\iota} \in A_2$ and write $w^r_{\iota}$ as $\langle x_{\iota},J \rangle$. Denote $B:=\dom(J)$.
Note that $B \cap (A_1 \restriction x_{\iota}) \in \mathbf{U}_{x_{\iota},k}$ for some $k<o^{\vec U}(\kappa_{x_\iota})$. Choose $x \in B \cap (A_1 \restriction x_{\iota})$.
Then $x_\iota\in\rlm(T^*)$ and a similar argument as before shows that $r$ is compatible with $p^*+ \langle x,B^x \rangle$. This completes the case where $n(T)=1$.

Now consider the case where $\hgt(T)=n+1$ for some $n>0$.
Let $i$ be such that $\{z\mid \langle z\rangle \in \Lev_0(T)\} \in U_{\kappa,i}$.
A typical element of $\mathbf S(T)$ has the form $s=\langle y_0,\ldots,y_n\rangle$,
and it is equipped with a sequence $\vec B^{\vec s}$ of sets of $s$-measure-one.
Recall that the $\iota^{\text{th}}$-element of $\vec B^{s}$ is in $\bigcap \mathbf U(y_\iota)$.
As there are fewer than $\kappa$-many sets in $\bigcap \mathbf U(y_\iota)$
and by the $\kappa$-completeness of $U_{\kappa,i}$,
we may prune $T$ to get a $\kappa$-fat subtree $T'\s T$ such that
the $0^{\text{th}}$-element of $\vec B^{s}$ for a sequence $s=\langle y_0,\ldots,y_n\rangle$ depends only on $y_0$. Thus, we denote it by $B_{y_0}$.

For a sequence $\vec{B}=\langle B_0, \ldots, B_n\rangle$, we denote $\drop(\vec{B})=\langle B_1,\ldots,B_n\rangle$.
Now, for each $x$ with $\langle x\rangle\in\Lev_0(T)$,
by the induction hypothesis, we may find a $p^x \leq^{**}p+ \langle x,B_x \rangle$ and a subtree $T^x\subseteq T\cap {}^{<\omega} x^\uparrow$ with $n(T^x)=n$ for which
$$\{p^x+\langle \vec s,\drop(\vec{\mathbf{B}}^{\langle x\rangle{}^\smallfrown\vec{s}})\rangle \mid \vec{s}\in\mathbf{S}(T^x)) \}$$
is predense below $p^x$.

As $p^x \leq^{**}p+\langle x,B_x \rangle$ , $p^x$ is of the form $\langle c_{-1},\langle x,I \restriction B_x' \rangle, I(x) , \langle \kappa^{+},I \restriction A_x \rangle \rangle$,
where $B'_x\s B_x$ and $A_x\s A$.
As before, we derive three sets:
\begin{itemize}
\item $A_0:=\pi_{j_{U_{\kappa,i}}``\kappa^+}(j_{U_{\kappa_i}}(x \mapsto B'_x)(j_{U_{\kappa,i}}``\kappa^+))$,
\item $A_1:=\{x \in A \mid (A_0 \restriction x)=B'_x\ \&\ \langle x\rangle \in \Lev_0(T)\} $, and
\item $A_2:=\{x \in A \mid \exists k<o^{\vec{U}}(\kappa_x)\,[(A_1 \restriction x)\in \mathbf{U}_{x,k}]\}$.
\end{itemize}
As before,
$A_0\in \bigcap_{i'<i} U_{\kappa, i'}$,
$A_1 \in U_{\kappa,i}$ and $A_2 \in \bigcap_{i'>i} U_{\kappa,i'}$.
Let $A^*:=(A_0 \cup A_1 \cup A_2) \cap A\cap\rlm(T)$, $I^*:=I \restriction A^*$, and $T^*$ be the unique tree to satisfy that $\Lev_0(T^*)={}^1A_1$, and $T^*\cap{}^{<\omega}x^\uparrow=T^x$ for each $x \in A_1$.
A similar analysis establishes that $$p^*:=\langle c_{-1},\langle \kappa^{+},I^* \rangle \rangle$$ together with $T^*$ are as sought.
\end{proof}
\subsection{Strong Prikry property} This subsection will be devoted to proving the following key theorem.
\begin{theorem}[Strong Prikry property]\label{strongprikry}

For every $p \in \mathbb{R}$ and every dense open set $D$ in $\mathbb{R}$,
there are $p^* \leq^* p$ and fat trees $T_0,\ldots, T_{\ell(p)-1},T_{\ell(p)}$ such that:
\begin{enumerate}[label=\textup{(\arabic*)}]
\item for every $i<\ell(p)$, $T_i$ is a (possibly empty) $x_i^p$-fat tree,
\item $T_{\ell(p)}$ is a (possibly empty) $\kappa$-fat tree,
\item for every $\vec{s}=\langle s_0,\ldots,s_{\ell(p)}\rangle$ in $\mathbf S(T_0,\ldots,T_{\ell(p)})$,
there are corresponding $\vec{B}_i$'s of $s_i$-measure-one such that
$$p^*+\langle \langle s_0,\vec{B}_0\rangle, \ldots, \langle s_{\ell(p)},\vec{B}_{\ell(p)} \rangle \rangle \in D.$$
\end{enumerate}
\end{theorem}

To motivate it, we point out the following consequence.

\begin{corollary}[Prikry property]\label{prikry}
Let $\varphi$ be a forcing statement of $\mathbb{R}$ and $p$ be a condition. There is $p^* \leq^* p$ such that $p^*$ decides $\varphi$, namely either $p^* \Vdash \varphi$ or $p^* \Vdash \neg \varphi$.
\end{corollary}
\begin{proof}
Consider $D:=\{r \in \mathbb{R} \mid r$ decides $\varphi\}$. Then $D$ is dense open. By the strong Prikry property that we are about to prove, let $p^* \leq^* p$ along with fat trees
$T_0,\ldots,T_{\ell(p)}$ be given by Theorem~\ref{strongprikry}. If necessary, we shrink fat trees so that all extensions using fat trees give the same decision. By a density argument and Lemma~\ref{predense}, it must be the case that $p^*$ decides $\varphi$.
\end{proof}

Hereafter, we prove Theorem~\ref{strongprikry}.
As before, we only deal with conditions $p$ with $\ell(p)=0$, and leave the general case to the reader.
\begin{lemma}\label{strongprikrylemma1}
Let $p=\langle c,\langle \kappa^{+}, I \rangle \rangle$ be a condition in $\mathbb R$, and $D$ be a dense open set in $\mathbb R$.
Then there are $p^*=\langle c, \langle \kappa^{+}, I^* \rangle \rangle \leq^* p$ and $\langle J_x\mid x \in \dom(I^*)\rangle$, such that for every $x \in \dom(I^*)$,
for every tuple $\langle \vec{c},\vec{u}\rangle$ below $x$ with $c_0 \leq c$,\footnote{Here, $c_0$ plays the role of the component of $\vec{c}$ as in Definition~\ref{tuplebelow}.}
the following holds:
\begin{itemize}
\item If there are $K_x,c_x,I_x$ such that $\langle \vec{c},\vec{u},\langle x,K_x \rangle, c_x, \langle \kappa^{+}, I_x \rangle \rangle$
extends $p^*$ and lies in $D$, then for some $B_x' \subseteq \dom (K_x)$, it is already the case that
$$\langle \vec{c},\vec{u},\langle x, J_x \restriction B_x' \rangle, I^*(x),\langle \kappa^{+},I^* \restriction x^\uparrow \rangle \rangle \in D.$$
\end{itemize}
\end{lemma}

\begin{proof}
Denote $A:= \dom (I)$. Let $x \in A$. Fix an enumeration
$\left\langle\langle \vec{c}^\beta,\vec{u}^\beta \rangle\mid{\beta<(\kappa_x)^+}\right\rangle$
of all tuples below $x$.
For each $\beta$, write $\langle \vec{c}^\beta,\vec{u}^\beta \rangle= \langle c_{-1}^\beta, w_0^\beta,c_0^\beta, \ldots, w_{m_\beta-1}^\beta,c_{m_\beta-1}^\beta \rangle$.
Write $D(x)$ for the collection of all $d\in \col((\kappa_x)^{++},{<}\kappa)$ such that
there is a pair $(J_x,I_x)$ for which one of the following holds:
\begin{enumerate}
\item $d \nparallel I(x)$;
\item for every $\beta<(\kappa_x)^+$, either
\begin{enumerate}
\item for all $J\in\mathbf G(x)$,
$$\langle \vec{c}^\beta,\vec{u}^\beta,\langle x,J\rangle, d,\langle \kappa^{+}, I_x \rangle \rangle \notin D,$$
or
\item for some $B_x^\beta\in\bigcap\vec{\mathbf{U}}(x)$,
$$\langle \vec{c}^\beta,\vec{u}^\beta,\langle x,J_x \restriction B_x^\beta \rangle, d,\langle \kappa^{+},I_x \rangle \rangle \in D.$$
\end{enumerate}
\end{enumerate}

\begin{claim}\label{strongprikryclaim1}
$D(x)$ is dense open in $\col((\kappa_x)^{++},{<}\kappa)$.
\end{claim}
\begin{proof}
Let $d \in \col((\kappa_x)^{++},{<}\kappa)$. If $d \nparallel I(x)$, then $d\in D(x)$ and we are done with any choice of $(J_x,I_x)$. Otherwise, by possibly extending $d$, we may assume that $d \leq I(x)$.
Next, we shall construct a sequence $\left\langle\langle d_\gamma,J_\gamma,I_\gamma \rangle\mid{\gamma<(\kappa_x)^+}\right\rangle$ such that:
\begin{itemize}
\item $\langle d_\gamma\mid \gamma<(\kappa_x)^+\rangle$ is a decreasing sequence of conditions below $d$;
\item $J_\gamma$ is either empty or it belongs to $\mathbf G(x)$ with $\dom(J_\gamma)\s A\restriction x$;
\item $I_\gamma \in\mathcal G(\kappa)$ with $\dom(I_\gamma)\s A$ ,
and $I_\gamma(y) \leq I(y)$ for every $y\in\dom(I_\gamma)$.
\end{itemize}
The construction is by recursion on $\gamma<(\kappa_x)^+$, as follows.

$\br$ For $\gamma=0$, we let $\langle d_0,J_0,I_0\rangle:=\langle d,\emptyset, I\restriction x^\uparrow \rangle$.

$\br$ For $\gamma=\beta+1$ such that $\langle d_\beta,J_\beta,I_\beta\rangle$ has already been defined, we do the following.

$\br\br$ If there are $J_\gamma, I_\gamma$ and $d_\gamma\leq d_{\beta}$ such that
$$\langle \vec{c}^\beta,\vec{u}^\beta,\langle x,J_\gamma \rangle, d_\gamma,\langle \kappa^{+}, I_\gamma\rangle \rangle$$ extends $p$ and lies in $D$,
then we keep them and form the triple $\langle d_\gamma,J_\gamma,I_\gamma\rangle$.

$\br\br$ Otherwise, we let $\langle d_\gamma,J_\gamma,I_\gamma\rangle:=\langle d_\beta,\emptyset,I_0\rangle$.

$\br$ For $\gamma$ a nonzero limit ordinal for which the sequence $\left\langle\langle d_\beta,J_\beta,I_\beta\rangle\mid \beta<\gamma\right\rangle$ has already been defined,
we simply let $d_\gamma$ be a lower bound of $\langle d_\beta \mid \beta<\gamma \rangle$, $J_\gamma:=\emptyset$, and $I_\gamma:=I_0$.

This completes the recursion.

Let $d^*$ be a lower bound for $\langle d_\gamma\mid \gamma<(\kappa_x)^+\rangle$.
We claim that $d^*$ is in $D(x)$.
By our setup, Case~(1) is not satisfied, thus, we need to cook up $J_x,A_x,I_x$ as in Case~(2).
For each $\gamma<(\kappa_x)^+$, there is $i_\gamma<\kappa^{++}$ such that $I_\gamma \subseteq t_{\kappa,i_\gamma}$.
Let $i:=\sup\{i_\gamma \mid \gamma<(\kappa_x)^+\}$. Since $\kappa_x<\kappa$, it is straightforward to verify that the following set is in $\bigcap\vec U(\kappa)$:
$$A_x:=\{y \mid \forall \gamma<(\kappa_x)^+\,[y\in \dom(I_\gamma)\ \&\ t_{\kappa,i}(y)\le I_\gamma(y)]\},$$
and then let $I_x:=t_{\kappa,i}\restriction A_x$. In particular, $I_x\in\mathcal G(\kappa)$.

To construct $J_x$, first for each $\gamma<(\kappa_x)^+$, fix $i_\gamma<(\kappa_x)^{++}$ witnessing that $J_\gamma\in\mathbf G(x)$ so that $J_\gamma \subseteq t_{x,i_\gamma}$. Let $i^*:=(\sup_{\gamma<(\kappa_x)^+} i_\gamma)+1$.
It follows that for each such $\gamma$, we may fix a club $C_\gamma$ in $\mathcal P_{\kappa_x}(x)$ such that for all $y \in C_\gamma$,
$t_{x,i^*}(y) \leq t_{x,i_\gamma}(y)$.

Finally, we take $J_x:=t_{x,i^*}$.
To verify that $J_x,A_x$ and $I_x$ are as sought, let $\beta<(\kappa_x)^+$.
Suppose that there is $J\in\mathbf G(x)$ such that
$$\langle \vec{c}^\beta,\vec{u}^\beta,\langle x,J\rangle, d^*,\langle \kappa^{+},I_x \rangle \rangle \in D,$$
and we will demonstrate the existence of $B_x^\beta\in\bigcap\vec{\mathbf{U}}(x)$ such that
$$\langle \vec{c}^\beta,\vec{u}^\beta,\langle x,J_x \restriction B_x^\beta \rangle, d^*,\langle \kappa^{+},I_x \rangle \rangle \in D.$$
By the definition of our recursion at step $\gamma:=\beta+1$, it is the case that
the triple $\langle d_\gamma,J_\gamma,I_\gamma\rangle$ was chosen to satisfy that
$$\langle \vec{c}^\beta,\vec{u}^\beta,\langle x,J_\gamma \rangle, d_\gamma,\langle \kappa^{+}, I_\gamma\rangle \rangle$$ extends $p$ and lies in $D$.
Then $B_x^\beta:=\dom(J_\gamma)\cap C_\gamma$ is as sought.
This completes the proof of Claim~\ref{strongprikryclaim1}.
\end{proof}

Fix $i_0<\kappa^{++}$ such that $I\s t_{\kappa,i_0}$.
As all the $D(x)$'s are dense, we may use the feature of the guru to find a large enough $i_1<\kappa^{++}$ such that for all $x\in A$, $t_{\kappa,i_1}(x) \in D(x)$.
For each $x\in A$, let $(J_x,I_x)$ be the witnessing pair for $t_{\kappa,i_1}(x) \in D(x)$,
and also let $i_x$ be such that $I_x\s t_{\kappa,i_x}$.
Now, let $i^*:=\sup\{ i_0,i_1,i_x\mid x\in A\}+1$.
Fix clubs $C_0,C_1$, and likewise for each $x$, fix clubs $C_x\s\dom(I_x)$,
such that for all $y \in \{0,1\} \cup \mathcal{P}_\kappa(\kappa^+)$ and $z \in C_y$, $t_{\kappa,i^*}(z) \leq t_{\kappa,i_y}(z)$.

Let $$C^*:=\{ z\in C_0 \cap C_1 \mid \forall x\,(x \ssim z \rightarrow z \in C_x)\},$$ and
$I^*:=t_{\kappa,i^*} \restriction A^*$, where $A^*:=A \cap C^*$.
Consider $p^*:=\langle c,\langle \kappa^{+}, I^* \rangle \rangle$. We claim that $p^* \leq^* p$ together with $\langle J_x\mid x \in A^*\rangle$ are as promised by Lemma~\ref{strongprikrylemma1}.
To this end, let $x \in A^*$,
let $\langle \vec{c},\vec{u}\rangle$ below $x$,
and assume that $K_x,c_x,I_x$ are such that
$\langle \vec{c},\vec{u}, \langle x,K_x \rangle,c_x,\langle \kappa^{+}, I_x \rangle \rangle$
extends $p^*$ and lies in $D$. Fix $\beta<(\kappa_x)^+$ such that $\langle \vec{c},\vec{u}\rangle=\langle\vec{c}^\beta,\vec{u}^\beta\rangle$.
Since $t_{\kappa,i_1}(x) \in D(x)$, the witnessing pair $(J_x,I_x)$ along with some $B^\beta_x$ satisfy
$$\langle \vec{c},\vec{u},\langle x,J_x \restriction B_x^\beta \rangle, t_{\kappa,i_1}(x),\langle \kappa^{+},I_x \rangle \rangle \in D.$$
Since $I^*(x) \leq t_{\kappa,i_1}(x)$, $A^*\cap x^\uparrow\s C_x\s \dom(I_x)$, and for each $z \in A^*\cap x^\uparrow$, $I^*(z) \leq I_x(z)$, and $D$ is open, we have that
$$\langle \vec{c},\vec{u},\langle x, J_x \restriction B_x^\beta \rangle,I^*(x),\langle \kappa^{+}, I^* \restriction x^\uparrow \rangle \rangle \in D,$$
as required.
\end{proof}
For the scope of this subsection, we introduce the following ad hoc concept.
\begin{definition}\label{Dn}
For a dense open set $D$ and an $n<\omega$, we say that a condition $p=\langle c, \langle \kappa^{+}, I \rangle \rangle$ satisfies $(D)_n$ iff for every $\langle\vec{c},\vec{u}\rangle$ with $x$ being the maximal working part of $\vec{u}$, $c_0 \leq c$ where $c_0$ is the first coordinate of $\vec{c}$, and additional sets $y_0 \ssim \cdots \ssim y_{n-1}$ with $x \ssim y_0$,
if there are $J_0, \ldots, J_{n-1}, e_0,\ldots, e_{n-1},I'$ such that
$$\langle \vec{c},\vec{u},\langle y_0,J_0 \rangle, e_0, \ldots, \langle y_{n-1},J_{n-1}\rangle, e_{n-1},\langle \kappa^{+}, I' \rangle \rangle$$
extends $p$ and lies in $D$,
then there are $p^*=\langle c^*, \langle \kappa^{+}, I^* \rangle \rangle \leq^* p$, and a fat tree $T$ of height $n$ such that for every $\vec s\in\mathbf S(T)$, there is a $\vec{B}$ of $s$-measure-one such that
\begin{align*}
\langle \vec{c},\vec{u},\langle \kappa^{+}, I^* \restriction x^\uparrow \rangle \rangle+\langle \vec{{s}},\vec{{B}} \rangle & \in D.
\end{align*}
\end{definition}

\begin{lemma}\label{strongprikrylemma2}
Let $p=\langle c,\langle \kappa^{+}, I \rangle \rangle$ be a condition in $\mathbb R$, and $D$ be a dense open set in $\mathbb R$.
Then there is a $p^*=\langle c, \langle \kappa^{+}, I^* \rangle \rangle \leq^* p$ witnessing as in $(D)_1$.
\end{lemma}

\begin{proof} Denote $A:=\dom(I)$.
By possibly passing to a direct extension of $p$, we may assume it satisfies the conclusion of Lemma~\ref{strongprikrylemma1} as witnessed by $\langle J_x\mid x \in A\rangle$.

For all $\langle\vec{c},\vec{u}\rangle$ and $x\in A$, if
\begin{itemize}
\item $\langle\vec{c},\vec{u}\rangle$ is a tuple below $x$, and
\item there exists a $B_x$ such that $\langle \vec{c},\vec{u},\langle x,J_x \restriction B_x \rangle, \langle \kappa^{+}, I \restriction x^\uparrow \rangle \rangle \in D$, then
\end{itemize}
we let $\varepsilon^{\vec{c},\vec{u}}(x):=1$. Otherwise,
set $\varepsilon^{\vec{c},\vec{u}}(x):=0$.

\begin{claim} Let $\iota<o^{\vec U}(\kappa)$. Then there exists an $\varepsilon_\iota<2$
such that $$X_\iota:=\{x \in A \mid\forall\langle\vec{c},\vec{u}\rangle\text{ a tuple below }x
\,(\varepsilon^{\vec{c},\vec{u}}(x)=\varepsilon_\iota)\}$$
is in $U_{\kappa,\iota}$.
\end{claim}
\begin{proof}
For each $\langle \vec c, \vec u\rangle$, fix some $A_{\vec c, \vec u}\in U_{\kappa,\iota}$ on which $x\mapsto \varepsilon^{\vec{c},\vec{u}}(x)$ is constant.
Let $$X := \{ x\in A \mid\text{ if } \langle \vec c,\vec u\rangle \text{ is a tuple below } x \text{ then } x\in A_{\vec c, \vec u}\},$$
so that $X \in U_{\kappa,\iota}$. Then by an application of the pigeonhole principle we find $X_\iota\s X$ in $U_{\kappa,\iota}$ as sought.
\end{proof}
Let $\langle X_\iota \in U_{\kappa,\iota}\mid \iota<o^{\vec U}(\kappa)\rangle$ be given by the claim.
Let $A^*:=\bigcup_{\iota<o^{\vec{U}}(\kappa)}X_\iota$ and $I^*:=I \restriction A^*$.
We claim that $p^*:=\langle c, \langle \kappa^{+}, I^* \rangle \rangle$ satisfies $(D)_1$.

Suppose $q \leq p^*$ with $q=\langle \vec{c},\vec{u},\langle x,J \rangle,e,\langle \kappa^{+},I' \rangle \rangle \in D$. Since $p$ satisfies Lemma~\ref{strongprikrylemma1}, there is $B_x$ such that
$$\langle \vec{c},\vec{u},\langle x, J_x \restriction B_x \rangle, I^*(x),\langle \kappa^{+}, I \restriction x^\uparrow \rangle \rangle \in D.$$
Fix the least $\iota$ such that $x \in X_\iota$. In particular, $\varepsilon_\iota=\varepsilon^{\vec{c},\vec{u}}(x)=1$.
Thus, for every $y \in X_\iota$ such that $\langle\vec c,\vec u\rangle$ is a tuple below $y$,
there is a $B_y$ such that
$$\langle \vec{c},\vec{u},\langle y,J_y \restriction B_y \rangle,I(y),\langle \kappa^{+}, I \restriction y^\uparrow \rangle \rangle \in D.$$
Now, to construct $p^{**}$, consider $J:=j_{\kappa,\iota}(y \mapsto J_y)(j_{\kappa,\iota}`` \kappa^+)$.
Then $A_0 :=\dom(J)$ is in $\bigcap_{\varsigma<\iota} U_{\kappa,\varsigma}$ and for some $i< \kappa^{++}$, $J=t_{\kappa, i} \restriction A_0$.
Then $A_1:=\{y \in \mathcal{P}_\kappa(\kappa^+) \mid J\restriction (\dom(J) \restriction y)=J_y\}$ is in $U_{\kappa,\iota}$.
Finally, let
$$A_2:=\{y \in \mathcal{P}_\kappa(\kappa^+) \mid \exists \varsigma<o^{\vec{U}}(\kappa_x)\,[A_1 \restriction y \in U_{x,\varsigma}]\}.$$
Similar as in the proof of Lemma~\ref{strongprikry}, one can show that $A_2 \in \bigcap_{\varsigma>\iota}U_{\kappa,\varsigma}$.
Now let $I^{**}$ be an implicit guru with domain $A^{**}:=A^* \cap (A_0 \cup A_1 \cup A_2)$
such that for every $y \in A^{**}$, $I^{**}(y)\leq J(y),I^*(y)$.
Finally, let $p^{**}:=\langle c, \langle \kappa^{+},,I^{**} \rangle \rangle$ and
$T:=A^* \cap A_1$. Then it is straightforward to check that $p^{**}$ and $T$ are as required: for each $y \in T$, the set $B_y$ will be the witness for the last clause of Definition~\ref{Dn}.
\end{proof}

\begin{lemma}\label{strongprikrylemma3}
Let $D$ be a dense open set in $\mathbb R$ and $p=\langle c, \langle \kappa^{+}, I \rangle \rangle$ be a condition that satisfies $(D)_n$ for a given positive integer $n$.

Then there are $p^*=\langle c, \langle \kappa^{+}, I^* \rangle \rangle \leq^* p$, and $\langle J_x\mid x \in \dom (I^*)\rangle$ such that for every $x \in \dom (I^*)$ and every $\langle\vec{c},\vec{u}\rangle$ a tuple below $x$, if there are $d',J'$, $y_0\ssim \cdots \ssim y_{n-1}$, $J_0,\ldots J_{n-1},e_0,\ldots,e_{n-1}$ and $I'$ such that
$$ \langle \vec{c},\vec{u},\langle x,J' \rangle, d', \langle y_0,J_0\rangle, e_0, \ldots, \langle y_{n-1},J_{n-1}\rangle, e_{n-1}, \langle \kappa^+, I'\rangle\rangle$$
extends $p$ and lies in $D$,
then there are $B_x \s \dom (J')$ and a fat tree $T$ of height $n$ compatible with $p^*$ such that for every maximal node $s$ of $T$, there is an $s$-measure-one $\vec{B}$ such that
$$\langle \vec{c},\vec{u},\langle x, J_x \restriction B_x \rangle, I^*(x), \langle \kappa^{+}, I^* \restriction x^\uparrow \rangle + \langle s,\vec{B} \rangle \in D.$$
\end{lemma}

\begin{proof}
Let $A:= \dom (I)$ and fix $x \in A$.
Let $\left\langle\langle \vec{c}^\beta,\vec{u}^\beta\rangle \mid {\beta<(\kappa_x)^+}\right\rangle$ be an enumeration of all the tuples below $x$.

Write $D(x)$ for the collection of all $d\in\col((\kappa_x)^{++},{<}\kappa)$ such that there is pair $(J_x,I_x)$ for which one of the following holds:
\begin{enumerate}
\item $d \nparallel I(x)$;
\item for every $\beta<(\kappa_x)^+$, either
\begin{enumerate}
\item for all $d',J',I',y_0,\ldots,y_{n-1},J_0,\ldots J_{n-1},e_0,\ldots,e_{n-1}$ such that:
\begin{itemize}
\item $d'\leq d$;
\item $J'\in \mathbf G (x)$ with $J'\leq J_x$;
\item $I'\in \mathcal G(\kappa)$ with $I'\leq I$;
\item $x\ssim y_0 \ssim \cdots \ssim y_{n-1}$;
\item for all $k<n$,
$y_k\in \dom (I_x)$, $e_k\leq I_x (y_k)$ and $J_k\in \mathbf{G}(y_k)$,
\end{itemize}
it is the case that
$$\langle \vec{c}^\beta,\vec{u}^\beta, \langle x,J' \rangle,d', \langle y_0,J_0 \rangle,e_0, \ldots, \langle y_{n-1},J_{n-1} \rangle,e_{n-1},\langle \kappa^{+},I' \rangle \rangle \notin D,$$
or
\item for some $B_x^\beta\in\bigcap\vec{\mathbf{U}}(x)$, there is an $x$-fat tree $T(x)$ of height $n$ such that for every $\vec{s}\in \mathbf{S}(T(x))$, there is an $\vec{s}$-measure-one $ \vec{{B}}$ such that
$$\langle \vec{c}^\beta,\vec{u}^\beta, \langle x,J_x \restriction B_{x}^\beta \rangle, d,\langle \kappa^{+}, I_x \rangle \rangle+\langle \vec{s},\vec{{B}} \rangle \in D.$$
\end{enumerate}
\end{enumerate}

\begin{claim}\label{strongprikryclaim2}
$D(x)$ is dense open.
\end{claim}

\begin{proof}

Let $d \in \col((\kappa_x)^{++},{<}\kappa)$. If $d \nparallel I(x)$, then we are done with any choice of $(J_x,I_x)$. Otherwise, by possibly extending $d$, we may assume that $d \leq I(x)$.
Next, we shall construct a sequence $\left\langle\langle d_\gamma,J_\gamma,I_\gamma, T_\gamma \rangle\mid{\gamma<(\kappa_x)^+}\right\rangle$ such that:
\begin{itemize}
\item $\langle d_\gamma\mid \gamma<(\kappa_x)^+\rangle$ is a decreasing sequence of conditions below $d$;
\item $J_\gamma$ is either empty or it belongs to $\mathbf G(x)$ with $\dom(J_\gamma)\s A\restriction x$;
\item $I_\gamma \in\mathcal G(\kappa)$ with $\dom(I_\gamma)\s A$,
and $I_\gamma(y) \leq I(y)$ for every $y\in\dom(I_\gamma)$;
\item $T_\gamma$ is an $x$-fat tree.
\end{itemize}
The construction is by recursion on $\gamma<(\kappa_x)^+$, as follows.

$\br$ For $\gamma=0$, we let $\langle d_0,J_0,I_0,T_0\rangle:=\langle d,\emptyset, I\restriction x^\uparrow, \emptyset \rangle$.

$\br$ For $\gamma=\beta+1$ such that $\langle d_\beta,J_\beta,I_\beta, T_\beta\rangle$ has already been defined, we do the following.

$\br\br$ If there are $J_\gamma$, $d_\gamma\leq d_{\beta}$, $y_0\ssim\cdots\ssim y_{n-1}$, $J_0,\ldots J_{n-1}$, $d_0,\ldots e_{n-1}$ and $I_\gamma$ such that
$$\langle \vec{c}^\beta,\vec{u}^\beta,\langle x,J_\gamma \rangle, d_\gamma, \langle y_0, J_0\rangle,\ldots,e_{n-1} ,\langle \kappa^{+}, I_\gamma\rangle \rangle$$
extends $p$ and lies in $D$,
then by $(D)_n$ applied to $\langle \vec{c}^\beta,\vec{u}^\beta,\langle x,J_\gamma \rangle, d_\gamma\rangle$, we may find a $p^*\leq^* p$ and a fat tree $T_\gamma$ of height $n$ such that for every $\vec s\in\mathbf S(T_\gamma)$, there is a $\vec{B}$ of $s$-measure-one such that
\begin{align*}
\langle \langle \vec{c}^\beta,\vec{u}^\beta,\langle x, J_\gamma\rangle, d_\gamma, \langle \kappa^{+}, I_\gamma \restriction x^\uparrow \rangle \rangle+ \langle \vec{{s}},\vec{{B}} \rangle & \in D,
\end{align*}
then we keep them and form the quadruple $\langle d_\gamma, J_\gamma, I_\gamma, T_\gamma \rangle$.

$\br\br$ Otherwise, let $\langle d_\gamma,J_\gamma,I_\gamma, T_\gamma\rangle:=\langle d_\beta,J_\beta,I_0, \emptyset\rangle$.

$\br$ For $\gamma$ a nonzero limit ordinal for which the sequence $\left\langle\langle d_\beta,J_\beta,I_\beta, T_\beta\rangle\mid \beta<\gamma\right\rangle$ has already been defined,
we simply let $d_\gamma$ be a lower bound of $\langle d_\beta \mid \beta<\gamma \rangle$, $J_\gamma:=\emptyset$, $I_\gamma:=I_0$ and $T_\gamma:=\emptyset$.
This completes the recursion.

Let $d^*$ be a lower bound for $\langle d_\gamma\mid \gamma<(\kappa_x)^+\rangle$.
We need to show that $d^*$ is in $D(x)$.
However, from this point on, the verification is very similar to the one from Claim~\ref{strongprikryclaim1}, and is left to the reader.
\end{proof}

The conclusion of Lemma~\ref{strongprikrylemma3} now follows from an application of Claim~\ref{strongprikryclaim2} in the same way the conclusion of Lemma~\ref{strongprikrylemma1} follows from Claim~\ref{strongprikryclaim1}.
\end{proof}

\begin{lemma}\label{strongprikrylemma4}
Let $D$ be a dense open set in $\mathbb R$ and $p=\langle c, \langle \kappa^{+}, I \rangle \rangle$ be a condition that satisfies $(D)_n$ for a given positive integer $n$. Then there is $p^*= \langle c, \langle \kappa^+, I^*\rangle \rangle \le^* p$ that satisfies $(D)_{n+1}$.
\end{lemma}
\begin{proof}
Denote $A:=\dom(I)$.
By possibly passing to a direct extension of $p$, we may assume it satisfies the conclusion of Lemma~\ref{strongprikrylemma3} as witnessed by $\langle J_x\mid x \in A\rangle$.

For all $\langle\vec{c},\vec{u}\rangle$ and $x\in A$, if
\begin{itemize}
\item $\langle\vec{c},\vec{u}\rangle$ is a tuple below $x$, and
\item there exist $B_x$ and $y_0\ssim\cdots\ssim y_{n-1}$ in $A\cap x^\uparrow$, $K_0, \ldots, K_{n-1}, e_0, \ldots, e_{n-1}$ and $I'$ such that
$$\langle \vec{c},\vec{u},\langle x,J_x \restriction B_x \rangle, I(x),\langle y_0,K_0\rangle, e_0,\ldots,e_{n-1}, \langle \kappa^{+}, I'\rangle \rangle$$
extends $p$ and lies in $D$, then
\end{itemize}
we let $\varepsilon^{\vec{c},\vec{u}}(x):=1$. Otherwise,
set $\varepsilon^{\vec{c},\vec{u}}(x):=0$.

The following is obvious.
\begin{claim} Let $\iota<o^{\vec U}(\kappa)$. Then there exists an $\varepsilon_\iota<2$
such that $$X_\iota:=\{x \in A \mid\forall\langle\vec{c},\vec{u}\rangle\text{ a tuple below }x
\,(\varepsilon^{\vec{c},\vec{u}}(x)=\varepsilon_\iota)\}$$
is in $U_{\kappa,\iota}$. \qed
\end{claim}

Let $\langle X_\iota \in U_{\kappa,\iota}\mid \iota<o^{\vec U}(\kappa)\rangle$ be given by the claim.
Let $A^*:=\bigcup_{\iota<o^{\vec{U}}(\kappa)}X_\iota$ and $I^*:=I \restriction A^*$.
We claim that $p^*:=\langle c, \langle \kappa^{+}, I^* \rangle \rangle$ satisfies $(D)_{n+1}$.
To this end, suppose $q \leq p^*$ with $$q=\langle \vec{c},\vec{u},\langle y_0,J_0 \rangle, e_0, \ldots, \langle y_{n},J_{n}\rangle, e_{n},\langle \kappa^{+}, I' \rangle \rangle \in D,$$
where $\langle\vec{c},\vec{u}\rangle$ is a tuple below $x$, and $x=y_0 \ssim \cdots \ssim y_{n}$.

Since $p^*\le p$ and the latter satisfies Lemma~\ref{strongprikrylemma3}, there are $B_x \s \dom (J_0)$ and a fat tree $T$ of height $n$ compatible with $p^*$ such that for every maximal node $\vec{{s}}$ of $T$, there is an $\vec{{s}}$-measure-one $\vec{B}$ such that
$$\langle \vec{c},\vec{u},\langle x, J_x \restriction B_x \rangle, I^*(x), \langle \kappa^{+}, I^* \restriction x^\uparrow \rangle\rangle + \langle \vec{{s}},\vec{B} \rangle \in D.$$

Fix the least $\iota$ such that $x \in X_\iota$. In particular, $\varepsilon_\iota=\varepsilon^{\vec{c},\vec{u}}(x)=1$.
Thus, for every $y \in X_\iota$ such that $\langle\vec c,\vec u\rangle$ is a tuple below $y$, there are $B_y$ and $T_y$ a fat tree of height $n$ such that
$$\langle \vec{c},\vec{u},\langle y,J_y \restriction B_y \rangle,I^*(y),\langle \kappa^{+}, I ^* \restriction y^\uparrow \rangle \rangle + \langle \vec{{s}},\vec{B} \rangle \in D.$$
Now, to construct $p^{**}$, consider $J:=j_{U_{\kappa,\iota}}(x \mapsto J_x)(j_{U_{\kappa,\iota}}`` \kappa^+)$.
Then $A_0 :=\dom(J)$ is in $\bigcap_{\varsigma<\iota} U_{\kappa,\varsigma}$ and for some $i< \kappa^{++}$, $J=t_{\kappa, i} \restriction A_0$.
In addition, $A_1:=\{x \in \mathcal{P}_\kappa(\kappa^+) \mid J\restriction (\dom(J) \restriction x)=J_x\}$ is in $U_{\kappa,\iota}$.
Finally, let
$$A_2:=\{x \in \mathcal{P}_\kappa(\kappa^+) \mid \exists \zeta<o^{\vec{U}}(\kappa_x)\,[A_1 \restriction x \in \mathbf{U}_{x,\zeta}]\},$$
and note that $A_2 \in \bigcap_{\varsigma>\iota}U_{\kappa,\varsigma}$.
Now let $I^{**}$ be an implicit guru with domain $A^{**}:=A^* \cap (A_0 \cup A_1 \cup A_2)$
such that for every $x \in A^{**}$, $I^{**}(x)\leq J(x),I^*(x)$.
Finally, let $p^{**}:=\langle c, \langle \kappa^{+},I^{**} \rangle \rangle$ and
$T:=A^* \cap A_1$. Then it is straightforward to check that $p^{**}$ and $T$ are as required.
\end{proof}

\begin{proof}[Proof of Theorem \ref{strongprikry}]
Let $p:=\langle c,\langle \kappa^+,I\rangle \rangle$ be a condition.
Define a $\leq^*$-decreasing sequence $\langle p_n \mid n<\omega \rangle$ by recursion, as follows.
Let $p_0:= p$ and then appeal to Lemma~\ref{strongprikrylemma2} to receive a $p_1:=\langle c,\langle \kappa^+, I_1\rangle \rangle \le^* p_0$ that satisfies $(D)_1$.
Next, for a positive integer $n$ such that $p_n=\langle c,\langle \kappa^+, I_n\rangle \rangle$ has already been defined and it satisfies $(D)_n$,
by Lemma~\ref{strongprikrylemma4}, let $p_{n+1}:=\langle c,\langle \kappa^+, I_1\rangle \rangle \le^* p_n$ be such that it satisfies $(D)_{n+1}$.
Finally, let $p^*$ be a $\leq^*$-lower bound for $\langle p_n \mid n<\omega \rangle$.
Now, since $D$ is dense, let $q \leq p^*$ be such that $q \in D$.
Consider $n:=\ell(q)$, and write $\stem (q)$ as $\langle x_0, \ldots ,x_{n-1}\rangle$.
Since $q\leq p^* \leq p_n$ with $p_n$ satisfying $(D)_n$, there is a $p_n^*\leq^* p_n$ and a $\kappa$-fat tree $T$ of height $n$ such that for
all $\vec s\in \mathbf{S}(T)$ and corresponding measure-one sets $\vec B$,
$$\langle c_{-1}^q,\langle \kappa^+, I^{p_n^*}\rangle\rangle +\langle \vec s, \vec{B}\rangle\in D.$$

Let $p^{**}\leq^* p_n^*$ be such that $c_{-1}^{p^{**}}\leq c_{-1}^q$, $T^{*}:= T\cap{}^{<\omega}\dom(I^{p^{**}})$ and for any $s\in \mathbf{S}(T^{*})$ and $\vec{{B}}$ is the associated $s$-measure-one, let $\vec{B}^*=\langle B_0 \cap \dom (I)^{p^{**}}, \ldots, B_{n-1} \cap \dom (I)^{p^{**}} \rangle$.
Then $p^{**}$ satisfies the strong Prikry property with the witness $T^*$ and the associated measure-one sets for each maximal node through $T^*$.
\end{proof}

\section{\texorpdfstring{The cardinal structure in $V^{\mathbb{R}}$}{The cardinal structure in the Radin extension}}\label{cardinalstructureradin}

We continue with our setup from Section~\ref{radinforcing}.
Fix a generic object $G$ for $\mathbb{R}$. Define
\begin{itemize}
\item $\mathbf X:=\{x \mid \exists p \in G\,(x \in \stem(p))\}$;
\item $K_0:=\{\kappa_x \mid x \in \mathbf X\}$;
\item $K_1:=\{((\kappa_x)^+)^V \mid x \in \mathbf X\}$;
\item $K_2:=\{((\kappa_x)^{++})^V \mid x \in \mathbf X\}$;
\item $\theta:=\otp(\mathbf X,\ssim)$.\footnote{Recall that $(\mathbf X,\ssim)$ is a well-ordering since $x\mapsto \kappa_x$ constitutes an injection over $\mathbf X$.}
\end{itemize}

Note that $K_0$ is cofinal in $\kappa$.
Indeed, given $\alpha<\kappa$ there are $p\in G$ and $x\in\stem(p)$ such that $\otp(x)>\alpha^+$, thus $x\in\mathbf X$, and $\kappa_x\in K_0$ with $\kappa_x>\alpha$.
Likewise, $\bigcup \mathbf X=(\kappa^+)^V$.
Let $\langle x_\tau \mid \tau<\theta\rangle$ be the $\ssim$-increasing enumeration of $\mathbf X$.
For every $\gamma\in\acc(\theta)$, by density, $x_\gamma=\bigcup_{\beta<\gamma} x_\gamma$.
This implies that $\{\sup(x) \mid x \in \mathbf X\}$ is also continuous and cofinal in $(\kappa^+)^V$. As a consequence, $K_0$ is closed below its supremum.
Furthermore, we write $\kappa_\tau:=\kappa_{x_\tau}$,
so that $K_0=\{\kappa_\tau\mid \tau<\theta\}$.
Maintaining continuity and for notational simplicity, we also write $x_\theta:=(\kappa^+)^V$ and $\kappa_\theta:=\kappa$.
Define
\begin{itemize}
\item $C_{-1}:=\{c_{-1}^p \mid p \in G\}$, and for every $\tau<\theta$,
\item $C_{\tau}:=\{c_k^p \mid p \in G, k<\ell(p) , w_k^p=\langle x_\tau,J \rangle \text{ for some }J \}$.
\end{itemize}

\begin{proposition}
$C_{-1}$ is a generic for $\col(\omega_1,{<}\kappa_0)$, and
likewise for each $\tau<\theta$, $C_\tau$ is a generic for $\col(\kappa_\tau^{++},{<}\kappa_{\tau+1})$.
\end{proposition}
\begin{proof} Let $D\s\col(\omega_1,{<}\kappa_0)$ be dense.
Clearly, $\bar D:=\{q\in\mathbb R\mid c_{-1}^q\in D\}$ is dense in $\mathbb R$, so that we may pick $q\in \bar D\cap G\neq \emptyset$.
Then $c_{-1}^q\in D\cap C_{-1}$.

Next, let $\tau<\theta$, and $D\s \col(\kappa_\tau^{++},{<}\kappa_{\tau+1})$ be dense.
First, find $p\in G$ such that $x_\tau=x_i^p$ for some $i<\ell(p)$.
Define $\bar D:=\{q\le p\mid \forall j<\ell(q)\,[(x_\tau=x_j^q)\rightarrow(c_{j}^q\in D)]\}$.
Clearly, $\bar D$ is dense in $\mathbb R$ below $p$,
so that we may pick $q\in \bar D\cap G$ extending $p$. Since $q\le p$, we may find a $j<\ell(q)$ with $x_\tau=x_j^q$.
Therefore, $c_j^q\in C_\tau\cap D$.
\end{proof}
Let $$\mathcal{C}:=\{C_{\tau} \mid \tau \in \{-1\} \cup \theta\}.$$
Due to the Lévy collapses, we have
\begin{proposition}\label{cardcollapse1}
In $V$, for a cardinal $\lambda$ with $\omega_1<\lambda<\kappa$, if $\lambda \notin K_0 \cup K_1 \cup K_2$, then $\lambda$ is collapsed in $V[G]$. \qed
\end{proposition}

A standard inductive argument shows the following proposition.
\begin{proposition} For $\alpha\le\kappa$ with $0<o^{\vec{U}}(\alpha)<\alpha$,
$\otp(K_0\cap\alpha)=\omega^{o^{\vec{U}}(\alpha)}$, where the last term concerns ordinal exponentiation. \qed
\end{proposition}

For $\tau\in\acc(\theta)$, we have $\bigcup_{\varsigma<\tau} \pi^{-1}_{x_\tau}[x_\varsigma]=\pi^{-1}_{x_\tau}[x_\tau]=(\kappa_\tau)^+$. We also recall that $\bigcup \mathbf X=(\kappa^+)^V$.
In addition, for every $\tau\in\acc(\theta)$, $\otp (x_{\kappa_\tau})=\kappa_\tau^+$ and $\bigcup \{x_\nu\mid \nu<\tau\}=x_{\kappa_\tau}$. That is:
\begin{proposition}\label{cardcollapse2}
$(\kappa^+)^V$ is collapsed. For every $\tau\in \acc (\theta)$, $(\kappa_\tau^+)^V$ is collapsed.\qed
\end{proposition}

\begin{proposition}\label{cardpreserve1}
All cardinals in $\{\omega_1\}\cup K_0 \cup K_2$ are preserved.
\end{proposition}

\begin{proof}
We first show that cardinals in $K_2$ are preserved.
Consider $\tau<\theta$ and suppose for a contradiction that $(\kappa_\tau^{++})^V$ is collapsed, as witnessed by a surjection $f:(\kappa_\tau^+)^V\rightarrow (\kappa_\tau^{++})^V$.
By the definition of $K_2$, there are $p \in G$ and $k<\ell(p)$ such that $w_k^p=\langle x_\tau,J \rangle$. Recalling Definition~\ref{factorizationdefn}, factor $\mathbb{R}/p$ as $\mathbb{R}^{p,x_\tau}_l \times \mathbb{R}_u^{p,x_\tau}$ and write $p$ as $(p_l,p_u)$.
As the right factor ordered by $\le^*$ is $\kappa_\tau^{++}$-closed, by the Prikry property,
we may obtain by recursion a $p_u' \leq^* p_u$ such that for every $\beta<(\kappa_\tau^+)^V$, there is a maximal antichain $X_\beta$ below $p_l$ such that for every $q \in X_\beta$, $(q,p_u')$ decides $\dot{f}(\beta)$.
This means that $(p_l,p_u')$ forces that $f$ lies in the extension by $\mathbb{R}^{p,x_\tau}_l$, contradicting the fact that the latter has the $(\kappa_\tau^{++})^V$-cc.

We leave it to the reader to verify that any $\kappa_\tau^+$, with $\tau \in \nacc(\theta)$, is preserved, using a similar factorization as in Proposition~\ref{cardpreserve1}. A similar argument (without any factorizations) shows that $\omega_1$ is preserved.
Thus, we are left with dealing with the cardinals from $K_0$.
To this end, let $\tau<\theta$, and we will show that $\kappa_\tau$ is not collapsed.

$\br$ If $\tau\in\acc(\theta)$, then $\kappa_\tau=\sup_{\beta<\tau} (\kappa_\beta^{++})^V$, and hence $\kappa_\tau$ is preserved.

$\br$ If $\tau=\bar\tau+1$, then we argue as follows.
Suppose towards a contradiction that $\kappa_\tau$ is collapsed, as witnessed by some map $f$. Find $p \in G$ and $k$ with $0<k+1<\ell(p)$ such that $w^p_{k+1}=\langle x_\tau,\emptyset \rangle$.
Write $w^p_k$ as $\langle x_{\bar\tau},J \rangle$.
As before, $\mathbb{R}/p$ is the product of two factors, and here the lower part can be factored as $\mathbb{R}_l^{p,x_{\bar\tau}} \times \col(\kappa_{\bar{\tau}}^{++},{<}\kappa_\tau)$. That is, $\mathbb R/p$ is isomorphic to
$$\mathbb{R}_l^{p,x_{\bar\tau}} \times \col(\kappa_{\bar{\tau}}^{++},{<}\kappa_\tau)\times \mathbb{R}_u^{p,x_\tau},$$
and we identify $p$ with $(p_l,p_m,p_u)$.
An argument as before yields $p_u'\le^* p_u$ such that $(p_l,p_m,p_u')$ forces that $f$ lies in
the extension by $\mathbb{R}_l^{p,x_{\bar\tau}} \times \col(\kappa_{\bar{\tau}}^{++},{<}\kappa_\tau)$. However, this is the product of $\kappa_\tau$-Knaster forcings, and hence is $\kappa_\tau$-cc, which yields a contradiction.

$\br$ If $\tau=0$, then we argue as follows.
Suppose towards a contradiction that $\kappa_0$ is collapsed, as witnessed by some map $f$. Find $p \in G$ with $\ell(p)>1$ such that $w^p_0=\langle x_0,\emptyset\rangle$ and $w^p_1=\langle x_1,\emptyset\rangle$,
so that $\mathbb R/p$ is isomorphic to
$$\col(\omega_1,{<}\kappa_0)\times\col(\kappa_0^{++},{<}\kappa_1)\times \mathbb R_u^{p,x_1},$$
so a similar analysis yields a contradiction.
\end{proof}

Putting the last findings together, yields the following.

\begin{proposition} For $\alpha\le\kappa$ with $0<o^{\vec{U}}(\alpha)<\alpha$,
$V[G]\models \cf(\alpha)=\cf(\omega^{o^{\vec{U}}(\alpha)})$. \qed
\end{proposition}

Recall that for each $\tau\le\theta$, $2^{\kappa_\tau}=(\kappa_\tau)^+$, $2^{(\kappa_\tau)^+}=(\kappa_\tau)^{++}$ and $2^{(\kappa_\tau)^{++}}=(\kappa_\tau)^{+3}$ in $V$.
We leave to the reader to calculate the cardinal arithmetic in $V[G]$ since the analysis is essentially the same as that from Proposition~\ref{cardpreserve1}.
For instance:
\begin{proposition}\label{prop46}
In $V[G]$, for every $\tau\in\theta$, $2^{\kappa_\tau}=(\kappa_\tau)^+$. \qed
\end{proposition}

\begin{proposition}
$\kappa$ is strongly inaccessible in $V[G]$.
\end{proposition}

\begin{proof}
By Proposition~\ref{prop46}, it suffices to prove that $\kappa$ remains regular in $V[G]$.
Towards a contradiction, suppose $V[G]$ admits a cardinal $\delta<\kappa$ and cofinal map from $\delta$ to $\kappa$.
Find $p=\langle c_0,\langle x, J\rangle, c_1,\langle\kappa^+,I\rangle\rangle\in G$ forcing it for which $\delta<\kappa_x$.
Factor $\mathbb{R}/p$ as $\mathbb{R}_l^{p,x} \times \mathbb{R}_u^{p,x}$.
As $\mathbb{R}_l^{p,x}$ has the $(\kappa_x)^{++}$-cc (see page \pageref{lowerpart}), it is in particular $\kappa$-cc,
so we may fix a $\mathbb{R}_u^{p,x}$-name $\dot{f}$ which is $\mathbb{R}_u^{p,x}$-forced by $p^x_u$ to be a cofinal map from $\delta$ to $\kappa$.
Denote $\mathbb R':= \mathbb R_u^{p,x}$.

\begin{claim} Let $\xi<\delta$. Then the following set is dense open:
$$D_\xi=\{q \in \mathbb{R'} \mid \exists i <\ell(q)\,(q \Vdash \dot{f}(\check\xi)< \kappa_{x_i^q})\}.$$
\end{claim}
\begin{proof} It is clearly open. To show it is dense, consider an arbitrary $q \in \mathbb{R}'$.
Extend $q$ to some $q'$ that decides the value of $\dot{f}(\xi)$ to be, say, $\xi_q$.
Pick $q'' \le q'$ such that $\xi_{q''}<\kappa_{x_i^{q''}}$ for some $i<\ell(q'')$. Then $q'' \in D_\xi$.
\end{proof}

Let $x\in \mathcal P_\kappa(\kappa^+)$. For each tuple $t$ below $x$,
let $p_t^x:=\langle t, \langle\kappa^+, I\restriction\dom(I)\uparrow^x\rangle\rangle$.
For each $\xi<\delta$, obtain $p_t^{x,\xi}\le^* p^x_t$ and $T_{\ell(p_t)}^{p_t, x,\xi}$ by appealing to the Strong Prikry property \ref{strongprikry} with $p_t^x$ and $D_\xi$.
For each non-top $s\in T_{\ell(p_t)}^{p_t, x,\xi}$, let $\alpha^{x,\xi}(s)<o^{\vec{\mathcal U}}(\kappa)$
be such that $\Succ_{T_{\ell(p_t)}^{p_t, x,\xi}}(s) \in U_{\kappa,\alpha^{x,\xi}(s)}$.
Next, let $$\beta^{x,\xi}(t):=\sup\{\alpha^{x,\xi}(s) \mid s \text{ is a non-top node of } T_{\ell(p_t)}^{p_t, x,\xi}\}.$$

Using $o^{\vec{U}}(\kappa)=\kappa^{+3}$, find $\gamma^{x,\xi}<o^{\vec{U}}(\kappa)$ above $\beta^{x,\xi}(t)$ for all tuples $t$ below $x$.
Finally, find $\gamma^*<o^{\vec{U}}(\kappa)$ above each $\gamma^{x,\xi}$ for all $x\in\mathcal P_\kappa(\kappa^+)$ and $\xi<\delta$.
Consider
$$B:=\{x \in \mathcal{P}_\kappa(\kappa^+) \mid \forall \xi<\delta\ \forall \text{tuple }t\text{ below } x \,\big(T_{\ell(p_t)}^{p_t,x,\xi} \cap \mathcal P_{\kappa_x}(x)\text{ is } x\text{-fat}\big)\}.$$
By a usual reflection argument, $B \in U_{\kappa,\gamma^*}$.
By Proposition~\ref{closure}, $\langle \mathbb R',\le^*\rangle$ is $\delta^+$-closed,
so we may recursively build a $\leq^*$-deceasing sequence $\langle p^*_{\xi}\mid \xi\le\delta\rangle$ of conditions in $\mathbb{R}'$,
where at stage $\xi$ we obtain $p^*_{\xi+1}$ by appealing to the Strong Prikry Property \ref{strongprikry} with $p_\xi^*$ and $D_\xi$.
Denote $p^*:=p^*_\delta$, and then continue as in the proof of the claim in \cite[Theorem~5.19]{MR2768695} (where our set $B$ plays the same role as the set $B$ in the reference),
inferring that $p^*$ forces that the range of $\dot{f}$ is in fact bounded in $\kappa$. This is a contradiction.
\end{proof}

\begin{proposition}\label{mediumcirc}
Let $\alpha\le\kappa$ with $o^{\vec{U}}(\alpha)\in\{\alpha,\alpha^+\}$.
Then $\cf(\alpha)=\omega$ in $V[G]$.
\end{proposition}

\begin{proof} $\br$ Suppose first that $o^{\vec{U}}(\alpha)=\alpha$.
Pick $p\in G$ with some $k<\ell(p)$ such that $w_k^p=\langle x_\alpha,I_k\rangle$ with $\kappa_{x_\alpha}=\alpha$.
As $o^{\vec{U}}(\alpha)=\alpha$, we see that $\{ x\in \mathcal{P}_\alpha(x_\alpha)\mid o^{\vec U}(\kappa_x)<\kappa_x\}\in\mathbf U(x_\alpha)$, so by density, we may assume that such $p$ satisfies $o^{\vec U}(\kappa_x)<\kappa_x$ for every $x$ in $A:=\dom(I_k)$. Next, we partition $A$ by letting $A_\iota:=\{x\in A\mid o^{\vec U}(\kappa_x)=\iota\}$ for each $\iota<\alpha$.
As $o^{\vec{U}}(\alpha)=\alpha$, it is the case that $\sup\{\iota<\alpha\mid A_\iota\neq\emptyset\}=\alpha$.
Consequently, we may recursively define a sequence $\langle \iota_n \mid n<\omega\rangle$ as follows:
\begin{itemize}
\item $\alpha_0:=\min\{\rank(\langle \vec{c},\vec{u}\rangle )\mid \langle \vec{c},\vec{u}\rangle \text{ is a tuple below }x_\alpha\}$,\footnote{Here, $\rank(z)$ stands for the least ordinal $\rho$ such that $z\in V_{\rho+1}$.} and
\item $\alpha_{n+1}:=\min(\{\kappa_z\mid z\in A_{\alpha_n}\}\cap \acc(K_0)\cap(\alpha_n,\alpha))$.
\end{itemize}

Note that there is a working part without referring to measures and $o^{\vec{U}}(\alpha)=\alpha$, so, there is $\langle \vec{c},\vec{u} \rangle$ with $\rank \langle \vec{c},\vec{u} \rangle<\alpha$. Thus, $\alpha_0<\alpha$. Next, we claim that $\nu:=\sup_{n<\omega}\alpha_n$ is equal to $\alpha$.
Indeed, otherwise we may pick $q \leq p$ with $q \in G$ such that, for some $i$, $w_i^q$ is of the form $\langle x,J \rangle$ with $\kappa_x=\nu$.
Pick the unique $\iota$ such that $x\in A_\iota$.
Hence, $\iota<\kappa_x=\nu$.
In addition, $B:=\dom(J)$ is a subset of $\bigcup_{\iota'<\iota}A_{\iota'}$ lying in $\mathbf U(x)$.
Consider $\varepsilon:=\min\{\kappa_y\mid y\in B\}$ which is less than $\kappa_x$.
Fix $n<\omega$ large enough that satisfies $\iota,\varepsilon<\alpha_n$ and such that for some $z\in B$ $\kappa_z=\alpha_{n+1}$ (it is possible since $\{\alpha_m \mid m<\omega\}$ is cofinal in $\nu$, hence $B\cap \left\{y\mid \kappa_y \in \{\alpha_m \mid m< \omega\}\right\}\neq \emptyset$).
Hence $z\in B\s \bigcup_{\iota'<\iota}A_{\iota'}$ which is disjoint from $A_{\alpha_n}$. Since $\kappa_z=\alpha_{n+1}$ which, by definition, is in $A_{\alpha_n}$, so this is a contradiction.

$\br$ Suppose now that $o^{\vec{U}}(\alpha)=\alpha^+$. Pick $p\in G$ with some $k<\ell(p)$ such that $w_k^p=\langle x_\alpha,I_k\rangle$.
We can disjointify $A:=\dom (I_k)$ as a union of $A_\iota$ (for $\iota<\alpha^+)$, where for every $x \in A_\iota$, $o^{\vec{U}}(\kappa_x)=\otp(x \cap \iota)$ and $\iota<\sup(x)$.
Define $\langle (y_n,\alpha_n, \varsigma_n)\mid {n<\omega}\rangle$ as follows.
\begin{itemize}
\item $\alpha_0:=\min\{\rank(\langle \vec{c},\vec{u}\rangle )\mid \langle \vec{c},\vec{u}\rangle \text{ is a tuple below }x_\alpha\}$, $y_0$ be any working part from $\mathbf X$ above $\langle\vec{c},\vec{u}\rangle$ with $\kappa_{y_0}=\alpha_0$ and $\varsigma_0=\sup(y_0)$.
\item Let $\alpha_{n+1}=\min\{ \gamma \in \acc(K_0) \setminus \alpha_n+1 \mid \exists x \in A_{\varsigma_n}(\kappa_x=\gamma)\}$ and $y_{n+1}$ be the working part with $\kappa_{y_{n+1}}=\alpha_{n+1}$, $\varsigma_{n+1}=\sup(y_{n+1})$.
\end{itemize}
We first show that $\nu:=\sup_n \varsigma_n=\alpha^+$.
Otherwise, by the closedness of $K_0$, we may pick $q \leq p$ with $q \in G$ such that, for some $i$, $w_i^q$ is of the form $\langle x,J \rangle$ with $\sup(x)=\nu$.
Pick the unique $\iota<\alpha^+$ such that $x \in A_\iota$.
Hence, $\iota<\sup (x) = \nu$.
In addition, $B:=\dom(J)$ is a subset of $\bigcup_{\iota'<\iota}A_{\iota'}$ lying in $\mathbf U(x)$.
Fix large enough $n<\omega $ be such that $\iota,\min(\{\kappa_z \mid z\in B\})<\varsigma_n$ and $y_{n+1} \in B$.
Thus, $y_{n+1} \in \bigcup_{\iota'<\iota} A_{\iota'}$, but $y_{n+1} \in A_{\varsigma_n}$ by definition and we get a contradiction.

This concludes that $V[G]\models\cf((\alpha^+)^V)=\omega$. Now in $V[G]$, we see that $\sup(x_\gamma) \mapsto \kappa_{x_\gamma}$ is a well-defined (because of the nature of the strong inclusion relation) cofinal map from a cofinal subset of $\alpha^+$ to $\alpha$. Hence, $\cf(\alpha)=\omega$ in $V[G]$.
\end{proof}

\section{Projecting from Radin to Prikry}\label{pradinforcing}

In this section, we give the explicit formulation of a poset where the generic includes partial information from the generic of the Radin forcing $\mathbb{R}$,
for which we can only read off $K_0$ but not $\mathbf{X}$ (defined in Section~\ref{cardinalstructureradin}).
Gurus play a vital role in showing that our forcing in this section is $\kappa^+$-cc, and hence, $\kappa^+$ is preserved.

Recall that for $s \in \vec{t}_\alpha$ and $x \in \mathcal{P}_\alpha(\alpha^+)$, we have a notion $s_x$ whose domain is $\mathcal{P}_{\alpha_x}(\alpha_x^+)$ and $s_x(y)=s(\pi_x[y])$ for all $y$. We also apply this notion to any $I \subseteq s$. Namely, if $I \subseteq t_{\alpha,i}$ for some $i$, define $I_x$ to be a function whose domain is $\pi_x^{-1}[\dom(I) \restriction x]$ and $I_x(y)=I(\pi_x[y])$ for all $y$.

We continue with our setup from Section~\ref{cardinalstructureradin}.
We consider a projected forcing of $\mathbb{R}_{\vec{U},\vec{\mathbf{t}}}$ which we call $\mathbb{P}_{\vec{U},\vec{\mathbf{t}}}$.

\begin{definition}
The forcing $\mathbb{P}_{\vec{U},\vec{\mathbf{t}}}$ consists of conditions of the form
$$p=\langle c_{-1},v_0,c_0, \ldots, v_{n-1},c_{n-1},v_n \rangle,$$
where

\begin{enumerate}
\item $n<\omega$;
\item for every $i\le n$:
\begin{enumerate}
\item if $0=i<n$:
\begin{itemize}
\item $v_0$ is a pair $\langle \kappa_0,I_0 \rangle$;
\item $\kappa_0<\kappa$;
\item if $o^{\vec{U}}(\kappa_{0})>0$, then $I_0 \in \mathcal{G}({\kappa_0})$;
otherwise, $I_0=\emptyset$;
\item $c_{-1} \in \col(\omega_1,{<}\min\{\kappa_{0},\kappa_x\mid x\in\dom(I_0)\})$;
\end{itemize}
\item if $0<i<n$:
\begin{itemize}
\item $v_i$ is a pair $\langle \kappa_i,I_i \rangle$;
\item $\kappa_i<\kappa$;
\item if $o^{\vec{U}}(\kappa_{i})>0$, then $I_i \in \mathcal{G}({\kappa_i})$;
otherwise, $I_i=\emptyset$;
\item $c_{i-1} \in \col((\kappa_{{i-1}})^{++},{<}\min\{\kappa_{i},\kappa_x\mid{x\in\dom(I_i)}\})$;
\end{itemize}
\item if $i=n$:
\begin{itemize}
\item $v_n = \langle \kappa_n, I_n \rangle$;
\item $\kappa_n=\kappa$;
\item $I_n \in \mathcal{G}({\kappa})$;
\item $c_{n-1} \in \col((\kappa_{{n-1}})^{++},{<}\min\{\kappa_x \mid x\in\dom(I_n)\})$;
\end{itemize}
\end{enumerate}
\item $\kappa_0 < \cdots < \kappa_{n-1}$.

\end{enumerate}
\end{definition}
\begin{remark}
We use similar notation to that we used for Radin forcing, namely we define $\stem(p)=\langle \kappa_0, \ldots, \kappa_{n-1} \rangle$, and other components in the same fashion: collapse part, the top part, etc'.
Unlike with $\mathbb R$, here a generic object will not give sufficient information to collapse $\kappa^+$
or successors of singular cardinals
below $\kappa^+$; for example, in $V[G]$ we have the full set of stems $\mathbf X$ whose union is $\kappa^+$, while if we derive stems from a generic of $\mathbb{P}_{\vec{U},\vec{\mathbf{t}}}$, we get only a set of ordinals below $\kappa$.
\end{remark}

Given $p,q \in \mathbb{P}_{\vec{U},\vec{\mathbf{t}}}$, say
\begin{itemize}
\item $p=\langle c_{-1}^p,v_0^p,c_0^p,\ldots, v_{\ell(p)-1}^p,c_{\ell(p)-1}^p,v_{\ell(p)}^p \rangle$, and
\item $q=\langle c_{-1}^q,v_0^q,c_0^q,\ldots, v_{\ell(q)-1}^q,c_{\ell(p)-1}^q,v_{\ell(q)}^q \rangle$,
\end{itemize}
define $p \leq q$ ($p$ is stronger than $q$) iff all of the following hold:
\begin{enumerate}
\item $\ell(p) \geq \ell(q)$;
\item $c_{-1}^p \leq c_{-1}^q$;
\item there are $0 \leq i_0 < i_1<\cdots <\ell(p)$ such that for each $k\le\ell(q)$,
\begin{itemize}
\item $\kappa_{i_k}^p=\kappa_k^q$, $\dom(I_{i_k}^p)\s\dom(I_k^q)$, and $c_{i_k}^p \leq c_k^q$;
\item if $i_{k-1}<i<i_k$ (where $i_{-1}=-1$), then there is an $x \in\dom(I_k^q)$ such that $\kappa_x=\kappa_i^p$, $c_i^p \leq I_k^q(x)$, and $I_i^p(y) \leq (I_k^q)_x(y)$ for every $y$.
\item for every $k \leq \ell(q)$, $\dom(I_{i_k}^p) \subseteq \dom(I_k^q)$ and $I_{i_k}^p(y) \leq I_k^q(y)$ for every $y$;
\end{itemize}
\end{enumerate}

We say that $p$ is a \emph{direct extension of} $q$, denote $p \leq^*q$, if $p\leq q$ and $\ell(p)=\ell(q)$.
\begin{definition}[$0$-step extension]
We say that $p$ is \emph{$0$-step extension of} $q$, denoted $p \leq^{**} q$ iff
\begin{enumerate}
\item $p \leq^* q$;
\item the collapse parts of $p$ and $q$ are equal;
\item for every $i\le\ell(p)$, for every $x \in\dom(I^p_i)$, $I_i^p(x)=I_i^q(x)$.
\end{enumerate}
\end{definition}

\begin{definition}[$1$-step extension]
Let $p$ be a condition.
For $i\le\ell(p)$ and $x \in \dom(I_i^p)$, the \emph{$1$-step extension of $p$ by $x$}, denote $p+\langle x \rangle$, is the condition $$q=\langle c_{-1}^p,v_0^p,c_0^p,\ldots,u',c',v',c_i^p,v_{i+1}^p,c_{i+1}^p,\ldots, v_{n-1}^p,c_{n-1}^p,v_n^p\rangle,$$ where

\begin{enumerate}
\item $u'=\langle \kappa_x, (I_i^p)_x\rangle$
\item $c'=I_i^p(x)$;
\item $v'=\langle \kappa_i^p, I_i^p \restriction x^\uparrow \rangle$.
\end{enumerate}
\end{definition}

Define $p+\langle x_0,\ldots, x_n \rangle$ by recursion, in the obvious way,
which we hereafter call an \emph{$n$-step extension of} $p$.
As in the previous section, $p \leq q$ iff $p$ is the direct extension of some $n$-step extension of $q$ for some $n<\omega$.

\begin{definition}[$1$-step extension while shrinking]\label{def36P}
Let $p$ be a condition.
For $x \in \dom(I_i^p)$ with $i\le\ell (p)$ and $B\in\bigcap\vec{{U}}(\kappa_i^p)$,
the \emph{$1$-step extension of $p$ by $\langle x,B\rangle$}, denoted $p+\langle x,B \rangle$,
is the condition $$q=\langle c_{-1}^p,u_0^p,c_0^p,\ldots,u',c',v',c_i^p,v_{i+1}^p,c_{i+1}^p,\ldots, v_{\ell (p)-1}^p,c_{\ell (p)-1}^p, v_{\ell(p)}^p \rangle,$$ where

\begin{enumerate}
\item $u'=\langle x, (I_i^p \restriction B)_x \rangle$;
\item $c'$ and $v'$ are as in the definition of $p+\langle x\rangle$.
\end{enumerate}
\end{definition}

We note that for a condition $p$ and $\alpha=\kappa_i^p$ for some $i$, then $$\langle c_{-1}^p,v_0^p,c_0^p, \ldots, v_{i-1}^p,c_{i-1}^p, v_i \rangle$$
is a condition in the natural variation $\mathbb{P}_{\vec{U} \restriction (\alpha+1),\vec{\mathbf{t}} \restriction (\alpha+1)}$.
Several properties which we will verify for $\mathbb{P}_{\vec{U},\vec{\mathbf{t}}}$ will have a reflected version for $\mathbb{P}_{\vec{U} \restriction (\alpha+1),\vec{\mathbf{t}} \restriction (\alpha+1)}$.
\begin{definition}[Translation map]\label{translation}
For $\langle x,I\rangle$ such that $x\in\mathcal P_\kappa(\kappa^+)$ and $I\in\mathbf G(x)$,
define $$\Proj(\langle x,I\rangle):=\langle\kappa_x,I_x\rangle.$$
\end{definition}
Throughout this section, unless stated otherwise, let $\mathbb{R}=\mathbb{R}_{\vec{U},\vec{\mathbf{t}}}$ and $\mathbb{P}=\mathbb{P}_{\vec{U},\vec{\mathbf{t}}}$
\begin{proposition}\label{projectionrp}
There is a projection $\Pi$ from $\mathbb{R}$ to $\mathbb{P}$.
\end{proposition}

\begin{proof}
Given $r=\langle c_{-1}, w_0,c_0, \ldots, w_{n-1},c_{n-1},w_n\rangle$ in $\mathbb{R}$,
define $$\Pi(r)= \langle c_{-1}, \Proj(w_0),c_0, \ldots, \Proj(w_{n-1}),c_{n-1},w_n\rangle.$$

We show that $\Pi: \mathbb{R} \to \mathbb{P}$ is a projection. From the definition of the translation, we see that if $r' \leq r$ in $\mathbb{R}$, then $\Pi(r') \leq \Pi(r)$ in $\mathbb{P}$.
Now, let $r \in \mathbb{R}$ and
$$p=\langle d_{-1},v_0,d_0, \ldots, v_{m-1},d_{m-1}, v_m\rangle$$
in $\mathbb{P}$ be an extension of $\Pi(r)$.
So $p \leq^* \Pi(r)+s$ for some $s=\langle y_0, \ldots, y_{n-1} \rangle$.
Define $\vec z=\langle z_0,\ldots,z_{n-1}\rangle$ as follows:
$$z_i:=\begin{cases}
\pi_{x_j^r}^{-1}[y_i],&\text{if }y_i\in\dom(I_j^{\Pi(r)})\text{ with }j<\ell(r);\\
y_i,&\text{otherwise}
\end{cases}$$

Enumerate $\stem(r) \cup\{z_0,\ldots,z_{n-1}\}$
$\ssim$-increasingly as $\{a_0,\ldots, a_{m-1} \}$.
Note that for every $i<m$, if $v_i=\langle \kappa_i,J_i \rangle$, then $\kappa_{a_i}=\kappa_i$, so we let $w_i':=\langle a_i, (y \mapsto J_i(\pi_{a_i}^{-1}[y])) \rangle$.
Finally, let $r'=\langle d_{-1},w_0',d_0,\ldots, w_{m-1}', d_{m-1},\langle \kappa^{+} , J_m^p \rangle \rangle$. We leave to the reader to verify that $r' \leq^* r+s$, so that $r' \leq r$, and $\Pi(r')= p$.
\end{proof}

\begin{proposition}\label{pradincc}
$\mathbb{P}$ has the $\kappa$-Linked$_0$ property. In particular, it is $\kappa^+$-Knaster and has the $\kappa^+$-cc.
\end{proposition}

\begin{proof}
We define a map $\psi:\mathbb{P}\rightarrow \mathcal H_\kappa$, as follows.
Given a condition $p = \langle c_{-1}^p, v_0^p, \allowbreak c_0^p, \ldots,\allowbreak v_{\ell(p)-1}^p, c_{\ell(p)-1}^p,w_{\ell(p)}^p \rangle$, $\psi(p)$ is obtained by removing the implicit gurus and the top component of $p$:
\[\psi (p) = \langle c_{-1}^p,\kappa_0^p,c_0^p,\ldots, \kappa_{\ell(p)-1}^p,c_{\ell(p)-1}^p\rangle.\]
To see this works, let $p,q\in \mathbb{P}$ with $\psi(p)=\psi (q)$. By the definition of $\psi$ we get $\ell(p)=\ell(q)$ and for all $i<\ell(p)$ $\kappa_i^p=\kappa_i^q$ and $c_i^p=c_i^q$.
For all $i\leq \ell(p)$, it is the case that $I_i^{p},I_i^{q} \in \mathcal{G}({\kappa_i})$ hence there are $l^i_p,l^i_q<\kappa_{i}^{++}$ such that $I_i^{p}\s t_{\kappa_{i},l^i_p}$ and $I_i^{q} \s t_{\kappa_{i},l^i_q} $.
Let $l_i:=\max\{l^i_p,l^i_q\}$ and by the properties of the guru $\vec{t}_{\kappa_{i}}$, there is a club $C_i$ such that $t_{\kappa_{i},l_i}(x)\leq t_{\kappa_{i},l^i_p}(x),t_{\kappa_{i},l^i_q}(x)$ for all $x\in C_i$.
Set $A_i=\dom(I_i^p)\cap\dom(I_i^q)\cap C_i$ and let $\tilde{v}_i=\langle \kappa_i, t_{\kappa_{i},l_i}\restriction{ A_i}\rangle $.
Consider
$$r :=\langle c_{-1}^p,\tilde{v}_0,c_0^p,\dots,c_{\ell(p)-1},\tilde{v}_{\ell(p)} \rangle.$$
Then $r\leq^{*}p,q$, as sought.
\end{proof}

\begin{definition}[$\mathbb{P}-$Factorization]\label{factorizationdefnP}
Given a condition $p\in\mathbb{P}$ and $i<\ell (p)$ with $o^{\vec{U}}(\kappa_{i}^p)>0$,
letting $\alpha:=\kappa_i^p$, we factor $\mathbb{P}/p$ as $\mathbb{P}^{p,\alpha}_l\times \mathbb{P}^{p,\alpha}_u$ as follows. Each $q\leq p$ is viewed as a pair $(q_l, q_u)$ where
\begin{itemize}
\item $q_l= \langle c_{-1}^q,v_0^q,\ldots, c_{{i'}-1}^q,v_{i'}^q\rangle$, where
$i':=i_q(\alpha)$, and
\item $q_u= \langle c_{{i'}}^q,v_{{i'}+1}^q,\ldots, c_{\ell(q)-1}^q,v_{\ell(q)}^q\rangle$.
\end{itemize}

Note that $\mathbb{P}^{p,\eta}_l$ is a cone of a natural variation of $\mathbb{P}$ that we denote by $\mathbb{P}_{\vec{U} \restriction (\eta+1), \vec{\mathbf{t}} \restriction (\eta+1)}$.
Also note that $\mathbb{P}^{p,\eta}_u$ is a regular subposet of $\mathbb{P}$ in which the first component of a condition is an element of $\col(\eta^{++},{<}\kappa)$,
so that it is $\eta^{++}$-closed.
\end{definition}
We have the following closure property.

\begin{proposition}\label{pradinclosed}
Let $p \in \mathbb{P}$ and $\alpha \in \stem(p)$. Then $(\mathbb{P}_u^{p,\alpha},\leq^*)$ is $\alpha^{++}$-closed.\qed
\end{proposition}
\begin{definition}[$\mathbb P$-Tuple below]

For each $\alpha\in\dom(o^{\vec U})\cap\kappa$, a \emph{$\mathbb{P}$-tuple below $\alpha$}
is a sequence $$t=\langle c_{-1},v_0,c_0, \ldots, v_{m-1},c_{m-1} \rangle$$ such that there exists a condition $q$ in in $\mathbb{P}$
of the form
$$\langle c_{-1},v_0,c_0,\ldots, v_{m-1},c_{m-1},\langle \alpha,J \rangle, c_m ,\ldots, v_{n-1},c_{n-1},v_n\rangle,$$
so that $q=t{}^\smallfrown\langle\langle\alpha,\ldots\rangle\ldots\rangle$.
\end{definition}
An easy calculation yields the following.

\begin{proposition}\label{pcountbelowx}
The collection of $\mathbb{P}$-tuples below $\alpha$ has size at most $\alpha$.\qed
\end{proposition}

\begin{definition}[Fat trees] For $\alpha<\kappa$, \emph{$\alpha$-fat tree} is a tree $T$ such that:
\begin{enumerate}
\item every $s\in T$ is $\ssim$-increasing sequence of elements of $\mathcal P_\alpha(\alpha^+)$;
\item for some $i<o^{\vec U}(\alpha)$, $\{ y\in\mathcal P_\alpha(\alpha^+)\mid \langle y\rangle\in \Lev_0(T)\} \in U_{\alpha,i}$;
\item for every $y\in\Lev_{j}(T)$ that is not a maximal node,
for some $i_y<o^{\vec U}(\alpha)$, $\Succ_T(y) \in {U}_{\alpha,i_y}$.
\end{enumerate}
\end{definition}

The following lemma is analogous to Lemma~\ref{predense} and is proved in a similar fashion.

\begin{lemma}\label{ppredense}
Suppose:
\begin{itemize}
\item $p$ is a $\mathbb{P}$-condition;
\item for every $\iota\le \ell(p)$, $T_\iota$ is a $\kappa_\iota^p$-fat tree;
\item for each $\vec{s}$ in $\mathbf S(T_0,\ldots,T_{\ell(p)})$, one attaches a sequence $\vec{\mathbf{B}}^{\vec{s}}=\langle \vec{B}_\iota^{\vec{s}} \mid \iota \leq \ell(p) \ \&\ T_\iota\neq\emptyset\rangle$ such that
each $\vec{B}^{\vec s}_\iota$ is of $s_\iota$-measure-one.
\end{itemize}

Then there are $p^* \leq^{**} p$ and for each $\iota\le\ell(p)$, there is a fat subtree $T_\iota^* \subseteq T_\iota$ with $n(T_\iota^*)=n(T_\iota)$ such that the following set
$$\{p^*+\langle \vec{s},\vec{\mathbf{B}}^{\vec{s}} \rangle \mid \vec s\in\mathbf S(T^*_0,\ldots,T^*_{\ell(p)}) \}$$
is predense below $p^*$. \qed
\end{lemma}

The following lemma is analogous to Theorem~\ref{strongprikry} and is proved in a similar fashion.

\begin{theorem}[Strong Prikry property]\label{pstrongprikry}
For every $p \in \mathbb{P}$ and every dense open set $D$ in $\mathbb{P}$,
there are $p^* \leq^* p$ and fat trees $T_0,\ldots, T_{\ell(p)-1},T_{\ell(p)}$ such that:
\begin{enumerate}[label=\textup{(\arabic*)}]
\item for every $i\leq \ell(p)$, $T_i$ is a (possibly empty) $\kappa_i^p$-fat tree, and
\item for every $\vec{s}=\langle s_0,\ldots,s_{\ell(p)}\rangle$ in $\mathbf S(T_0,\ldots,T_{\ell(p)})$,
there are corresponding $\vec{B}_i$'s of $s_i$-measure-one such that
$$p^*+\langle \langle s_0,\vec{B}_0\rangle, \ldots, \langle s_{\ell(p)},\vec{B}_{\ell(p)} \rangle \rangle \in D.\eqno\qed$$
\end{enumerate}
\end{theorem}

As a corollary, we get:
\begin{corollary}[Prikry property]\label{pprikry}
Let $\varphi$ be a forcing statement of $\mathbb{P}$ and $p$ be a condition. There is $p^* \leq^* p$ such that $p^*$ decides $\varphi$.\qed
\end{corollary}

We will turn to the analysis of the cardinal arithmetic in $V^{\mathbb{P}}$, where we note that all of the proofs are similar to the proofs of Section~\ref{cardinalstructureradin}. Let $H$ be $\mathbb{P}$-generic.
Define
\begin{itemize}
\item $\bar{K}_0:=\{\alpha \mid \exists p\in H (\alpha \in \stem(p))\}$;
\item $\bar{K}_1:=\{(\alpha^+)^V \mid \alpha \in \bar{K}_0\}$;
\item $\bar{K}_2:=\{(\alpha^{++})^V \mid \alpha \in \bar{K}_0\}$.
\end{itemize}

By a density argument, $\bar{K}_0$ is cofinal in $\kappa$ and $\bar{K}_0$ is closed below its supremum $\kappa$.
Let $\langle \kappa_\tau \mid \tau<\theta\rangle$ denote the increasing enumeration of $\bar{K}_0$.
Define
\begin{itemize}
\item $\bar{C}_{-1}:=\{c_{-1}^p \mid p \in H\}$, and for every $\tau<\theta$,
\item $\bar{C}_{\tau}:=\{c_k^p \mid p \in H, k<\ell(p) , v_k^p=\langle \kappa_\tau,J \rangle \}$.
\end{itemize}

Then $\bar{C}_{-1}$ is a generic for $\col(\omega_1,{<}\kappa_0)$, and likewise for each $\tau<\theta$, $\bar{C}_\tau$ is a generic for $\col(\kappa_\tau^{++},{<}\kappa_{\tau+1})$.
Let $$\bar{\mathcal{C}}:=\{\bar{C}_{\tau} \mid \tau \in \{-1\} \cup \theta\}.$$

\begin{theorem}\label{cardinalsinpradin}
In $V[H]$, all of the following hold:
\begin{enumerate}[label=\textup{(\arabic*)}]
\item $\omega_1$ is preserved;
\item all cardinals in $\bar{K}_0 \cup \bar{K}_1 \cup \bar{K}_2$ and cardinals above $\kappa$ are preserved;
\item for every $\alpha \in\bar{K}_0$ with $0<o^{\vec{U}}(\alpha)<\alpha$, $\cf(\alpha)=\cf(\omega^{o^{\vec{U}}(\alpha)})$;
\item for every $\alpha \in\bar{K}_0$ with $o^{\vec{U}}(\alpha)=\alpha$, $\cf(\alpha)=\omega$;
\item for every $\alpha \in\bar{K}_0$ with $\cf(o^{\vec{U}}(\alpha)) \geq \alpha^{++}$, $\alpha$ remains strongly inaccessible. In particular $\kappa$ remains strongly inaccessible;
\item for every $\alpha\in \bar{K}_0$, $2^\alpha=\alpha^{+}$.
\qed
\end{enumerate}
\end{theorem}
\begin{remark} Clause~(6) follows from factorization, the $(\kappa_\alpha)^+$-cc of $\mathbb P_{\alpha,l}$ and the Prikry property of that forcing notion.

We do not know whether a result analogous to Lemma~\ref{mediumcirc} holds here. Specifically, we speculate that if $o^{\vec{U}}(\alpha)=\alpha^+$, then $\alpha$ remains regular in $V^{\mathbb{P}}$.
\end{remark}

\section{Intermediate forcings}\label{intermediateforcing}
We continue with our setup from Section~\ref{pradinforcing}.
Since $\mathbb{R}$ projects to $\mathbb{P}$, let $G$ be $\mathbb{R}$-generic and $H$ the $\mathbb P$-generic set induced from $G$, hence $V \subseteq V[H] \subseteq V[G]$.
This section is devoted to analyzing various intermediate forcing notions whose generic extensions lie in between $V[H]$ and $V[G]$.

Let us consider the following forcing notion:
\begin{definition}
Let $x,y \in \mathcal{P}_\kappa(\kappa^+)$ with $y \ssim x$.
Define $\mathbb{Q}_{y,x}$ as the collection of $$q=\langle c_{-1},v_0,c_0, \ldots, v_{k-1},c_{k-1},w_k,c_k, \ldots, w_{m-1},c_{m-1},v_m,c_m, \ldots, v_{n-1},c_{n-1},v_n \rangle,$$
where
\begin{enumerate}
\item $\langle c_{-1},v_0,c_0, \ldots, v_{k-1},c_{k-1}, \langle \kappa_y ,I_k \rangle \rangle \in \mathbb{P}_{\vec{\mathbf{U}} \restriction \kappa_y+1,\vec{\mathbf{t}} \restriction \kappa_y+1}$;
\item $w_k=\langle y,I_k \rangle$;
\item for $k \leq i<m$, write $w_i=\langle z_i,I_i \rangle$. Then for some tuple $\langle\vec{d},\vec{u}\rangle$ below $y$,
$\langle \vec{d},\vec{u}, \langle \pi_x^{-1}[z_k],(I_k)_x\rangle, c_k, \ldots, \langle \pi_x^{-1}[z_{m-1}],(I_{m-1})_x \rangle,,c_{m-1},\langle \kappa_x^{+},I_m \rangle \rangle$ is in $\mathbb{R}_{\vec{\mathbf{U}} \restriction \kappa_x+1,\vec{\mathbf{t}} \restriction \kappa_x+1}$ (Recall the notion $I_x$ in the beginning of Section~\ref{pradinforcing});
\item $v_m=\langle \kappa_x^+,I_m \rangle$;
\item for some tuple $\vec{e},\vec{v}$ below $\kappa_x$, $\langle \vec{e},\vec{v}, v_m, c_m, \ldots, v_{n-1},c_{n-1},v_n \rangle \in \mathbb{P}$.
\end{enumerate}
\end{definition}
Similarly to $\mathbb{Q}_{y,x}$ we define
\begin{definition}
$\mathbb{Q}_{0,x}$ is the collection of all $$q=\langle c_{-1},w_0,c_0, \ldots, w_{m-1},c_{m-1},v_m,c_m, \ldots, v_{n-1},c_{n-1}, v_n \rangle,$$
where
\begin{enumerate}
\item $v_m=\langle \kappa_x,I_m \rangle$;
\item for $i<m$, write $w_{i}=\langle z_{i},I_{i} \rangle$. Then,
$$\langle c_{-1},\langle \pi_x^{-1}[z_0],(I_0)_x\rangle,c_0, \ldots,\langle \pi_x^{-1}[z_{m-1}],(I_{m-1})_x\rangle, c_{m-1}, v_m\rangle$$
is a condition in $\mathbb{R}_{\vec{\mathbf{U}} \restriction \kappa_x+1,\vec{\mathbf{t}} \restriction \kappa_x+1}$;
\item for some $\vec{e},\vec{v}$ below $\kappa_x$, $\langle \vec{e},\vec{v}, v_m, c_m, \ldots, v_{n-1},c_{n-1}, v_n \rangle \in \mathbb{P}$.
\end{enumerate}
\end{definition}

From now on, when we write $\mathbb{Q}_{y,x}$, we mean that $y \ssim x$ or $y=0$.
The order of $\mathbb{Q}_{y,x}$ is the one inherited naturally from $\mathbb{R}$ and $\mathbb{P}$ as defined in Section~\ref{radinforcing} and in Section~\ref{pradinforcing}.
Intuitively this forcing notion can be thought as forcing with $\mathbb{R}$ between a specific interval and with $\mathbb P$ outside this interval.

For $p \in \mathbb{R}$ and $x,y \in \stem(p)$, write
$$p=\langle c_{-1},w_0,c_0, \ldots, w_{k-1},c_{k-1},w_k,c_k, \ldots, w_{m-1},c_{m-1},w_m,c_m, \ldots, w_{n-1},c_{n-1},w_n \rangle,$$
with the $k^{\text{th}}$ working part being $y$ and the $m^{\text{th}}$ working part being $x$.
Then there is a natural projection from $p$ to the following corresponding condition in $\mathbb{Q}_{y,x}$:
\begin{align*}
\Pi_0^{y,x}(p)=&\langle c_{-1},\Proj(w_0),c_0, \ldots, \Proj(w_{k-1}),\\
&~c_{k-1},\Proj_x( w_k),c_k, \ldots, \Proj_x (w_{m-1}), c_{m-1},\Proj(w_m),\\
&~c_m, \ldots,\Proj(w_{n-1}),c_{n-1},\langle \kappa, I\rangle \rangle,
\end{align*}
where $\Proj$ is defined in Definition~\ref{translation} and $\Proj_x (\langle z, J\rangle):= \langle \pi_x^{-1}[z], J\circ \pi_x\rangle$.

Now, if
$$q=\langle c_{-1},v_0,c_0, \ldots, v_{k-1},c_{k-1},w_k,c_k, \ldots, w_{m-1},c_{m-1},v_m,c_m, \ldots, v_{n-1},c_{n-1}, v_n \rangle$$
is a condition in $\mathbb{Q}_{y,x}$, define
\begin{align*}
\Pi_1^{y,x}(q)=&\langle c_{-1},v_0,c_0, \ldots, v_{k-1},c_{k-1},\\
&~\Proj^{x}(w_k), c_k, \ldots, \Proj^{x}(w_{m-1}), \\
&~ c_{m-1},v_m, \ldots, v_{n-1},c_{n-1},v_n \rangle,
\end{align*}
where for each $i \in \{k, \ldots, m-1\}$, $w_i=\langle x_i^q,J_i^q \rangle$ and $\Proj^{x} (w_i)=\langle \kappa_{x_i^q},(J_i^q)_x \rangle$. Then $\Pi_1^{y,x}(q) \in \mathbb{P}$.

\begin{proposition}\label{projxy}
Let $\mathbb{R}(y,x):=\{p \in \mathbb{R} \mid x,y \in \stem(p) \ \& \ y\ssim x\}$.
Let $\Proj:\mathbb{R} \to \mathbb{P}$ be the projection as in Proposition~\ref{projectionrp}.
Then the maps $\Pi^{y,x}_0$ from $\mathbb{R}(y,x)$ to $\mathbb{Q}_{y,x}$,
and $\Pi_1^{y,x}$ from $\mathbb{Q}_{y,x}$ to $\mathbb{P}$ are projections such that $$\Pi_1^{y,x} \circ \Pi_0^{y,x}=\Proj \restriction \mathbb{R}(y,x).\eqno\qed$$
\end{proposition}

\begin{proposition}
In $V^{\mathbb{Q}_{y,x}}$,
$(\kappa_x)^+$ is collapsed and if $y\ssim x$, then $\kappa_y^+$ is preserved. \qed
\end{proposition}

Let us consider another forcing notion:
\begin{definition}
Let $x \in \mathcal{P}_\kappa(\kappa^+)$.
Define $\mathbb{Q}^x$ as the collection of all
$$q=\langle c_{-1},v_0,c_0, \ldots, v_{m-1},c_{m-1},w_m,c_m, \ldots, w_{n-1},c_{n-1},w_n \rangle,$$
where
\begin{enumerate}

\item $w_m=\langle x,I_m \rangle$;
\item $\langle c_{-1},v_0,c_0, \ldots, v_{m-1},c_{m-1}, \langle \kappa_x,(I_m)_x \rangle \rangle \in \mathbb{P}_{\vec{\mathbf{U}} \restriction \kappa_x+1,\vec{\mathbf{t}} \restriction \kappa_x+1}$;
\item for some tuple $\langle\vec{d},\vec{u}\rangle$ below $x$,
$\langle \vec{d},\vec{u}, c_{m-1}, w_{m}, \ldots, c_{n-1},w_n \rangle \in \mathbb{R}$.
\end{enumerate}
\end{definition}
The order of $\mathbb{Q}^x$ is inherited from the order of $\mathbb R$ and $\mathbb P$ as defined in Section~\ref{radinforcing} and in Section~\ref{pradinforcing}.
One can intuitively think of this forcing notion as forcing with $\mathbb P$ below some cardinal and with $\mathbb R$ above it.

Suppose that $p \in \mathbb{R}$ with $x \in \stem(p)$.
Then there is a projection from $p$ to a particular condition $\Pi_0^x(p)$ in $\mathbb{Q}^x$, namely,
$$\Pi_0^x(p):= \langle c_{-1},\Proj(w_0), c_0 ,\ldots, \Proj(w_{m-1}),c_{m-1},w_m,c_m, \ldots, w_{n-1},c_{n-1}, w_n\rangle,$$
where $\Proj$ is defined as in Proposition~\ref{projectionrp}.

For a condition
$$q=\langle c_{-1},v_0,c_0, \ldots, v_{m-1},c_{m-1},w_m,c_m, \ldots, w_{n-1},c_{n-1},w_n \rangle$$
in $\mathbb{Q}^x$, define $\Pi_1^x(q)$ as
\begin{align*}
\langle c_{-1},v_0,c_0, \ldots, v_{m-1},c_{m-1},\Proj(w_m),c_m, \ldots, \Proj(w_{n-1}),c_{n-1},\langle \kappa, I_n\rangle\rangle,
\end{align*}
and note that it is a condition in $\mathbb{P}$.

\begin{proposition}\label{projx}
Let $\mathbb{R}^x=\{p \in \mathbb{R} \mid x \in \stem(p)\}$ and $\Proj: \mathbb{R} \to \mathbb{P}$ be the projection as in Proposition~\ref{projectionrp}.
Then the maps $\Pi_0^x$ from $\mathbb{R}_x$ to $\mathbb{Q}^x$ and $\Pi_1^x$ from $\mathbb{Q}^x$ to $\mathbb{P}$ are projections,
and $\Pi_1^x \circ \Pi_0^x=\Proj \restriction \mathbb{R}^x$. \qed
\end{proposition}

The following proposition follows from the inherited strong Prikry property, factorization, and the $(\kappa_x)^+$-cc of $\mathbb{Q}^x$.

\begin{proposition}
In $V^{\mathbb{Q}^x}$, $(\kappa_x)^+$ is preserved.\qed
\end{proposition}

\section{Weak homogeneity}\label{weakhomogeneity}
We continue with our setup from Section~\ref{intermediateforcing}.

Let us define:
$$\Aut_\kappa(\kappa^+)=\{\Gamma\in{}^{\kappa^+}\kappa^+ \mid \Gamma\text{ is a bijection } \&\ \exists \gamma \in (\kappa,\kappa^+) \forall \xi \in [0,\kappa] \cup (\gamma,\kappa^+) (\Gamma(\xi)=\xi)\}.$$
For $x \in \mathcal{P}_\kappa(\kappa^+)$, define:
\begin{align*}
\Aut_{\kappa_x}(x) =\{\Gamma\in{}^xx\mid &\ \Gamma\text{ is a bijection}\ \&\ \\
&\ \exists \gamma \in (\kappa_x,\sup(x)) \forall \xi \in [0,\kappa_x) \cup \{\kappa\} \cup (\gamma, \sup(x))\,(\Gamma(\xi)=\xi)\}.
\end{align*}

We lift those automorphisms to automorphisms of $\mathbb{R}, \mathbb{Q}_{y,x},$ and $\mathbb{Q}^x$ as follows.
Let $\Gamma \in \Aut_\kappa(\kappa^+)$.
\begin{enumerate}
\item For $x \in \mathcal{P}_\kappa(\kappa^+)$, let $\Gamma(x):=\Gamma[x]$, so that $\kappa_x=\kappa_{\Gamma(x)}$.
\item For $A \in \bigcap \vec{U}(\kappa)$, let $\Gamma(A):=\{\Gamma[a]\mid a\in A\}$. Note that as in \cite[Lemma~3.4]{Mag77}, $\Gamma(A) \in \bigcap \vec{U}(\kappa)$.
\item $\{y \in \mathcal{P}_{\kappa}(\kappa^+) \mid \Gamma(y)=y\} \in \bigcap \vec{U}(\kappa)$.
\item For $I$ a function with $\dom(I)=A \in \vec{U}(\kappa)$, let $\Gamma(I)$ be the function with domain $\Gamma(A)$ and $\Gamma(I)(\Gamma(y))=I(y)$ for all $y$.
\item For $c\in \col(\eta,\kappa)$, let $\Gamma(c):=c$.
\end{enumerate}

\begin{lemma}\label{preservecollapse}
For every $I \in \mathcal{G} (\kappa)$, $\Gamma(I) \in \mathcal{G} (\kappa)$.
\end{lemma}

\begin{proof}
Let $A:=\dom(I)$ and note that $\Gamma(A) \in \bigcap \vec{U}(\kappa)$.
Fix $B\subseteq A$ with $B \in \bigcap \vec{U}(\kappa)$, and fix $\gamma$ such that $I=t_{\kappa,\gamma} \restriction B$.
Since $\{y \in \mathcal{P}_\kappa(\kappa^+) \mid \Gamma(y)=y\} \in \bigcap \vec{U}(\kappa)$,
and for each $y \in B$ with $\Gamma(y)=y$, $\Gamma(I)(y)=\Gamma(I)(\Gamma(y))=I(y)$,
it is the case that $\Gamma(I)$ agrees with $t_{\kappa,\gamma}$ on a measure-one set.
Therefore $\Gamma(I) \in \mathcal{G}(\kappa)$.
\end{proof}

\begin{lemma}
For every $A \in \bigcap \mathbf{U}(x)$, $\Gamma(A) \in \bigcap \mathbf{U}{(\Gamma(x))}$.
\end{lemma}

\begin{proof}
Let $A \in \bigcap\mathbf{U}(x)$.
Then $$\{\pi_x^{-1}[a]\mid a\in A\} \in \bigcap \vec{U}(\kappa_x).$$
Fix $i<o^{\vec{U}}(\kappa_x)$, and let $j=j_{U_{\kappa_x,i}}$.
Then, $j``x \in j(A)$.
Since $j(\Gamma)(j``x)=j``(\Gamma(x))$ and $j(\Gamma)(j(A))=j(\Gamma(A))$, we have $j``(\Gamma(x)) \in j(\Gamma(A))$.
Hence, $\Gamma(A) \in \bigcap \mathbf{U}{(\Gamma(x))}$.
\end{proof}

\begin{lemma}\label{premutation preserveing implicitness}
For every $I \in \mathbf{G}(x)$, $\Gamma(I) \in \mathbf{G}({\Gamma(x)})$.
\end{lemma}

\begin{proof}
Let $A:=\dom(I)$, hence $I=t_{x,i} \restriction B$, for some $B \subseteq A$ such that $B \in \bigcap \vec{\mathbf{U}}(x)$.
Since $\{ y\in \mathcal{P}_{\kappa_x}((\kappa_x)^+)\mid \pi^{-1}_{\Gamma(x)}\left[\Gamma[\pi_x[y]]\right]=y\} \in \bigcap \vec{U}(\kappa_x)$
and $\Gamma(B)\s \dom (\Gamma(I))$ we get that $\tilde{B}:=\{ y\in \pi_x^{-1}[B]\mid \pi^{-1}_{_{\Gamma(x)}}\left[\Gamma[\pi_x[y]]\right]=y\} \in \bigcap \vec{U}(\kappa_x)$.
Let $D:=\{\pi_{\Gamma(x)} [z]\mid z\in \tilde{B}\}$ which is an element of $\bigcap \mathbf{\vec{U}}(\Gamma(x))$.
Hence $B\s \Gamma(\tilde B)\s \Gamma(A)$ and we get that
$\Gamma(I)\restriction D= t_{\Gamma(x),i}\restriction \Gamma[\tilde B]\in \mathbf{G}(\Gamma(x))$.
\end{proof}

Fix some $\Gamma'\in \Aut_{\kappa_x}(x)$.
Note that $\Gamma:=\Gamma'\cup\id_{\kappa^+ \setminus x}$ belongs to $\Aut_\kappa(\kappa^+)$.
Likewise, $\Gamma'$ can be lifted to an automorphism of $\Aut(\mathbb{Q}_{y,x})$ or of $\Aut(\mathbb{Q}^x)$ in a similar manner to how $\Gamma$ is lifted.

Recall that if $G$ is $\mathbb{R}$-generic, then we can derive the projected Prikry sequence $K_0$ and the collection of collapses $\mathcal{C}$ from the generic $G$.
Let $\dot{K}_0$ and $\dot{\mathcal{C}}$ be their canonical names.\footnote{We can take canonical names for $K_0$ and $\mathcal{C}$ similarly to taking a canonical name to a Prikry sequence for vanilla Prikry Forcing.}
Then each $\Gamma \in \Aut_\kappa(\kappa^+)$ can be lifted to an automorphism of $\mathbb{R}$, and since $\Gamma$ fixes all ordinals below $\kappa$, $\Gamma(\dot{K}_0)=\dot{K}_0$ and $\Gamma(\dot{\mathcal{C}})=\dot{\mathcal{C}}$.

\begin{proposition}\label{hom}
Let $p,q \in \mathbb{R}$. Suppose that $\ell(p)=\ell(q)=n$, for every $i<n$, $\kappa_{x_i^p}=\kappa_{x_i^q}$, and for every $j \in n \cup \{-1\}$, $c_j^p \parallel c_j^q$.
Then there is $\Gamma \in \Aut_\kappa(\kappa^+)$ such that $\Gamma(p) \parallel q$.
\end{proposition}

\begin{proof}
List $\stem(p)$ and $\stem(q)$ $\ssim$-increasingly as $\{x_0,\ldots, x_{n-1}\}$ and $\{y_0, \ldots, y_{n-1}\}$, respectively.
We build $\Gamma$ block by block.
We begin with $\Gamma_{-1}=\id_\kappa$.
Since $\otp(x_0)=\kappa_{x_0}^+=\kappa_{y_0}^+=\otp(y_0)$, we extend $\Gamma_{-1}$ to $\Gamma_0$ which is a bijection from $\kappa \cup x_0$ to $\kappa \cup y_0$.
Note that $\im(\Gamma_0 \restriction x_0) \subseteq y_1$, and $\kappa_{x_1}=\kappa_{y_1}$, so we extend $\Gamma_0$ to $\Gamma_1$ which is a bijection from $\kappa \cup x_1$ to $\kappa \cup y_1$.
Continue this process, we obtain partial functions $\Gamma_0 \subseteq \Gamma_1 \subseteq \dots \subseteq \Gamma_{n-1}$ such that for each $i$, $\Gamma_i$ is a bijection from $\kappa \cup x_i$ to $\kappa \cup y_i$. In particular,
$\Gamma_{n-1}$ is a bijection from $\kappa \cup x_{n-1}$ to $\kappa \cup y_{n-1}$.
Since $\kappa_{n-1}<\kappa<\kappa^+$, $\dom(\Gamma_{n-1}), \im(\Gamma_{n-1})$ is bounded in $\kappa^+$. Extend $\Gamma_{n-1}$ to $\Gamma \in \Aut_\kappa(\kappa^+)$.
Thus, $\Gamma(x_i)=y_i$ for all $i$.
We now show that $\Gamma(p) \parallel q$.
For each $i \in n\cup\{-1\}$, let $c_i^*:=c_i^p \cup c_i^q$.

For each $i<n$, by Lemma~\ref{premutation preserveing implicitness}, $\Gamma(I_i^p)\in \mathbf{G}(\Gamma(x_i))=\mathbf{G}(y_i)$ hence there is $B\in \bigcap \mathbf{U}(y_i)$ and some $\iota < \kappa_{y_i}^{++}$ such that
$\Gamma(I^p_i)=t_{y_i,\iota}\restriction B$.
By $I_i^q\in\mathbf{G}(y_i)$ there is some $\iota' < \kappa_{y_{i}}^{++}$ such that $\Gamma(I^p_i)=t_{y_{i},\iota'}\restriction (\dom I_i^q)$.
By the properties of the guru $\vec{t}_{y_{i}}$ there is a club $C$ such that for $\iota^*_i= \sup \{\iota, \iota'\} +1$, we have for $z \in C$, $t_{y_i,\iota^*_i}(z) \leq \Gamma(I_i^p)(z),I_i^q(z)$. Let $I_i^* = t_{y_i,\iota^*_i}\restriction (B\cap C\cap \dom (I_i^q))$
This implies that
$$\langle c_{-1}^*,\langle y_0,,I_0^* \rangle,c_0^* ,\ldots, \langle \kappa^{+}, I^* \rangle \rangle \leq \Gamma(p),q,$$ as required.
\end{proof}

Similar proofs show that
\begin{proposition}\label{homyx}
Suppose $p,q \in \mathbb{Q}_{y,x}$ with $\ell(p)=\ell(q)=n$. Assume
\begin{itemize}
\item $\stem(p)=\{z_0, \ldots, z_{k-1},y, \kappa_{k+1}, \ldots, \kappa_{m-1},x,z_{m+1},\ldots, z_{n-1}\}$,
\item $\stem(q)=\{z'_0,\ldots, z'_{k-1},y,\lambda_{k+1},\ldots, \lambda_{m-1},x,z'_{m+1},\ldots, z'_{n-1}\}$,
\end{itemize}
are such that for all $i<k$ or $i>m,$ $\kappa_{z_i}=\kappa_{z'_i}$, for all $i \in (k,m)$, $\kappa_i=\lambda_i$, and for all $i<n$, $c_i^p \parallel c_i^q$. Then, there is some $\Gamma \in \Aut_{\kappa_x}(\kappa_x^+)$ such that $\Gamma(p) \parallel q$. \qed
\end{proposition}
\begin{proposition}\label{homx}
Suppose $p,q \in \mathbb{Q}^x$ with $\ell(p)=\ell(q)=n$. Assume
$$\stem(p)=\{ \kappa_0, \ldots, \kappa_{m-1},x,z_{m+1},\ldots, z_{n-1}\},$$
$$\stem(q)=\{\lambda_0,\ldots, \lambda_{m-1},x,z'_{m+1},\ldots, z'_{n-1}\},$$
are such that for all $i<m$, $\kappa_{z_i}=\kappa_{z'_i}$, for each $i \in [m+1,n-1]$, $\kappa_{i'}=\lambda_{i'}$, and for $i \in [0,n-1]$, $c_i^p \parallel c_i^q$. Then, there is some $\Gamma \in \Aut_{\kappa}(\kappa^+)$ such that $\Gamma(p) \parallel q$. \qed
\end{proposition}

\section{The final model}\label{finalmodel}
We are now turning to prove our main result. We continue with our setup from Section~\ref{weakhomogeneity}. Let $G$ be $\mathbb{R}$-generic and $H=\Pi[G]$. Then $H$ is $\mathbb{P}$-generic.
Let $K_0=K_0^G$ and $\mathcal{C}=\mathcal{C}_G$ be as in Section~\ref{cardinalstructureradin}.
From the projection $\Pi$ from Proposition~\ref{projectionrp},
it is easy to see that $K_0,\mathcal{C} \in V[H]$, and they can be derived from the working parts and the collapse parts of the conditions in $H$.

\begin{definition}
Let $W:=V[K_0,\mathcal{C}]$ be the smallest $\zfc$ extension of $V$ containing $K_0$ and $\mathcal{C}$, so $W \subseteq V[H] \subseteq V[G]$.
\end{definition}
An element $a \in W$ has an $\mathbb{R}$-name $\dot{a}$ which is invariant under automorphisms of $\mathbb{R}$ which fix elements in $K_0$ and $\mathcal{C}$. In particular, for $\Gamma \in \Aut_\kappa(\kappa^+)$, $\Gamma(\dot{a})=\dot{a}$. Our main goal in this section is to derive several indecomposable ultrafilters in $W$.

To analyze further, recall that we defined
\begin{itemize}
\item $\mathbf X:=\{x \mid \exists p \in G\,(x \in \stem(p))\}$;
\item $K_0:=\{\kappa_x \mid x \in \mathbf X\}$;
\item $K_1:=\{((\kappa_x)^+)^V \mid x \in \mathbf X\}$;
\item $K_2:=\{((\kappa_x)^{++})^V \mid x \in \mathbf X\}$;
\item $\theta:=\otp(\mathbf X,\ssim)$.
\end{itemize}

Note that $\mathbf X$ is a cofinal chain in $(\mathcal{P}_\kappa (\kappa^{+}),\ssim)$. Hence $\bigcup \mathbf X=\kappa^+$.
Furthermore, $\{\sup(x) \mid x\in \mathbf X\}$ is a club in $\kappa^+$.
As for the cardinal structure of $W$, by Theorem~\ref{cardinalsinpradin} all cardinals below $\kappa$ outside of $\{\omega,\omega_1\}\cup K_0\cup K_1 \cup K_2$ are collapsed.
In addition, by Theorem~\ref{cardinalsinpradin}, all singular cardinals below $\kappa$ are from $\acc (K_0)$.
Moreover, an analog of Proposition~\ref{prop46} holds in $W$, that is, $\gch$ holds below $\kappa$.

\subsection{Ultrafilters at successors of singulars}
Fix a singular cardinal $\lambda<\kappa$ in $W$. We recall a few more facts:
\begin{itemize}
\item $\lambda$ was a measurable cardinal in $V$, $\square^B_\lambda$ held, and $\cf^W(\lambda)=o^{\vec{U}}(\lambda)^V$.
\item $(\lambda^+)^V$ is preserved in $V[H]$, but collapsed in $V[G]$. In particular, $(\lambda^+)^V$ is preserved in $W$.
So, by Proposition~\ref{552}, $\square_{\lambda,\cf(\lambda)}$ holds.
\item Let $x^\lambda \in \mathbf X$ be the unique $x$ such that $\kappa_{x}=\lambda$ so that $\otp(x^\lambda)=(\lambda^+)^V$.
\item $\mathbf X ^\lambda:=\{x \in \mathbf X \mid x \ssim x^\lambda\}$ is a cofinal chain in $(\mathcal{P}_{\kappa(x^\lambda)}(x^\lambda),\ssim)$, so that $\bigcup \mathbf X^{\lambda} =x^\lambda$.
\item Let $\alpha_\lambda:= \otp(\mathbf X ^\lambda,\ssim)$, and recall that $\alpha_\lambda=\otp\{\kappa_y\mid y\in \mathbf X^\lambda\}$.
\item Let $\mathbf X_\lambda := \{\pi_{x^\lambda}^{-1}[x] \mid x \in \mathbf X^\lambda\}$ and notice that it is a chain in $(\mathcal{P}_\lambda(\lambda^+)^V,\ssim)$ with $\bigcup \mathbf X_\lambda =(\lambda^+)^W$ and $\otp (\mathbf X_\lambda,{\ssim}) =\alpha_\lambda$.

\item Let $\theta_\lambda:=\cf^W(\lambda)=\cf^W(\alpha_\lambda)$ and let $\langle\nu^\lambda_\tau\mid \tau<\theta_\lambda\rangle\in W$ be some cofinal sequence in $\alpha_\lambda$.
As $\cf(\otp (\mathbf X_\lambda,{\ssim})) =\cf(\lambda)$, we may let
$\langle x^\lambda_\tau \mid \tau<\theta_\lambda\rangle$ be the $\ssim$-increasing enumeration of a cofinal chain in $\mathbf X_\lambda$ induced from $\langle\nu^\lambda_\tau\mid \tau<\theta_\lambda\rangle$.
\item Let $G^{x^\lambda}$ be the $\mathbb{Q}_{0,x^\lambda}$-generic which is generated by $\Pi_0^{0,x}[G]$,
where the projection map is from Proposition~\ref{projxy}.
Note that $W \subseteq V[G^{x^\lambda}]$, and that $\mathbf X_\lambda$ and $\langle x^\lambda_\tau \mid \tau<\theta_\lambda\rangle$ are both in $V[G^{x^\lambda}]$.

\end{itemize}

\begin{definition}
In $V[G^{x^\lambda}]$, define a filter $F_\lambda$ over $\mathcal{P}_\lambda(\lambda^+)$ via:
$$ F_\lambda :=\{A\s \mathcal{P}_\lambda (\lambda^+)\mid \{\tau<\theta_\lambda \mid x^\lambda_\tau\in A\} \text{ is co-bounded in }\theta_\lambda\}.$$

\end{definition}

\begin{lemma}\label{intersectfilter}
$F_\lambda \cap W\in W$.
\end{lemma}

\begin{proof}
We aim to find $F^* \in W$ such that $F^*=F_\lambda \cap W$. Define $F^* \in W$ as the collection of all $A$ such that:
\begin{enumerate}
\item $A$ admits a $\mathbb{Q}:=\mathbb{Q}_{0,x^\lambda}$-name $\dot{A}$ that is forced to be in $\dot{W}$, and
\item there is $p\in \mathbb{Q}$ such that:
\begin{enumerate}
\item $\lambda =\kappa_n^p$ for some $n<\ell(p)$;
\item for every $i<n$, $x_i^p \in \mathbf X_\lambda$, and let $\tau_i<\theta_\lambda$ be such that $\kappa_{x_i^p}=\kappa_{\nu^\lambda_{\tau_i}}$;
\item $c_{-1} \in C_{-1}$, and for every $i \geq 0$, $c_i \in C_{\tau_i}$;\footnote{Recall the definition of $C_{\tau_i}$ in the beginning of Section~\ref{cardinalstructureradin}, where we note that $C_{\tau_i} \in \mathcal{C}$, so it is in $W$.}
\item for all $i<\ell(p)$ and $\tau \in (\tau_i,\tau_{i+1})$, there is $x \in \dom (I_{i+1}^p)$ with $\kappa_x=\kappa_{\nu^\lambda_\tau}$ and $I_{i+1}^p(x) \in C_{\tau}$;
\item for every $\tau>\tau_{\ell(p)-1}$, there is $x \in \dom (I_{\ell (p)}^p)$ with $\kappa_x=\kappa_{\nu^\lambda_\tau}$ and $I_{\ell (p)}^p(x) \in C_{\tau}$;
\item $p \Vdash_{\mathbb{Q}} ``\dot{A} \in \dot{F}_\lambda"$.
\end{enumerate}
\end{enumerate} We first check that $F^*$ is well-defined. Suppose $A \in W$ with a $\mathbb{Q}$-name $\dot{A}$ which is invariant under automorphisms fixing the elements of $K_0,\mathcal{C}$, and there are $p_0,p_1 \in \mathbb{Q}$ satisfying (2)(a)--(e)
as above, and $p_0 \Vdash ``\dot{A} \in \dot{F}_\lambda"$, but $p_1 \Vdash ``\dot{A} \notin \dot{F}_\lambda"$.
By extending $p_0$ and $p_1$ if necessary using those $x$'s in those requirements, we may assume that $\ell(p_0)=\ell(p_1)$.

By Lemma~\ref{homyx}, there is a $\Gamma \in \Aut_\lambda(\lambda^+)$ such that $\Gamma(p_0) \parallel p_1$. Since $\Gamma$ fixes $\dot{K}_0$ and $\dot{\mathcal{C}}$, we have that $\Gamma(\dot{A})=\dot{A}$.
Furthermore, since $\Gamma$ fixes the tail below $\lambda^+$, $\Gamma$ does not change the definition of $\dot{F}_\lambda$, and hence, $\Gamma(p_0) \Vdash \dot{A} \in \dot{F}_\lambda$.
Now, if $p_2 \leq \Gamma(p_0),p_1$, so $p_2$ forces the opposite statements, which is a contradiction.

We now show that $F^*=F_\lambda \cap W$. If $A \in F^*$, then there is $p \in \mathbb{Q}$ satisfying the requirements (a)--(e),
so $p \Vdash_{\mathbb{Q}} ``\dot{A} \in \dot{F}_\lambda"$. Find $p' \in G$ which decides $``\dot{A} \in \dot{F}_\lambda"$, but since $p'$ also satisfies (a)--(e), there is $\Gamma \in \Aut_\lambda(\lambda^+)$ with $\Gamma(p) \parallel p'$, hence, $p'$ must force that $``\dot{A} \in \dot{F}_\lambda"$.
Thus, $A \in F_\lambda \cap W$. The proof that $F_\lambda \cap W \subseteq F^*$ is simpler since we can pick $p$ directly from $G$ deciding $``\dot{A} \in \dot{F}_\lambda"$ and so $p$ satisfies (a)--(e).
\end{proof}

By the definition of $F_\lambda$, we can see that $F_\lambda \cap W$ generates a filter on $\mathcal{P}_\lambda(\lambda^+)$.
We now move on to define the following objects in $W$:
\begin{itemize}
\item $F_\lambda':=\{\{\sup(x) \mid x \in A\} \mid A \in F_\lambda \cap W\}$;
\item $\mathcal{U}_\lambda$ is an ultrafilter on $\mathcal{P}_\lambda(\lambda^+)$ that extends $F_\lambda \cap W$, and
\item $\mathcal{W}_\lambda$ is an ultrafilter on $\lambda^+$ extending $F_\lambda'$.
\end{itemize}
Note that $\mathcal{W}_\lambda$ is a uniform ultrafilter over $\lambda^+$.

\begin{lemma}\label{indecomposableukappa1}
In $W$, for every $\rho \in (\cf(\lambda),\lambda)$ which is regular in $V[G^{x^\lambda}]$, $\mathcal{W}_\lambda$ is $\rho$-indecomposable.
\end{lemma}
\begin{proof}
Let $\langle A_i \mid i< \rho \rangle$ be a partition of $\lambda^+$, where $\rho \in (\cf(\lambda),\lambda)$ is regular in $V[G^{x^\lambda}]$.
In $V[G^{x^\lambda}]$, for $j<\cf(\lambda)$, let $\eta(j)$ be the unique $\eta$ such that $\sup(x_j) \in A_\eta$.
Since $\cf(\lambda)<\rho$ and $\rho$ is regular, there is an $\alpha<\rho$ such that for all $j<\rho$, $\eta(j)<\alpha$. Hence, $\bigcup_{i<\alpha}A_i \in F_\lambda \cap W$, so $\bigcup_{i<\alpha}A_i \in \mathcal{W}_\lambda$.
\end{proof}

\begin{corollary}\label{C89}
In $W$, for every $\rho \in (\cf(\lambda),\lambda)$ that is regular in $W$, $\mathcal{W}_\lambda$ is $\rho$-indecomposable.
\end{corollary}
\begin{proof}
Let $\rho \in (\cf(\lambda),\lambda)$. Let $y,x\in \mathbf{X}$ be such that $y \ssim x$, $\rho<\kappa_y$, and $\kappa_x=\lambda$.
Let $\eta:= \kappa_y$.
Let $G^{y,x}$ be the filter generated by $\Pi_0^{y,x}[G]$.
Then, $G^{y,x}$ is generic for $\mathbb{Q}^{y,x}$. Note that $\rho$ is preserved in $V[G^{y,x}]$. Define $\mathbf{X}_{\lambda,\eta}:=\{z \in \mathbf{X}_\lambda \mid y \ssim z\}$. Note that $\mathbf{X}_{\lambda,\eta}=\{\pi_x^{-1}[z]\mid z\in X^{\lambda,\eta}\}$, so $\mathbf{X}_{\lambda,\eta}$ is in $V[G^{y,x}]$.
In $V[G^{y,x}]$, define a filter $F_{\lambda,\eta}$ over $\mathcal{P}_\lambda(\lambda^+)$ by letting $A \in F_{\lambda,\eta}$ iff
there is an $\alpha<\theta_\lambda$ such that $\{x^\lambda_\tau \in \mathbf{X}_{\lambda,\eta}\mid \alpha<\tau<\theta_\lambda\}\s A$.

\begin{claim}\label{prop96}
$V[G^{y,x}] \cap F'_{\lambda}=F_{\lambda,\eta}$.\qed
\end{claim}

As a consequence, we have $W\cap F_{\lambda} = W\cap F_{\lambda,\eta}$. Therefore $\mathcal{W}_{\lambda}$ expands $F_{\lambda,\eta}$.
One can show that $\mathcal{W}_{\lambda}$ is $\rho$-indecomposable in a similar manner to Lemma~\ref{indecomposableukappa1} (by working in $V[G^{y,x}]$).
\end{proof}

\begin{corollary} Let $\rho\in (\cf(\lambda), \lambda)$ be some cardinal in $W$.
Then $\mathcal{W}_\lambda$ is $\rho$-indecomposable in $W$.
\end{corollary}

\begin{proof} By Corollary~\ref{C89} we can assume that $\rho$ is singular in $W$ and let $\eta:= \cf^{W}(\rho)$.
Let $\langle A_i\mid i<\rho\rangle$ be a partition in $W$ of $\lambda^{+}$ into $\rho$ parts.

Let $x,y\in \mathbf{X}$ such that $\rho<\kappa_y<\kappa_x=\lambda$. Then in $W[G^{y,x}]$ let $f:\eta\rightarrow\rho$ denote the function $f(i):=j\iff \sup(x_i)\in A_j$. But $\mathbb{Q}^{y,x}/W$ is $\leq^*$-$\kappa_y^{+}$-closed and since $\rho<\kappa_y$ there are no new bounded subsets of $\rho$, hence $f\in W$.
Therefore in $W$, let $a=f``\eta$ and then $\bigcap_{i\in a}A_i\in \mathcal{W}_\lambda$.
\end{proof}
\begin{lemma}\label{smallquotient}
For every $\eta \in (\cf(\lambda),\lambda)$ that is regular in $V[G]$, $|\Ult(\eta,\mathcal{U}_\lambda)|=\eta$.
\end{lemma}

\begin{proof}
Work in $W$. Let $f: \mathcal{P}_\lambda(\lambda^+) \to \eta$. Define in $V[G]$ a function $g_f:\cf(\lambda) \to \eta$ by $g_f(\gamma)=f(x_\gamma)$. Note that for $f,f'$, if $g_f$ and $g_{f'}$ are equal on their tails, then $[f]_{\mathcal{U}_\lambda}=[{f'}]_{\mathcal{U}_\lambda}$. In $V[G]$, define $F$ as follows.
For each $g:\cf(\lambda) \to \eta$, if $g=g_f$ for some $f$, let $F(g)=[f]_{\mathcal{U}_\lambda}$, otherwise, $F(g)=0$. We claim that $\im(F) \supseteq \Ult(\eta,\mathcal{U}_\lambda)$. For each $[f]_{\mathcal{U}_\lambda}$, we show that $F(g_f)=[f]_{\mathcal{U}_\lambda}$. The point is that if $g_f=g_{f'}$, then $[f']_{\mathcal{U}_\lambda}=[f]_{\mathcal{U}_\lambda}$. Hence, in $V[G]$, $|\Ult(\eta,\mathcal{U}_\lambda)| \leq \eta^{\cf(\lambda)}=\eta$. Then this is also true in $W$.
\end{proof}

A similar proof yields the following:

\begin{lemma}
For every $\eta \in (\cf(\lambda),\lambda)$ that is regular in $W$, $|\Ult(\eta,\mathcal{U}_\lambda)|=\eta$. \qed
\end{lemma}

By tweaking the proof in Lemma~\ref{smallquotient}
by redefining $g_f$ to $g_f(\gamma)=f(\sup x_\gamma)$, we yield the following.
\begin{lemma}
If $\lambda \leq \kappa$ is singular and $\eta \in (\cf(\lambda),\lambda)$ is regular (in $W$), then $|\Ult(\eta,\mathcal{W}_\lambda)|=\eta$. \qed
\end{lemma}
\subsection{Summing up}
We now conclude the main theorem.
\begin{theorem}\label{maintheorem1}
In $W$, for every singular cardinal $\lambda<\kappa$, all of the following hold:
\begin{enumerate}[label=\textup{(\arabic*)}]
\item $\gch$;
\item $\mathcal{W}_\lambda$ is an ultrafilter on $\lambda^+$ that is $\rho$-indecomposable for any cardinal $\rho \in (\cf(\lambda),\lambda)$.
Furthermore, for any regular $\eta \in (\cf(\lambda),\lambda)$, $|\Ult(\eta,\mathcal{W}_\lambda)|=\eta$;
\item $\mathcal{U}_\lambda$ is an ultrafilter on $\mathcal{P}_\lambda(\lambda^+)$
such that for every regular cardinal $\eta \in (\cf(\lambda),\lambda)$,
$\mathcal U_\lambda$ is $\eta$-indecomposable and $|\Ult(\eta,\mathcal{U}_\lambda)|=\eta$.
\item $\square_{\lambda,\cf(\lambda)}$ holds.
\end{enumerate}
In particular, since we start with $o^{\vec{U}}(\kappa)=\kappa^{++}$ in $V$, we get that $W':=V_\kappa^{W}$
is a $\zfc$ model such that for every singular cardinal $\lambda$, (1)--(4) holds. \qed
\end{theorem}

Additional features of the universe $W'$ are given by
Lemma~\ref{l25}, Corollaries \ref{atthesingulars}, \ref{cor211} and \ref{c215},
and Facts \ref{csequencenumber} and \ref{prop226}.
Optimality aspects of our result are highlighted by Facts \ref{limitation1} and \ref{limitation2} and Proposition~\ref{529}.
We also point out that forcing extensions of $W'$ give rise to plenty of anti-Silver-type theorems \cite{silver}.
Indeed, given a singular cardinal $\lambda$,
it is possible to use a $\lambda^+$-distributive forcing to add any of the objects of Section~\ref{conventions} seen to refute $(\cf(\lambda),\lambda)$-indecomposable ultrafilters at $\lambda^+$.
This way, $\lambda$ would be the first singular cardinal to carry such an object.

\section{Concluding remarks}\label{concluding}
A variation of our final model in which $\gch$ fails and moreover $\sch$ fails everywhere may be obtained, as follows.
First, in Section~\ref{mastersequenceandcoherentsequence}, instead of starting with a model $V$ of $\gch$, we start with a model in which $\kappa$ is a $\kappa^{++}$-supercompact and $2^\kappa=2^{\kappa^+}=\kappa^{++}$ (this can be achieved as demonstrated by Silver, forcing with Easton-support iteration using $\Add(\alpha,\alpha^{++})$ at every inaccessible $\alpha\le\kappa$, and a trivial forcing at all other $\alpha\le\kappa$).
Consequently, in Lemma~\ref{coherentsequence}, for every $\alpha \in \dom(o^{\vec{U}})$, $2^\alpha=2^{\alpha^+}=\alpha^{++}$.
Consequently, Proposition~\ref{prop46} would assert that, in $V[G]$, for every $\tau\in\acc(\theta)$, $2^{\kappa_\tau}=(\kappa_\tau)^+$,
and for every $\tau\in\nacc(\theta)$, $2^{\kappa_\tau}=(\kappa_\tau)^{++}$.
Consequently, the last clause of Theorem~\ref{cardinalsinpradin}
would assert that, in $V[H]$, for every $\alpha\in \bar{K}_0$, $2^\alpha=\alpha^{++}$.
Consequently, the model $W$ of Section~\ref{finalmodel} satisfies that $2^\mu=\mu^{++}$ for every $\mu\in K_0$,
and $2^\mu=\mu^{+}$ for every infinite cardinal $\mu\in\kappa\setminus K_0$. In particular, in the final model $W'$, every singular cardinal $\lambda$
is a strong limit satisfying $2^\lambda=\lambda^{++}$.

\section*{Acknowledgments}
The first author was partially supported by the Israel Science Foundation (grant agreement 1216/18).
The second author is the recipient of the Maria Pogonowska-Proner award for 2024 and is partially supported by the Israel Science Foundation (grant agreement 1967/21).
The third author is partially supported by the Israel Science Foundation (grant agreement 203/22)
and by the European Research Council (grant agreement ERC-2018-StG 802756).

The main result of this paper was presented by the second author at the \emph{120 Years of Choice} conference in Leeds July 2024,
and by the third author at the \emph{Set Theory} session of the Annual conference of the IMU, July 2025.
We would like to thank the organizers for this opportunity,
and the participants (especially Gitik on the first occasion) for a stimulating feedback.

\end{document}